\documentclass[11pt,reqno]{amsart}

\usepackage{amsmath}
\usepackage{amsthm}
\usepackage{amssymb}
\usepackage{mathtools}
\usepackage{enumitem}
\usepackage{tikz}
\usetikzlibrary{decorations.markings}
\usetikzlibrary{arrows,shapes,positioning}
\usepackage{url}
\usepackage{color}
\usepackage{mathrsfs}
\usepackage{tikz-cd}
\usepackage{comment}

\usepackage{environ}

\usepackage{ifthen}

\usepackage{etoolbox}
\newboolean{showansweredQs}

\usepackage[colorinlistoftodos]{todonotes}
\usepackage{hyperref}
\usepackage{cleveref}
\newcommand{\on}[1]{\operatorname{#1}}

\newcommand{\ca}[1]{{\mathcal{#1}}}

\renewcommand{\o}[1]{{\overline{#1}}}

\newcommand{\cat}[1]{\left(\rm{#1}\right)}

\theoremstyle{definition}
\newtheorem{theorem}{Theorem}[section]
\newtheorem{lemma}[theorem]{Lemma}

\newtheorem{corollary}[theorem]{Corollary}
\newtheorem{proposition}[theorem]{Proposition}
\newtheorem{conjecture}[theorem]{Conjecture}
\newtheorem{definition}[theorem]{Definition}
\newtheorem{example}[theorem]{Example}
\newtheorem{remark}[theorem]{Remark}

\newtheorem{problem}[theorem]{Problem}

\DeclareMathOperator{\spec}{Spec} % spectrum of ring

\usepackage[normalem]{ulem}
\newcommand{\scr}[1]{\ensuremath{\mathscr{#1}}}
\newcommand{\msout}[1]{\text{\sout{\ensuremath{#1}}}}
\newcommand{\ZZ}{\mathbb{Z}}
\newcommand{\CC}{\mathbb{C}}
\newcommand{\NN}{{\mathbb{N}}}
\newcommand{\QQ}{{\mathbb{Q}}}
\newcommand{\OO}{\mathcal{O}}
\newcommand{\RR}{\mathbb{R}}
\newcommand{\Aff}{{\mathbb{A}}}
\newcommand{\PP}{\mathbb{P}}
\newcommand{\GG}{\mathbb{G}}
\newcommand{\Spec}{{\rm{Spec}\:}}
\newcommand{\Spf}{{\rm{Spf}\:}}
\newcommand{\Hom}{{\rm Hom}}

\newcommand{\gp}[1]{#1^{\rm gp}}

\newcommand{\Log}{{\sf{Log}}}
\newcommand{\lpb}{{\arrow[dr, phantom, very near start, "\ulcorner {\rm fs}"]}}
\newcommand{\lpbstrict}{{\arrow[dr, phantom, very near start, "\ulcorner \msout{{\rm fs}}"]}}

\renewcommand{\tilde}[1]{\widetilde{#1}}
\renewcommand{\hat}[1]{\widehat{#1}}

\newcommand{\Ms}[1][{g, n}]{\overline{M}_{#1}}
\newcommand{\Mp}[1][{g, n}]{\mathfrak M_{#1}}
\newcommand{\Msl}[1][{g, n}]{\overline{M}^{\rm log}_{#1}}
\newcommand{\Mpl}[1][{g, n}]{\mathfrak M^{\rm log}_{#1}}

\usepackage{extarrows}
\newcommand{\longequals}{\xlongequal{\: \:}}

\DeclareMathOperator*{\colim}{colim}

\usepackage{stmaryrd}
\newcommand{\adj}[1]{\llbracket #1 \rrbracket}

\newcommand{\bra}[1]{\left[{#1}\right]}
\newcommand{\pra}[1]{\left({#1}\right)}
\newcommand{\AF}[1][{}]{\Theta_{{#1}}}
\newcommand{\af}[1]{\AF[#1]}

\newcommand{\Cone}{{\rm Cone} \,}

\newcommand{\cal}[1]{\ensuremath{\mathcal{#1}}}

\newcommand{\pt}{{\rm pt}}
\newcommand{\bb}[1]{\ensuremath{\mathbb{#1}}}
\newcommand{\Cones}{\cat{Cones}}

\newcommand{\Mons}{\cat{Mon}}
\newcommand{\longsimeq}{\overset{\sim}{\longrightarrow}}
\def\overnorm#1{\overline{#1}\vphantom{#1}}
\renewcommand{\bar}[1]{\ensuremath{\overnorm{#1}}}

\usepackage{xparse}
\NewDocumentCommand{\bl}{ O{\lambda} m }{{{\left({#2}\right)}^{\sim}_{#1}}}
\NewDocumentCommand{\ev}{ O{} m }{\hat{\wedge}_{#1} #2}
\NewDocumentCommand{\evs}{ O{} m }{\wedge_{#1} #2}

\renewcommand{\log}{{\sf{log}}}

\title{Log geometry and lifting rational points}

\author[Leo Herr]{Leo Herr$^\heartsuit$}
\author[Sara Mehidi]{Sara Mehidi$^\clubsuit$}
\author[Marta Pieropan]{Marta Pieropan$^\clubsuit$}
\author[Thibault Poiret]{Thibault Poiret$^\diamondsuit$}
\date{\today \\ $^\heartsuit$ Virginia Tech, $^\clubsuit$ Utrecht University, $^\diamondsuit$ Jagiellonian University Krakow.}

\begin{document}


\maketitle

\begin{abstract}
    For a morphism $f : X \to Y$ of schemes, we give a tropical criterion for which points of $Y$ (valued in a field, discrete valuation ring, number ring, or Dedekind domain) lift to $X$. Our criterion extends the firmaments of Abramovich \cite{abramovich09ClayNotes} to a wide range of morphisms, even logarithmic stable maps. 
\end{abstract}

\setcounter{tocdepth}{1}
\tableofcontents

\section{Introduction}

\begin{problem}\label{prob:rationalpoints}
Given a morphism $f:X\to Y$ of varieties over a field $k$, determine the image $f(X(k))$ as a subset of $Y(k)$.
\end{problem}

Versions of this problem are studied in \cite{Serre90split, Skorobogatov96split, CTSSD98split, CTS00split, HS03split, Campana04, Campana05surfaces, abramovich09ClayNotes, BMS14split, BBL16split, LS16split, Sofos16split, Loughran18split,  Denef19split, LSS20pseudo-split, HWW22split, BLS23split, LM24split}. See the end of this introduction for a historical overview.

Abramovich \cite{abramovich09ClayNotes} suggests  that toroidal embeddings and log geometry can capture insights of Campana \cite{Campana04, Campana05surfaces} on Problem \ref{prob:rationalpoints}, which is the goal of this paper. 
We conclude with applications to rational points on varieties
and log stable maps.

\subsection{Firm morphisms}

We define a logarithmic and tropical condition for points in $Y(k)$ to be in the image $f(X(k))$ inspired by the ``firmaments'' of \cite[\S2]{abramovich09ClayNotes}. 

\begin{definition}\label{def:firm}
    Let $f \colon X \to Y$ and $p \colon S \to Y$ be maps of fine and saturated log schemes. For any strict geometric point $\bar s \to S$, we say that the morphism $p \colon S \to Y$ is \emph{$f$-firm at $\bar s$}, or \emph{firm along $f$ at $\bar s$} if there exist a strict geometric point $\bar w \to X \times_Y^\mathsf{fs} \bar s$ and a morphism of monoids $t$ which factors the identity\footnote{We are using fine and saturated fibre products here, but one would get the same notion of firmness with the non-fs fibre product.}
    \begin{equation}\label{eqn:factorizationfirmness}
       \bar M_{S, \bar s} \to \bar M_{X \times_Y^\mathsf{fs} S, \bar w} \overset{t}{\longrightarrow} \bar M_{S, \bar s}. 
    \end{equation}
    We say $p$ is \emph{$f$-firm}, or \emph{firm along $f$}, if it is $f$-firm at every strict geometric point.
\end{definition}

Let $p : s \to Y$ be a \emph{log point}, that is, a morphism of log schemes such that the underlying scheme of the source is the spectrum of a field $s^\circ = \spec k$. 
We denote underlying schemes by $X^\circ, Y^\circ$, etc. 

For a morphism $f : X \to Y$, the properties 
\begin{enumerate}
    \item\label{eit:1} $s \to Y$ lifts to a log point $s \to X$, called a \emph{log lift},
    \item\label{eit:2} $s \to Y$ is $f$-firm, 
    \item\label{eit:3} $s^\circ \to Y^\circ$ lands in the set-theoretic image of $f : X \to Y$,

\end{enumerate}
have evident implications \eqref{eit:1} $\Rightarrow$ \eqref{eit:2} $\Rightarrow$ \eqref{eit:3}. 

Our results give partial converses \eqref{eit:3} $\Rightarrow$ \eqref{eit:2} $\Rightarrow$ \eqref{eit:1} for nice $f$.

\begin{theorem}[{\S \ref{s:thibaultthm}}]\label{thm:thibault}
    Let $f \colon X \to Y$ be log flat and locally of finite presentation and $s \to Y$ a log point of $Y$. 
    There exists a strict morphism of log points $s'\to s$ corresponding to a finite field extension and a log lift $s' \dashrightarrow X$ if and only if $s \to Y$ is $f$-firm. If $f$ is log smooth, one can take ${s'}^\circ \to s^\circ$ to be separable.
\end{theorem}

\begin{theorem}[{\S \ref{s:firm}}]\label{thm:firmlocus}
    Let $f \colon X \to Y$ be integral and saturated. A map $p \colon S \to Y$ is $f$-firm if and only if its set-theoretic image $p(S) \subseteq Y$ is contained in that of $f$: 
    \[p(S) \subseteq f(X) \subseteq Y.\]
    In other words, the map $p$ is $f$-firm if and only if the pullback\footnote{both fine saturated and ordinary in this case} $X \times_Y^{\msout{\rm fs}} S \to S$ is surjective. 
\end{theorem}

This theorem applies if $f$ is strict, i.e.,\ $f$ is just a morphism of schemes. 

Lifting a log point $s \to Y$ to $X$ is usually much stronger than lifting the point on underlying schemes. 

\begin{example}
    Let $R$ be a complete discrete valuation ring and equip it with its natural log structure $M_R = R \setminus \{0\}$ so that $\bar M_R = \NN$. Write $S = \spec R$ for the resulting log scheme and $s \in S$ for its closed point.
    Let $S \to Y$ and $f : X \to Y$ be morphisms of log schemes. If $f$ is log smooth, the Log Hensel Lemma (cf.~Lemma \ref{LogHensel}) shows that finding a log lift $S \to X$ is equivalent to finding a log lift of the closed point $s \to X$.
\end{example}

The log structure also significantly constrains possible lifts. 

\begin{example}\label{eg:trivial}
    Suppose $s = s^\circ$ has trivial log structure and $f : X \to Y$ is the map $X \to X^\circ$ forgetting log structure. The forgetful map $f$ is log flat if and only if $X = X^\circ$ has no log structure, so Theorem \ref{thm:thibault} likely does not apply. Indeed, scheme-theoretic points $s \to X^\circ$ with trivial log structure $s = s^\circ$ are all firm. Yet they will not lift to log points $s \dashrightarrow X$ unless they lie in the open, possibly empty locus $X_0 \subseteq X$ where the log structure is trivial. 
\end{example}

For nice enough maps $f$, we show a log point $s \to Y$ lifts to $X$ if and only if the point of the underlying scheme $s^\circ \to Y^\circ$ does. 

\begin{corollary} \label{cor:log lifts}
    For $f \colon X \to Y$ integral, saturated, log smooth, and of finite presentation and $s \to Y$ a log geometric point, the log point lifts as $s \dashrightarrow X$ if and only if the point of the underlying schemes does as $s^\circ \dashrightarrow X^\circ$. 
\end{corollary}

Using our results, we can characterize which maps are firm. 

\begin{definition}
    View $Y$ as a presheaf on fine saturated log schemes. Let $Y(f) \subseteq Y$ be the subpresheaf consisting of morphisms $p \colon S \to Y$ from fine saturated log schemes $S$ which are $f$-firm, dubbed the \emph{firm locus}. 
\end{definition}

The map $f : X \to Y$ factors through $Y(f) \subseteq Y$. Firmness detects whether points log lift along a map $f$ but has nothing to do with how many lifts there are. 

\begin{example}\label{ex:infmanyliftsstdpoint}
    Let $S = X$ be a log geometric point with rank-one log structure $\bar M_S = \bar M_X = \NN$ and take $Y = X^\circ$ with the natural map $f : X \to Y$. Any map to $Y$ is $f$-firm, so the firm locus is $Y(f) = Y$. Indeed, $S \to Y$ lifts to $S \dashrightarrow X$. In fact, there are infinitely many such lifts given by all maps $\NN \to \NN$.
\end{example}

Depending on $f$, we show the firm locus $Y(f) \subseteq Y$ is ``representable'' to various degrees. 

\begin{corollary}
    Let $f \colon X \to Y$ be a morphism. 
    \begin{itemize}
        \item If $f$ is integral and saturated, the firm locus $Y(f) \subseteq Y$ is the set-theoretic image of $f$. In particular, 
        \begin{itemize}
            \item If $f$ is flat and locally of finite presentation, then $Y(f) \subseteq Y$ is an open subscheme. 
            \item If $f$ is proper, then $Y(f) \subseteq Y$ is an ind-closed subscheme. 
        \end{itemize}
        \item If $Y$ is quasicompact and $f$ is both of finite presentation and either log flat or proper, there is a log alteration $W \to Y(f)$ of $Y(f)$ with $W$ a scheme or formal scheme, respectively.   
    \end{itemize}
\end{corollary}

In the latter case, the firm locus $Y(f)$ is a ``weak (ind-)logarithmic space'' -- a functor on fine saturated log schemes which admits a log alteration by an (ind-)log scheme.

\subsection{Application to rational points}
In Section~\ref{s:appnsrationalpoints}, we relate our work to that of Abramovich~\cite{abramovich09ClayNotes}. We compare our notion of \emph{firmness} (cf.~Definition~\ref{def:firm}) with the one introduced in~\cite[Definition~2.4.17]{abramovich09ClayNotes}, which we generalize to the logarithmic setting in Definition~\ref{def:firmamentcontactorder}. To avoid confusion, we refer to the latter as \emph{lying in the firmament}.

Under mild assumptions, we prove the two notions coincide. 
\begin{theorem}[{Proposition \ref{comparefirmlieinfirmament}, Theorem \ref{thm:firm=firmament}}]
    Let $f : X \to Y$ be a 
    surjective, log flat, log reduced, finite
presentation morphism of log schemes and consider a map $p : S \to Y$. If $S$ is 
    \begin{itemize}
        \item a log point with rank one log structure, or 
        \item the spectrum $S = \spec R$ of a discrete valuation ring with log structure $R\setminus \{0\}$, 
    \end{itemize}
    then $p$ is $f$-firm if and only if $p$ lies in the firmament of $f$. 
\end{theorem}

This identification allows us, via Theorem~\ref{thm:thibault}, to establish the following claim—stated without proof in \cite[paragraph following Theorem~2.4.18]{abramovich09ClayNotes}:

\begin{theorem}[{Theorem \ref{LiftAbr}}] \label{thm:intro} Let $S=\Spec A$ where $A$ is a Dedekind domain with fraction field $K$ of characteristic zero. Let $f: X \to Y$ be a proper, dominant map of integral proper $K$-varieties.

After replacing $S$ by a sufficiently small nonempty open subset, there exists a proper log smooth $K$-birational model $f':X' \to Y'$ of $f$ and 
a proper log smooth $S$-model $g:\mathscr X'\to\mathscr Y'$ of $f'$ such that every $K$-point on the locus $X'_0\subseteq X'$ where the log structure is trivial induces an $S$-point on $\mathscr{Y}'$ that lies in the firmament of $g$ (Definition \ref{def:firmamentcontactorder}). 

Conversely, every $S$-point $p$ on $\mathscr{Y}'$ that 
intersects the locus where the log structure of $\mathscr Y'$ is trivial and
lies in the firmament of $g$ lifts étale locally on $S$ to a rational point on $X'$. 
 \begin{figure}[h]  \[
\begin{tikzcd}
&X \arrow[r,"f"] & Y \\
&X' \arrow[r,"f'"] \arrow[u]\arrow[d] & Y' \arrow[d] \arrow[u]\\
&\mathscr X' \arrow[r,"g"] & \mathscr Y'\\
\Spec K \times_S S'\arrow[r]\arrow[uur,dashed]& S' \arrow[u,dashed] \arrow[r,"\text{\'et}"]& S \arrow[u,"p"']
\end{tikzcd}
\]
\caption{Illustration of the setup in Theorem \ref{thm:intro}. For the direct implication take $S'=S$.}
\end{figure}
\end{theorem}

The argument relies on the fact that dominant morphisms of varieties over a field of characteristic zero admit toroidal models as well as a known logarithmic Hensel's Lemma \ref{LogHensel}.
See Section \ref{ExamplesAbr} for myriad examples.

Theorem \ref{thm:intro} describes explicitly the set of points that lift \'etale locally on a suitable (i.e., log smooth) birational model of the initial morphism $f$. See Corollary \ref{cor:abramovich} for a statement without log structures.
A local lifting criterion in the framework of Kato fans is given in \cite[Proposition 5.4]{BLS23split}.
% \Marta{New sentence here above}

Appendix \ref{appendixLog} is a rapid introduction to necessary themes from log geometry and a comparison between 
\begin{itemize}
    \item Log schemes and their Artin fans, and
    \item Toroidal embeddings and their cone complexes. 
\end{itemize}
In Appendix \ref{sec:appendix}, we explore Campana-type necessary conditions on images of points under morphisms of varieties with no regularity assumptions.

\subsection{Context and state of the art}
%for Problem \ref{prob:rationalpoints}}
% \Marta{rename the section as ``Context and state of the art''?}
We recall the history of Problem \ref{prob:rationalpoints} in the case where $f$ is surjective.
If the field $k$ is separably closed, the induced map $f:X(k)\to Y(k)$ is surjective.
If the field is not separably closed, in general $f$ is not surjective on $k$-points, and
determining the image of the set of rational points under $f$ is a very difficult question in general. 

Rational points on families of varieties have been studied extensively in the case where every fiber has at least one reduced irreducible component. Under additional assumptions on the number of fibers that are ``split'', or ``simple'' in the sense of \cite{Skorobogatov96split}, or ``pseudosplit'' in the sense of \cite{LSS20pseudo-split}, there are several results on the Brauer-Manin obstruction to the Hasse principle and weak approximation \cite{Skorobogatov96split, CTSSD98split, CTS00split, HS03split, BMS14split, HWW22split}, there are
quantitative studies on the density of $f(X(k))$ \cite{Serre90split, Sofos16split, Loughran18split, LM24split}, of fibers  that satisfy weak approximation \cite{BBL16split} or that are everywhere locally soluble \cite{LS16split, LM24split}, and there are answers to Problem \ref{prob:rationalpoints} over local fields \cite{Denef19split, LSS20pseudo-split}.

The case where there are fibers with no reduced irreducible component has been investigated in \cite{abramovich09ClayNotes, LM24split, BLS23split} based on a geometric framework introduced by Campana \cite{Campana04, Campana05surfaces}. In \cite{abramovich09ClayNotes}, Problem \ref{prob:rationalpoints} was implicitly raised in the context where $f$ is a toroidal morphism of toroidal embeddings. Since such embeddings can naturally be equipped with a divisorial log structure, the same paper suggests that log geometry provides the appropriate framework to address the problem, with a promise to develop this perspective in future work. In this paper, we take up this direction and investigate the case of an arbitrary morphism of log schemes.

\subsection{Further applications}
%\textcolor{blue}{Sara: Since we already have an application in the paper and as a title in the intro, I think "further/more applications" suits better as a title here.}
%\Leo{Good idea. Maybe even "future applications" so it seems like we'll write these papers}\Sara{I agree}
%\Marta{"Future" or "Further" both sound good to me, with a slight preference for "Further". But I am against making any explicit statement about us working on something specific in the future.}
%\Marta{New text below}

Our firmness criterion is effectively computable in toroidal or logarithmic charts. In particular, for morphisms admitting toroidal structures, the firm locus can be described explicitly in terms of piecewise linear conditions on monoids. This gives an algorithmic way to determine whether a given local or global point lifts \'etale locally via a theory of tropical obstructions. 

% For a large class of morphisms (log smooth, integral, saturated), lifting logarithmic geometric points is necessary and sufficient 
% %\medit{for \'etale local lifts}
% to lift 
% the underlying points by Corollary \ref{cor:log lifts}. 
% \Marta{What is the mathematical meaning of the sentence here above? Is it that a schematic point lifts if and only if it lifts as a log point? I'm confused by the reference to Theorem 1.4 here. It characterizes firm points, not log lifts. Corollary \ref{cor:log lifts} seems a closer statement, but it applies only to lifts of geometric points}
% \Sara{I think we can delete this sentence, as its idea has already been expressed several times in the paper. I agree that citing Corollary 1.7 directly would make more sense if we keep the sentence.}
%The orbifold base of a morphism, discussed in Appendix \ref{sec:appendix} is used by \cite{} to predict the distribution of rational points in families. 

To study the distribution of rational points in families of varieties $X/B$, one needs to understand which rational points on the base $B$ lift to $X$. As we prove in Theorem \ref{thm:intro}, firmanents can encode \emph{all} geometric obstructions to such lifts. %So far, 
% Together with 
They generalize
coarser invariants, such as the orbifold base of the family and the corresponding set of Campana points described in Appendix \ref{sec:appendix}, 
% they
which
have been used to predict upper bounds for the number of everywhere locally soluble fibers for one-parameter families \cite{BLS23split}. 
Our notion of firm points applies equally well to higher dimensional bases 
%and they are expected to control the leading constants in asymptotic formulas for counting rational points in families, refining predictions based on Campana points and orbifold structures.
 and is expected to appear in the leading constant of asymptotic formulas.
% The notion of firm points offers an alternative perspective on firmaments that could facilitate an understanding of families with higher dimensional bases and of the leading constant in asymptotic formulas. 

%Firmaments are expected to lead to an understanding of families with higher dimensional base and  to play a role in the leading constant of asymptotic formulas. 
%\Marta{check this in references}

Points lying in the firmament are, similarly to Campana points, rational points satisfying intersection conditions with the boundary divisor determined by the firmament. Classical conjectures concerning rational points admit natural analogues for Campana points, which have been extensively studied \cite{Campana05surfaces,SJPA,PSTVA21}. One can expect the same for our firm points. 
% points lying in the firmament. 
For example, Abramovich formulated a strong density conjecture %which we can similarly expect for firm points:(Sara: I think this is repetitive)
% In particular, a density conjecture for {firm} points was formulated 
% by Abramovich:
%in \cite[Conjecture 2.4.19]{abramovich09ClayNotes}.

\begin{conjecture}[{\cite[Conjecture 2.4.19]{abramovich09ClayNotes}}]
\label{conj:density}
Let $(Y,\Delta)$ be a smooth projective orbifold with a firmament $\Gamma$. Then the set of points on $Y$ lying in the firmament $\Gamma$ is potentially dense if and only if $(Y,\Delta)$ is special.
\end{conjecture}

Here $(Y,\Delta)$ is a pair with boundary divisor $\Delta$ with standard coefficients determined by $\Gamma$ as in \cite[Definition 2.4.4]{abramovich09ClayNotes}. 
If $\Gamma$ is induced by a morphism $f:X\to Y$, then the coefficients of $\Delta$ depend on $f$ and are defined by the orbifold base invariants described in Appendix \ref{sec:appendix}.
%then $\Delta=\sum_{D}(1-\frac 1{m_{f,D}})D$ where $m_{f,D}$ is the invariant defined in Appendix \ref{sec:appendix}.
The theory of special orbifolds is developed in \cite{Campana04, Campana11}.
%We refer to \cite[]{} for the theory of special orbifolds.
%Special varieties and orbifold pairs were introduced by Campana in \cite{} for his birational classification of orbifold pairs. 
Typical examples are log Fano orbifolds, log Calabi-Yau orbifolds and rationally connected orbifolds \cite[\S\S 8, 10]{Campana11}.
%\Marta{For myself: add references to Campana's work.}

As stated in \cite{abramovich09ClayNotes}, Conjecture \ref{conj:density} is particularly interesting, as it would imply the well-known Campana conjecture, which predicts that for a variety $Y$ defined over a number field $k$, the set of rational points of $Y$ is potentially dense if and only if $Y$ is special. 

Our firm morphisms are a functorial version of 
Abramovich's
%these 
firmaments that make sense for an arbitrary morphism of log schemes and generalize the classical case $(Y, \Delta)$. They can be used to study Conjecture \ref{conj:density} and related problems. In fact, the ``firm locus'' we construct is a geometric moduli space object parameterizing firm maps.

% Campana orbifolds are also studied in birational geometry \cite{}. They lack functoriality in general, as the notion of a morphism relies on being able to pull back divisors, but pullbacks of divisors with standard coefficients don't have standard coefficients in general. This problem is automatically solved as part of the data of a log morphism, and one can give a fully functorial notion of Campana orbifold using the techniques of this paper. The power of such a notion is exemplified in the present paper by the universality of our firm locus. 

Campana orbifolds also play a central role in birational geometry \cite{Campana04,
Campana11,
CampanaPaun15,2024arXiv240710668K}. However, they are not functorial in a naive sense: morphisms of orbifolds require pulling back boundary divisors, and such pullbacks do not make sense for morphisms that land in the boundary of the target. This obstruction disappears in logarithmic geometry, where the pullback of the boundary is built into the structure of a log morphism. Using the techniques developed in this paper, one can formulate a fully functorial notion that generalizes Campana orbifolds in the logarithmic category. The strength of this perspective is illustrated by the universality of our firm locus, which is therefore the canonical receptacle for lifting problems.

We found an unexpected connection also with the theory of logarithmic stable maps. The ``type'' of a log stable map $C \to X$ defined by Gross-Siebert \cite[Definition 1.10]{gross2013logarithmic} is similar in spirit to the firm locus depicted in this paper. See Example \ref{gross2013logarithmic} for a detailed comparison. As a result, we can state a new obstruction to lifting log stable maps along a morphism $f : X \to Y$ in terms of the firm locus of the map $f$.

% Our formalization of points lying in the firmament, expressed through the language of logarithmic geometry and leading to the notion of firm morphisms, provides a new perspective that could be used to study related conjectures.

\subsection*{Conventions}

We assume all our schemes, algebraic spaces, and algebraic stacks are locally of finite type. We use fine saturated log structures exclusively except where noted otherwise, and we use the shorthand f.s.\ for fine and saturated. The algebraic stacks $\Log$, $\Log_S$ are M.\ Olsson's stacks of f.s. log structures, denoted ${\mathscr T}or$, ${\mathscr T}or_S$ in \cite{logstacks} 
and recalled in Definition \ref{def:logstacks}.
The f.s.\ fiber product is denoted $\times^{\mathsf{fs}}$ or $\ulcorner {\rm fs}$ if it appears in a pullback square in a diagram. We write $\times^{\msout{\rm fs}}$ or $\ulcorner \msout{\rm fs}$ when it also happens to be a fiber product in the category of underlying stacks. 
By ``localize,'' we mean to work on a strict-\'etale cover.

We spell out our conventions concerning Artin fans, log alterations, log subalterations, log modifications, log submodifications, log blowups, root stacks, etc.\ in Appendix \ref{appendixLog}. We write 
\[
\Aff_P \coloneqq \spec \ZZ[P]
\]
\[
\af{P} \coloneqq \bra{\Aff_P/\Aff_{\gp P}}, \qquad \af{} \coloneqq \af{\NN} 
\]
for the affine toric variety corresponding to a monoid $P$ and its Artin cone, the stack quotient by its dense torus. We also write $\af{X}$ for the Artin fan of a log algebraic stack $X$ (Definition \ref{Artinfan}).

\subsection*{Acknowledgments}

The project began in the \href{https://webspace.science.uu.nl/~piero001/index_2023seminarCampanaPoints.html}{Seminar on Campana Points} at Utrecht and D.\ Abramovich's 2024 LMS Invited Lecture Series \href{https://www.lms.ac.uk/events/lms-invited-lecture-series-2024}{Logs and stacks in birational geometry and moduli}. The authors were trying to translate the results of \cite{abramovich09ClayNotes} in terms of log geometry. We are grateful to D.\ Abramovich for his blessing to do so and helpful correspondence.
L.\ Herr is grateful to Sebastian Casalaina-Martin for asking about the case of stable maps after a talk he gave at CU Boulder. 
S. Mehidi thanks A. M. Botero for valuable discussions on toroidal geometry.
M.\ Pieropan thanks F.\ Bartsch for stimulating discussions about Appendix \ref{sec:appendix} and the organizers  of the workshop ``Diophantine and Rationality Problems'' in Sofia for the opportunity to give a talk and the participants for asking stimulating questions afterwards. The authors thank the anonymous referee for their comments, which improved the exposition of this paper.
% \Marta{Thanks to the anonymous referee here above}

L.\ Herr received partial support from the NWO grant VI.Vidi.193.006.
S. Mehidi is supported by the NWO grant VI.Vidi.213.019. 
M.\ Pieropan is partially supported by the NWO grants  OCENW.XL21.XL21.011 and VI.Vidi.213.019.
T. Poiret is supported by EPSRC New Investigator award EP/X002004/1.
For the purpose of open access, a CC BY public copyright license is applied to any Author Accepted Manuscript version arising from this submission.

\section{Firmness and log lifts}\label{s:thibaultthm}

The goal of this section is to prove Theorem \ref{thm:thibault}. We start with some basic properties of firmness.

\begin{lemma}\label{lem:firmbasics}
    Consider maps $f : X \to Y$ and $p : S \to Y$ as in Definition \ref{def:firm}. The map $p : S \to Y$ is firm along $f$ if and only if, for some (equivalently any) strict surjection locally of finite type $S' \to S$ there is a commutative diagram of log schemes
    \begin{equation}\label{eqn:localcoverfirmnesscheckdiagram}
        \begin{tikzcd}
        W \ar[rr] \ar[d]       &       &X \ar[d]      \\
        S' \ar[r]      &S \ar[r]      &Y,
        \end{tikzcd}   
    \end{equation}
    with the identity $S' \longequals S'$ firm along $W \to S'$. In particular, 
    \begin{enumerate}[label=(\roman*), ref=(\roman*)]
        \item Firmness can be checked strict-\'etale locally on $S$ and $Y$ and after replacing $X$ by a strict-\'etale cover. \label{firmbasics i}
        \item The map $p$ is $f$-firm if and only if the map from the strict reduced subscheme $S_{\rm red} \subseteq S \to Y$ is $f$-firm. \label{firmbasics3}
        \item The map 
        $p$ is $f$-firm if and only if the identity $S \longequals S$ is firm along the pullback $X \times_Y^{{\rm fs}} S \to S$. \label{firmbasics5}
        %Changed \msout{fs} fiber product to fs fiber product. It's true that you could take either fiber product but not that they necessarily agree. 
        \item 
        The map $p$ is $f$-firm if and only if all strict geometric points $\bar s \to S$ are $f$-firm. \label{firmbasics iv}
        %This is basically the definition of firmness, but I want to keep it because it tells the reader that firmness can be checked at geometric points.
    \end{enumerate}
    In addition, 
    \begin{enumerate}[label=(\roman*), ref=(\roman*), resume]
        \item If $p$ is $f$-firm and $X' \coloneqq X \times^{\rm fs}_Y S$ is the f.s.\ pullback, the map $X' \to S$ is surjective. \label{firmbasics2}
        \item The map $p$ is $f$-firm if and only if for all strict geometric points $y \to Y$ and $s \to S\times_Y y$, there exists a strict geometric point $x \to X \times_Y y$ such that $\o M_{Y,y} \to \o M_{S,s}$ factors through $\o M_{Y,y} \to \o M_{X,x}$. \label{firmbasics7}
        \item Given $t : Y' \to Y$, let $f' \colon X' \to Y', p' \colon S' \to Y'$ be the f.s.\ base changes of $f, p$ along $t$. If $p$ was $f$-firm, then $p'$ is $f'$-firm. \label{firmbasics4}
        \item If $p$ is $f$-firm, then any composite $S' \to S \to Y$ is also $f$-firm. I.e., the set of $S \to Y$ which are $f$-firm form a sieve over $Y$. \label{firmbasics6}
        \item 
        If $S$ is a log point, $p : S \to Y$ is $f$-firm if and only if there exists a strict geometric point $\bar s \to S$ which is $f$-firm. \label{it:firmpointsatgeompoints}
    \end{enumerate}
\end{lemma}

\begin{proof}
    We prove $p$ is firm along $f$ if and only if such a diagram \eqref{eqn:localcoverfirmnesscheckdiagram} exists with $S' \longequals S'$ firm along $W \to S'$. 
    % We prove the first statement. 
    Parts \ref{firmbasics i}-\ref{firmbasics iv} then follow directly. By \cite[\href{https://stacks.math.columbia.edu/tag/0487}{Tag 0487}]{stacks-project}, all geometric points $\bar s \to S$ lift to $S'$.
    Consider the map between the fibers 
    \[W_{\bar s} \to X_{\bar s}.\]
    As the identity $S' \longequals S'$ is firm along $W \to S'$, there is a strict geometric point $\bar w \to W_{\bar s}$ and a dashed retraction fitting in the sequence
\[
    \bar M_{S, \bar s} \to \bar M_{X\times_Y^\mathsf{fs} S, \bar w} \to \bar M_{W, \bar w} \dashrightarrow \bar M_{S, \bar s}.
\]
Such a retraction witnesses firmness of $\bar s \to S' \to S \to Y$ along $X \to Y$. The converse implication holds by properties of fiber product.

Part \ref{firmbasics7} follows from the universal property of pushout and the fact that $\bar M_{X\times_Y^{\mathsf{fs}}S,x}$ is the pushout of $\bar M_{Y,y}\to\bar M_{S,s}$ and $\bar M_{Y,y}\to \bar M_{X,x}$.
For \ref{firmbasics4}, let $\overline s'\to S'$ be a geometric point and $\overline s\to S$ be the induced point. Then firmness of $p$ along $f$ gives
\[
\begin{tikzcd}
    \bar M_{S, \bar s} \ar[r] \ar[d]      &\bar M_{X_S, \bar w} \ar[r] \ar[d]      &\bar M_{S, \bar s} \ar[d]         \\
    \bar M_{S', \bar s'} \ar[r]       &\bar M_{X'_{S'}, \bar w'} \ar[r, dashed]      &\bar M_{S', \bar s'}.
\end{tikzcd}
\]
The dashed arrow exists because the solid square is a pushout coming from the fiber product of fs log schemes. The proof of \ref{firmbasics6} is similar. For \ref{it:firmpointsatgeompoints} it suffices to argue that if a composite $\bar s \to \bar t \to S$ is firm, so is $\bar t \to S$. For this, one can use the same geometric point $\bar w \to X_{\bar s} \to X_{\bar t}$ for the requisite splitting in each case.
\end{proof}

The first part of Lemma \ref{lem:firmbasics} will routinely be used to assume $Y$ is atomic (Definition \ref{def:atomic}) and $X$ is a disjoint union of atomics, as \cite[Proposition 2.2.2.5]{logpic} shows any (locally finite type) log scheme admits such a strict-\'etale cover.

We express the notion of firmness in terms of Artin fans, reviewed in Appendix \ref{appendixLog}. Artin cones and Artin fans provide a geometric incarnation of log structures that has become ubiquitous since its introduction in 
\cite{skeletonsandfans, wisebounded, birationalinvarianceabramovichwise}.

Let $f : X \to Y$ and $p : S \to Y$ be morphisms inducing compatible solid arrows of Artin fans 
\[
\begin{tikzcd}
            &\af{X} \ar[d]         \\
    \af{S} \ar[r] \ar[ur, dashed]      &\af{Y}.
\end{tikzcd}
\]
To be firm means \emph{approximately} that there is a dashed lift $\af{S} \dashrightarrow \af{X}$ making the diagram commute. This is not literally true because the Artin fan of a pullback is not the pullback of the corresponding Artin fans, as we show in Example \ref{ex:AFbasechangecounterex}. But a version of this statement holds for geometric points. See the end of Section \ref{s:firm} for a general statement with the relative Artin fan.

\begin{proposition}\label{prop:logpointfirm=AFsection}
    If $S = Y$ is a log geometric point, the identity $S = Y$ is firm along a map $f : X \to Y$
    if and only if the map $\af{X} \to \af{Y}$ on Artin fans admits a section.
\end{proposition}

There exists such a map $\af{X} \to \af{Y}$ because $\af{Y} = \af{\Gamma(Y, \bar M_Y)}$ is an Artin cone by Lemma \ref{lem:AFmapstocone}. 
As we show in Example \ref{ex:AFbasechangecounterex},
the proposition is false if $Y$ is not a log geometric point.

The Artin fan of $X$ is the empty set $\af{X} = \varnothing$ if and only if $X = \varnothing$. The only morphism which is firm along the unique map $f : \varnothing \to Y$ is $f$ itself.

\begin{example}\label{ex:AFbasechangecounterex}
    We give an example of a log scheme map $f \colon X \to Y$ and a log geometric point $s \to Y$ which is not $f$-firm, but such that $\af s \to \af Y$ factors through $\af X$. 
    See Figure \ref{fig:AFbasechange} for a picture.
    The map $\af s \to \af Y$ will be the identity, thereby also showing that the base change of Artin fans
    \[
    \af X \times_{\af Y} \af s = \af X
    \]
    is not the Artin fan of the base change $X \times_Y s$.  

    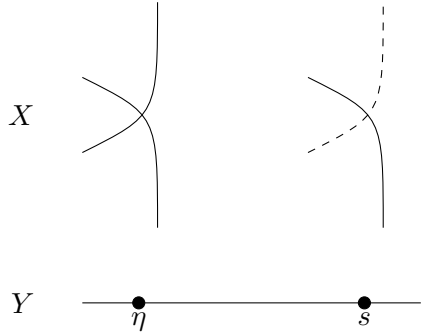
\begin{figure}[h]
    \centering
    \begin{tikzpicture}
        \draw [-] (0, 1) ..controls (1, 1.5).. (1, 3);
        \draw [-] (0, 2) ..controls (1, 1.5).. (1, 0);

        \draw[-] (0, -1) to (4.5, -1);

        \filldraw (.75, -1) circle (.08);
        \filldraw (3.75, -1) circle (.08);

        \node[below] at (3.75, -1){$s$};
        \node[below] at (.75, -1){$\eta$};

        \node [left] at (-.5, 1.5){$X$};
        \node [left] at (-.5, -1){$Y$};
        \begin{scope}[shift = {(3, 0)}]
        \draw [-, dashed] (0, 1) ..controls (1, 1.5).. (1, 3);
        \draw [-] (0, 2) ..controls (1, 1.5).. (1, 0);
        \end{scope}
    \end{tikzpicture}
    \caption{The family of open curves $X$ from Example \ref{ex:AFbasechangecounterex}. The generic fibre $X_\eta \to \eta$ admits a section which induces a section of the map $\af{X} \to \af{Y}$ of Artin fans, but there is no section over the closed point. }
    \label{fig:AFbasechange}
\end{figure}

    Let $Y = \Spec R$ where $R$ is a discrete valuation ring, and give $Y$ the constant log structure associated to the map 
    \begin{align*}
        \bar M_Y \coloneqq\NN t & \to \OO_Y \\
        t & \mapsto 0
    \end{align*}
    Let $\eta, s$ be respectively the generic point and special point of $Y$, both endowed with their $Y$-strict log structure. Let $\o X \to Y$ be the constant curve
    \[
    \o X := \Spec(R[x,y]/(xy^3))
    \]
    with log structure given by
    \[
    \o M_{\o X}(\o X) := (\o M_Y \oplus \bb N\log(x) \oplus \bb N\log (y))/(\log (x)+3\log (y)=t)
    \]
    mapping to $\ca O_{\o X}$ via $t \mapsto 0, \log(x) \mapsto x,\log(y) \mapsto y$. Let $D(x),D(y)$ be the nonvanishing loci in $\o X$ of $x$ and $y$ respectively. Let $X \subset \o X$ be the union of $D(x)$ and of the generic fibre $\o X_\eta$. We equip the open subschemes $D(x),D(y),X$ of $\o X$ with their $\o X$-strict log structures. The characteristic sheaf $\o M_{\o X}$ is constant on $D(y)$ with value $\bb N\log(x) = \o M_Y$, and constant on $D(x)$ with value $\bb N\log(y)$. The map 
    \[
    \bb N = \o M_Y \to \bb N\log y = \bb N
    \]
    is multiplication by $3$, which does not have a cosection, so the point $s \to Y$ is not firm with respect to $f \colon X \to Y$. However, the composite
    \[
    \af s = \af Y \xrightarrow{\sim} \af{D(y)} = \af{D(y) \times_{\spec R} \eta} \to \af X
    \]
    is a section of the natural map $\af X \to \af Y$.
\end{example}

We need a general lemma to prove Proposition \ref{prop:logpointfirm=AFsection}.

\begin{lemma}[{\cite[Remark II.1.2.8]{ogusloggeom}}]\label{lem:artinconesfactorthroughartincones}
    Let $\scr C$ be an Artin fan with a strict \'etale cover $\{\af{Q_i} \to \scr C\}_{i \in I}$ by Artin cones. If $P$ is a monoid, any map $\af{P} \to \scr C$ from its Artin cone factors through some $\af{Q_i} \to \scr C$. Scilicet, the family 
    \[\Hom(\af{P}, \af{Q_i}) \to \Hom(\af{P}, \scr C)\]
    is jointly surjective.
\end{lemma}

\begin{proof}
    The strict \'etale site of an Artin cone $\af{P}$ is trivial in that every strict \'etale cover admits a section \cite[Corollary 2.2.8]{birationalinvarianceabramovichwise}. Pull back the cover of $\scr C$ and find a section of the cover of $\af{P}$.
\end{proof}

\begin{proof}[{Proof of Proposition \ref{prop:logpointfirm=AFsection}}]
    By definition, firmness is equivalent to finding a geometric point $\bar x \to X$ and a retraction 
    \[
    \bar M_Y \to \bar M_{X, \bar x} \dashrightarrow \bar M_Y.
    \]
    We can replace $X$ by a strict \'etale cover using Lemma \ref{lem:firmbasics}, in particular by a disjoint union of atomics $X = \bigsqcup X_i$ \cite[Proposition 2.2.2.5]{logpic}. The Artin fan of $X$ is then a disjoint union of the Artin cones of each connected component $\af{X_i} = \af{P_i}$. Write $Q = \Gamma(Y, \bar M_Y)$.

    Sections of the map $\af{X} = \bigsqcup \af{P_i} \to \af{Y} = \af{Q}$ are the disjoint union of the sections of each map $\af{P_i} \to \af{Y}$ as $\af{Y}$ is a connected Artin cone. A section $\af{Y} \dashrightarrow \af{P_i}$ is dual to a morphism 
    \[P_i \to Q\]
    such that the composite $Q \to P_i \to Q$ is the identity. As $\af{P_i}$ is the Artin fan of $X_i$, there exists a geometric point $\bar x \to X_i$ such that $\bar M_{X_i, \bar x} = P_i$. This geometric point $\bar x$ necessarily maps to $Y$, so the proof is complete.
\end{proof}

\begin{remark}\label{rmk:originclosedpoint}
    Consider the affine toric variety $\Aff_P$ with dense torus $T = \spec \ZZ[\gp P]$ associated to a sharp (f.s.) monoid $P$. The ideal generated by $P^+ = P \setminus \{0\}$ in $\ZZ[P]$ cuts out a closed point $\tilde v \in \Aff_P$. We likewise get a closed stacky point $v \in \af{P}$ as the quotient of $\tilde v \in \Aff_P$ by $T$. Dub both the ``origin.'' 
    
    If $P = \NN^k$ for example, $\tilde v \in \Aff^k$ is the origin and $v \in \af{\NN^k} = \af{}^k$ is the stacky closed point $B\GG_m^k$ with rank-$k$ log structure. Every nonempty closed substack of $\af{P}$ contains $v$, as $P^+ \subseteq P$ is the largest proper monoid ideal of $P$. 

    If $X$ is a log scheme or log algebraic stack with Artin fan $\af{X} = \af{P}$ an Artin cone, the origin $v \in \af{P}$ is always in the image of $X$. Otherwise, $X$ would factor through its open complement $\partial \af{P} \subseteq \af{P}$. But the Artin fan $\af{X}$ of $X$ is defined as the initial Artin fan factoring $X \to \af{X} \to \Log$. This is a contradiction because an open substack of an Artin fan is again an Artin fan.
\end{remark}

\begin{lemma}\label{lem:opendeepeststratasurjectiveAF}
    Suppose the map $X \to \af{X}$ to its Artin fan factors through a closed substack $Z \subseteq \af{X}$ and the factorization $X \to Z$ is open. Then the map $X \to Z$ is surjective. 
\end{lemma}

\begin{proof}
    The question is local in $X$, so assume $X$ is atomic. Then $\af{X} = \af{P}$ is an Artin cone. We can assume $X \neq \varnothing$, so $Z \neq \varnothing$ and $Z$ must contain the origin $v \in \af{P}$ of Remark \ref{rmk:originclosedpoint}. 

    The point $v \in \af{P}$ is always in the image of $X \to \af{X}$ as in Remark \ref{rmk:originclosedpoint}. We assumed the set-theoretic image ${\rm Im}(X \to Z)$ is an open and it contains $v$. The only open of $\af{P}$ containing $v$ is the whole stack $\af{P}$, so ${\rm Im}(X \to Z)$ must likewise be $Z$. 
\end{proof}

\begin{proof}[{Proof of Theorem \ref{thm:thibault}}]
    A log lift clearly implies firmness. For the converse, f.s.\ base change to assume $s = Y$. \'Etale localize to assume the map $f : X \to Y$ has a compatible map on Artin fans $\af{X} \to \af{Y}$ by Lemma \ref{lem:localAFfunctoriality}. Remark \ref{rmk:originclosedpoint} shows the map $Y \to \af{Y}$ factors through the origin $v \in \af{Y}$. Write $v'$ for the pullback of $v \in \af{Y}$ to $\af{X}$:
    \[
    \begin{tikzcd}
        X \ar[r] \ar[dr]       &v' \ar[r] \ar[d] \lpbstrict        &v \ar[d]      \\
                &\af{X} \ar[r]         &\af{Y}.
    \end{tikzcd}
    \]
    The map $X \to v'$ is log flat and strict, hence flat. It is also finitely presented, so open \cite[\href{https://stacks.math.columbia.edu/tag/01UA}{Tag 01UA}]{stacks-project}. Lemma \ref{lem:opendeepeststratasurjectiveAF} ensures $X \to v'$ is surjective. 

    The point $s \longequals Y$ is $f$-firm if and only if there is a section $\af{Y} \dashrightarrow \af{X}$ by Proposition \ref{prop:logpointfirm=AFsection}. Then 
    by the universal property of fiber product
    there is a logarithmic section 
    \[
    \begin{tikzcd}
                &v' \ar[d]         \\
        s \ar[r] \ar[ur, dashed]       &v.
    \end{tikzcd}
    \]
    To come up with a logarithmic lift, it remains to find a section 
    \[
    \begin{tikzcd} 
                &X \ar[d]       \\
        s \ar[r] \ar[ur, dashed]       &v',
    \end{tikzcd}    
    \]
    or equivalently a section of the morphism of algebraic spaces with log structures $X \times_{v'} s \to s$. This morphism is strict, surjective, and finitely presented, so finding such a section is equivalent to finding a lift on the level of underlying algebraic spaces. This is possible after a finite field extension $s' \to s$ \cite[\href{https://stacks.math.columbia.edu/tag/0487}{Tag 0487}]{stacks-project}.
    
    If $f$ was further assumed log smooth, then $X \to v'$ is smooth. Firmness ensures a lift $s \to v'$ as above. Log lifts $s \dashrightarrow X$ correspond to $k$-points of the variety $X' = X \times_{v'} s$. Because $X'$ is smooth and nonempty, it has a $k'$-point for a finite separable extension $k'/k$.
\end{proof}

With similar proof, we can show the firm locus of a log flat map is closed under generization. We first illustrate how to reduce lifting to strictly henselian local rings.

\begin{lemma}[{Raynaud's limit argument}]\label{lem:strhenselianlifting}
    Let $f : X \to Y$ be a morphism of log schemes or log algebraic stacks locally of finite presentation and $p : S \to Y$ a morphism. Let $\hat S$ be the strict henselization of $S$ at a geometric point $\bar s \to S$ and endow $\bar s, \hat S$ with log structure pulled back from $S$. 
    
    The set of lifts of the strict henselization $\hat S \dashrightarrow X$ to $X$ is the colimit under refinements of \'etale neighborhoods $\bar s \to U \to S$ of $\bar s$ of lifts $U \dashrightarrow X$:
    \[
    \left\{
    \begin{tikzcd}[column sep=1.9em]
                &       &       &X \ar[d]      \\
        \bar s \ar[r]      &\hat S \ar[r] \ar[urr, dashed, bend left=15]         &S \ar[r]        &Y
    \end{tikzcd}
    \right\} = 
    \colim_{\bar s \to U \to S}
    \left\{
    \begin{tikzcd}[column sep=1.9em]
                &       &       &X \ar[d]      \\
        \bar s \ar[r]      &U \ar[r] \ar[urr, dashed, bend left=15]      &S \ar[r]      &Y
    \end{tikzcd}
    \right\}.
    \]
\end{lemma}

\begin{proof}
    We can assume $S$ is affine, and then $\hat S$ is a cofiltered limit of affine log schemes given by the strict \'etale neighborhoods $\bar s \to U \to S$ \cite[\href{https://stacks.math.columbia.edu/tag/04HX}{Tag 04HX}]{stacks-project}. We are reduced to \cite[Lemma 2.2.3.4]{logpic}, which states that maps from cofiltered systems $\{S_i\}$ of affine log schemes to locally finite presentation morphisms $X \to Y$ are compatible with limits/colimits:
    \[
        \colim_i \Hom_{{\rm log sch}/Y}(S_i, X) = \Hom_{{\rm log sch}/Y}(\lim S_i, X). \qedhere
    \]
\end{proof}

\begin{proposition}\label{prop:generizationfirmlocus}
    Let $f \colon X \to Y$ be a log flat morphism of log schemes. Let $\sigma \colon S \to Y$ be a morphism and $s \to S$ a strict geometric point. If $s \to S \to Y$ is $f$-firm, then so are all generizations of $s$ in $S$.
\end{proposition}

\begin{proof}
    Replace $S$ by the strict henselization at $s \to S$, to which all generizations lift. Assume $Y = S$ by Lemma \ref{lem:firmbasics}\ref{firmbasics5}. As $s \to Y$ is $f$-firm, there is a strict log geometric point $x \to X = X \times_Y^{\rm fs} s$ and a factorization 
    \[
    \bar M_{S, s} \to \bar M_{X, x} \to \bar M_{S, s} 
    \]
    of the identity map $id_{\bar M_{S, s}}$. We can replace $X$ by its strict henselization at $x$ as well. Write 
    \[P \coloneqq \bar M_{S, s}, \qquad Q \coloneqq \bar M_{X, x},\]
    and $Q'$ for the sharpening of the factorization $Q \to P$, so that there is a factorization
    \[
    P \to Q \to Q' \to P
    \]
    of the identity. Locally near $x$ and $s$, the morphism $X \to S$ is charted by the monoid map $P \to Q$, i.e., there exist strict \'etale neighbourhoods $U$ of $x$ in $X$ and $V$ of $s$ in $S$ and a commutative square
    \[
    \begin{tikzcd}
        U \ar[r] \ar[d]       &\Aff_Q \ar[d]         \\
        V \ar[r]       &\Aff_P
    \end{tikzcd}
    \]
    with strict horizontal maps. Since $X$ and $S$ are local and strictly henselian with closed points $x,s$, we may pick $U=X$ and $V=S$. Write $W \coloneqq S \times_{\Aff_P} \Aff_Q$ and expand the above chart:
    \[
    \begin{tikzcd}
        X \ar[r] \ar[dr]       &W \ar[r] \ar[d] \lpbstrict         &\Aff_Q \times S \ar[r] \ar[d] \lpbstrict         &\Aff_Q \ar[d]         \\
                &S \ar[r]      &\Aff_P \times S \ar[r]        &\Aff_P.
    \end{tikzcd}
    \]
    By \cite[Theorem IV.4.1.7]{ogusloggeom}, the map $X \to W$ is flat. 

    Write 
    \[v \in \Aff_Q(\Spec \ZZ), \qquad v' \in \Aff_{Q'}(\Spec \ZZ), \qquad w \in \Aff_P(\Spec \ZZ)\]
    for the closed points, with underlying schemes $\spec \ZZ$. These are the origins of Remark \ref{rmk:originclosedpoint}. The point $(v, s) \in \Aff_Q \times s$ is a specialization of the point $(v', s) \in \Aff_{Q'} \times s \subseteq \Aff_Q \times s$. Both points map to the image of $s \in \Aff_P \times S$, so they lie in the closed subscheme $W_s = W \times_S s$. The point $x \to X$ maps to $(v, s) \in \Aff_Q$. We can lift the specialization $(v', s) \rightsquigarrow (v, s) \in W_s$ to a specialization $y \rightsquigarrow x$ in $X$ by flatness of $X \to W$ \cite[\href{https://stacks.math.columbia.edu/tag/03HV}{Tag 03HV}]{stacks-project}. 
    
    Replace $X$ by its strict henselization at $y \in X$ and $x$ by $y$. The point $s \to S$ remains $f$-firm, as the same section
    \[P \to Q \to Q' \to P\]
    witnesses firmness of the point $y \in X$ with characteristic monoid $\bar M_{X, y} = Q'$. After replacing $X$, our retraction $Q \to P$ is sharp and we have $Q = Q'$ and $v = v'$. 

    Continue to use notation as above. Consider any geometric point $t \to S$, necessarily a generization of the closed point $s \to S$. The image $t'$ of the composite $t \to \Aff_P \to \Aff_Q$ with the section is a generization of the unique closed ``point'' $v \in \Aff_Q$. So we have a specialization $(t', t) \rightsquigarrow (v, s)$ in $\Aff_Q \times S$. This specialization lies inside the closed subscheme $W$. As $x \to X$ maps to $(v, s) \in \Aff_Q$, we can again lift this specialization $(t', t) \rightsquigarrow (v, s)$ to a specialization $z \rightsquigarrow x$ in $X$ by flatness of $X \to W$ \cite[\href{https://stacks.math.columbia.edu/tag/03HV}{Tag 03HV}]{stacks-project}. 

    We claim $z \to X$ witnesses the firmness of $t \to S$. Write 
    \[\bar P \coloneqq \bar M_{S, t}, \qquad \bar Q \coloneqq \bar M_{X, z}.\]
    As $t' \in \Aff_Q$ was the image of $t$ under the section $\Aff_P \to \Aff_Q$, we have a commutative diagram 
    \[
    \begin{tikzcd}
        P \ar[r] \ar[d] \ar[rr, bend left, equals]       &Q \ar[r] \ar[d]      &P \ar[d]      \\
        \bar P \ar[r]      &\bar Q \ar[r]         &\bar P. 
    \end{tikzcd}
    \]
    It results that the bottom row composes to the identity, so $t \to S$ is firm. 
\end{proof}

\begin{example}
    Proposition \ref{prop:generizationfirmlocus} is not saying that the firm locus is an open subscheme, but that it is a log subalteration with open image. 
    Consider the map 
    \[
    \bra{r} : \Aff^1_x \to \Aff^1_y; \qquad y = x^r
    \]
    for some positive integer $r \in \NN^+$. It factors through the stack quotient $X = \bra{\Aff^1/\mu_r}$. The map $X \to \Aff^1$ is a root stack and it is a monomorphism in the category of (f.s.) log algebraic stacks, while $\bra{r}$ is a Kummer map. 

    We claim the firm locus of $\bra{r}$ is precisely $X \to \Aff^1$. We claim a morphism $S \to \Aff^1$ of log schemes corresponding to $t \in \Gamma(S, M_S)$ lifts to $X$ (uniquely) if and only if there is an $r$th root of $t$. A lift of $S$ along $\bra{r}$ is a \emph{choice} of an $r$th root of $t$. A map $S \to \Aff^1$ is firm along $\bra{r}$ if and only if it factors through $X$, so $X = \Aff^1(\bra{r})$ is the firm locus of $\bra{r}$.
\end{example}

The next section systematically investigates the firm locus.

\section{The firm locus}\label{s:firm}

This section proves Theorem \ref{thm:firmlocus}. As an application, we describe to what extent the firm locus $Y(f) \subseteq Y$ is representable. 

\begin{lemma}\label{lem:intsatsection}
    Let $p : \scr B \to \tau$ be an integral, saturated morphism from an Artin fan $\scr B$ to an Artin cone $\tau$. Either $p$ admits a section $\tau \dashrightarrow \scr B$ or $p$ factors through the open complement $\partial \tau$ of the closed origin of $\tau$. 
\end{lemma}

\begin{proof}
 We can assume $\scr B = \sigma$ is an Artin cone because Lemma \ref{lem:artinconesfactorthroughartincones} guarantees maps $\tau \to \scr B$ factor through Artin cones $\sigma \to \scr B$ in a strict \'etale cover. Write $P = \Gamma(\tau, \bar M_\tau), Q = \Gamma(\sigma, \bar M_\sigma)$ so that $\sigma \to \tau$ comes from a map of sharp monoids $\theta : P \to Q$. The map $\theta$ is sharp if and only if the map $\sigma \to \tau$ does not factor through $\partial \tau \subseteq \tau$. 

    Assuming $\theta$ is sharp, we need to show $p$ admits a section by producing a map $Q \dashrightarrow P$ such that the composite $P \to Q \to P$ is the identity. Note that $P, Q$ are ``toric'' monoids as they are sharp, f.s., with free associated groups $\gp P, \gp Q$ and that the integral map $P \to Q$ is locally exact by \cite[Theorem I.4.7.7]{ogusloggeom}. So Theorem I.4.8.14 of \cite{ogusloggeom} provides an isomorphism $P \times (Q \setminus \sqrt{K_\theta}) \simeq Q$ for a certain monoid ideal $\sqrt{K_\theta} \subseteq Q$. 
    % By \cite[Theorem I.4.8.14 (7)]{ogusloggeom}, there is an isomorphism $P \times (Q \setminus \sqrt{K_\theta}) \simeq Q$ for a certain ideal $\sqrt{K_\theta} \subseteq Q$. We can use this theorem because 
    % \begin{itemize}
    %     \item sharp f.s.\ monoids have free associated groups and are thus toric and 
    %     \item locally exact morphisms of f.s.\ monoids are the same as $\QQ$-integral morphisms \cite[Theorem I.4.7.7]{ogusloggeom}.
    % \end{itemize}
    Projecting away from this ideal gives us our retraction $Q \dashrightarrow P$. 
\end{proof}

\begin{proof}[{Proof of Theorem \ref{thm:firmlocus}}] 
    Consider a strict log geometric point $\bar y \to Y$. Replace $Y$ by its strict henselization at $\bar y$ and $X$ by a strict \'etale cover by a disjoint union of atomics using Lemma \ref{lem:firmbasics}. 
    
    Let $X' = X \times_Y^{\msout{\rm fs}} \bar y$ be the (f.s.\ and scheme-theoretic) fiber over $\bar y$. There are compatible maps of Artin fans 
    \[
    \begin{tikzcd}
        \af{X'} \ar[r] \ar[d, "\pi", swap]         &\af{X} \ar[d]         \\
        \af{\bar y} \ar[r, equals]         &\af{Y}
    \end{tikzcd}
    \]
    by Lemmas \ref{lem:strictAFfunctorial}, \ref{lem:localAFfunctoriality} because $X, Y$ are disjoint unions of atomics and $X' \to X, \bar y \to Y$ are strict. Because the rest of the morphisms in the square are compatible with maps of log schemes, there is a commutative square 
    \[
    \begin{tikzcd}
        X' \ar[r] \ar[d]      &\af{X'} \ar[d, "\pi"]        \\
        \bar y \ar[r]      &\af{\bar y}.
    \end{tikzcd}    
    \]
    
    The point $\bar y \to Y$ is in the set-theoretic image of $f$ if and only if $X'$ is nonempty. In that case, the map $\pi : \af{X'} \to \af{\bar y}$ admits a section if and only if $\bar y \to Y$ is $f$-firm by Proposition \ref{prop:logpointfirm=AFsection}. It is clear that firm points are in the set-theoretic image of $f$.
    So it suffices to show that if $\pi$ does not admit a section, then $X'$ is empty. 

    By Lemma \ref{lem:intsatsection}, $\pi$ factors through $\partial \af{\bar y} \subseteq \af{\bar y}$. Then we have a commutative rectangle
    \[
    \begin{tikzcd}
        X' \ar[r] \ar[d]      &\af{X'} \ar[r]        &\partial \af{\bar y} \ar[d]       \\
        \bar y \ar[r]      &v \ar[r]       &\af{\bar y},
    \end{tikzcd}
    \]
    where $v \in \af{\bar y}$ is the closed point as in Remark \ref{rmk:originclosedpoint}. But $\partial \af{\bar y}$ is the complement $\af{\bar y} \setminus v$, so $X' = \varnothing$. 
\end{proof}

\begin{lemma}\label{lem:firmlocuspullsback}
    Let $f : X \to Y$ be a morphism of log schemes or log algebraic stacks and $Y(f) \subseteq Y$ the firm locus. Given a map $Y' \to Y$, consider the base change $f' : X' \coloneqq X \times_Y^{\rm fs} Y' \to Y'$ and its firm locus $Y'(f')$. There is an f.s.\ pullback square 
    \[
    \begin{tikzcd}
        Y'(f') \ar[r] \ar[d] \lpb      &Y' \ar[d]         \\
        Y(f) \ar[r]        &Y.
    \end{tikzcd}
    \]
\end{lemma}

\begin{proof}
    Suppose given a map $S \to Y'$. The composite $S \to Y' \to Y$ factors through $Y(f)$ if and only if, for all log geometric points $\bar s \to S$, the identity $\bar s \longequals \bar s$ is firm along the pullback $X_{\bar s} \to \bar s$ by Lemma \ref{lem:firmbasics}\ref{firmbasics5}\ref{firmbasics iv}. But this is the same as $\bar s \longequals \bar s$ being firm along the pullback $X'_{\bar s}=X_{\bar s} \to \bar s$, so $S$ equivalently factors through $Y'(f')$. 
\end{proof}

We need a version of a well-known blowup lemma (see 
\cite{AbrKaru00, fkatointegralblowup, universalsemistablereductionmolcho}).

\begin{lemma}[{\cite[Theorem 1.1, Corollary 1.2]{fkatointegralblowup}, \cite[Theorem 1.0.1]{universalsemistablereductionmolcho}}]\label{lem:logaltnint+sat}
    Let $f : X \to Y$ be a quasicompact map of (f.s.) log schemes with $Y$ quasicompact. There are 
    \begin{itemize}
        \item a log blowup $Y_1 \to Y$ and 
        \item a log alteration $Y_2 \to Y$
    \end{itemize}
    with f.s.\ pullbacks $f_i : X_i \coloneqq Y_i \times^{\rm fs}_Y X \to Y_i$ such that $f_1$ is integral and $f_2$ is integral and saturated.
\end{lemma}

\begin{proof}
    F.\ Kato's results \cite[Theorem 1.1, Corollary 1.2]{fkatointegralblowup} show that such a base change is possible after passing to a strict \'etale neighborhood of any point of $Y$. Instead of a log alteration, he takes a finite Kummer map. Molcho \cite[Theorem 1.0.1]{universalsemistablereductionmolcho} instead provides a log alteration. To conclude, we need to find a global log blowup/alteration $Y' \to Y$ refining any local one. 

    One can define a special family $Y'_n \to Y$ of log blowups/alterations indexed by $n \in \NN$. First, take a log blowup of $\af{Y}$ to assume its cones are free as in the proof of \cite[Chapter 1, \S 2, Theorem 11]{KTE}. Subdivide further to also assume $\af{Y}$ is without monodromy in the sense that it admits a \emph{Zariski} cover by free Artin cones $\af{}^r = \bra{\Aff^1/\GG_m}^r$.

    For each cone $\sigma = \af{}^r$ of $\af{Y}$, we define a system of log blowups and log alterations and leave the reader to check that they glue along faces. See Figure \ref{fig:superbarycentricsubdivision}. 

    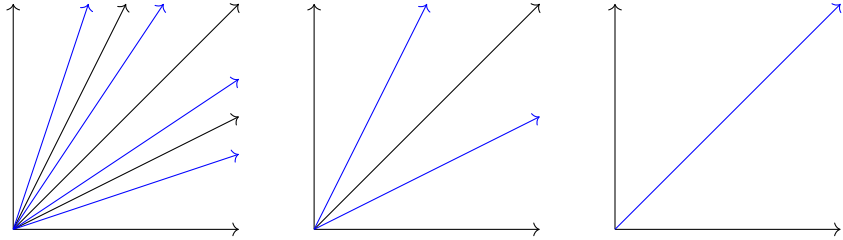
\begin{figure}[h]
    \centering
    \begin{tikzpicture}
        \draw[->] (0, 0) to (3, 0);
        \draw[->] (0, 0) to (0, 3);
        \draw[->, blue] (0, 0) to (3, 3);
        \begin{scope}[shift = {(-4, 0)}]
        \draw[->] (0, 0) to (3, 0);
        \draw[->] (0, 0) to (0, 3);
        \draw[->] (0, 0) to (3, 3);
        \draw[->, blue] (0, 0) to (3, 1.5);
        \draw[->, blue] (0, 0) to (1.5, 3);
        \end{scope}
        \begin{scope}[shift = {(-8, 0)}]
        \draw[->] (0, 0) to (3, 0);
        \draw[->] (0, 0) to (0, 3);
        \draw[->] (0, 0) to (3, 3);
        \draw[->] (0, 0) to (3, 1.5);
        \draw[->] (0, 0) to (1.5, 3);
        
        \draw[->, blue] (0, 0) to (3, 1);
        \draw[->, blue] (0, 0) to (1, 3);
        \draw[->, blue] (0, 0) to (3, 2);
        \draw[->, blue] (0, 0) to (2, 3);
        \end{scope}
    \end{tikzpicture}
    \caption{The sequence $\Sigma_n$ of log blowups of $\af{}^2$ for $n = 1, 2, 3$. This sequence eventually refines all log blowups. The log alteration $\Sigma'_n$ also adds in an increasing sequence of root stacks to refine all log alterations. }
    \label{fig:superbarycentricsubdivision}
\end{figure}
    %New paragraph
    A log alteration $\scr B \to \sigma$ factors as 
    \[
        \scr B \to \scr B' \to \sigma
    \]
    with $\scr B \to \scr B'$ a root stack and $\scr B' \to \sigma$ a log modification. By quasicompactness, there is a natural number $N \in \NN$ and a factorization
    \[
        [N] : \scr B' \dashrightarrow \scr B \to \scr B'
    \]
    of the multiplication by $N$ map $[N]$ on $\scr B'$. The log modification $\scr B' \to \sigma$ is refined by a log blowup, which can be refined by a stellar subdivision. So it suffices to construct a family of subdivisions that refines all stellar subdivisions and all roots to finite order $N \in \NN$. 
    
    Write $e_1, \cdots, e_r \in \NN^r$ for the standard basis vectors. For each set of integers $a_1, \cdots, a_r \in \{0, 1, 2, \cdots, n\}$ bounded by $n$, we get a map $\NN^r \to \NN$ defined by the $(1 \times r)$-matrix $[a_1 \cdots a_r]$. 
    
    Consider the star subdivision of $\sigma$ at the dual map of Artin cones $\af{} \to \af{}^r$. If we perform star subdivision at a finite set of vectors $\af{} \to \af{}^r$, the result depends on the order of the vectors. Nevertheless, there is a unique minimal subdivision $\Sigma_n \to \af{}^r$ refining each of these star subdivisions given by taking their intersection as subfunctors of $\af{}^r$ depicted in Figure \ref{fig:orderedstellarsubdivisions}.  

\begin{figure}[h]
    \centering
    \begin{tikzpicture}
      % Define coordinates
      \coordinate (A) at (0,0);
      \coordinate (B) at (3,0);
      \coordinate (C) at (1.5,3);
      \coordinate (D) at (1,1);
      \coordinate (E) at (2,1);
      \coordinate (F) at (1.5, .5);
    
      % Draw lines between them
      \draw (A) -- (B) -- (C) -- cycle; % a closed polygon
      % or draw individual lines
      \draw (D) -- (A);
      \draw (D) -- (B);
      \draw (D) -- (C);
      \draw (E) -- (D);
      \draw (E) -- (B);
      \draw (E) -- (C);
      \draw (F) -- (D);
      \draw (F) -- (A);
      \draw (F) -- (B);

      \fill (D) circle[radius=1.5pt];
      \fill (E) circle[radius=1.5pt];
      \fill (F) circle[radius=1.5pt];
      \node[left] at (D) {$u$};
      \node[right] at (E) {$v$};
      \node[below] at (F) {$w$};

    \begin{scope}[shift = {(4, 0)}]
           % Define coordinates
      \coordinate (A) at (0,0);
      \coordinate (B) at (3,0);
      \coordinate (C) at (1.5,3);
      \coordinate (D) at (1,1);
      \coordinate (E) at (2,1);
      \coordinate (F) at (1.5, .5);
    
      % Draw lines between them
      \draw (A) -- (B) -- (C) -- cycle; % a closed polygon
      % or draw individual lines
      \draw (E) -- (A);
      \draw (E) -- (B);
      \draw (E) -- (C);
      \draw (D) -- (E);
      \draw (D) -- (C);
      \draw (D) -- (A);
      \draw (F) -- (A);
      \draw (F) -- (B);
      \draw (F) -- (E);
      \fill (D) circle[radius=1.5pt];
      \fill (E) circle[radius=1.5pt];
      \fill (F) circle[radius=1.5pt];
      \node[left] at (D) {$u$};
      \node[right] at (E) {$v$};
      \node[below] at (F) {$w$};
    \end{scope}

    \begin{scope}[shift = {(8, 0)}]
           % Define coordinates
      \coordinate (A) at (0,0);
      \coordinate (B) at (3,0);
      \coordinate (C) at (1.5,3);
      \coordinate (D) at (1,1);
      \coordinate (E) at (2,1);
      \coordinate (F) at (1.5, .5);
    
      % Draw lines between them
      \draw (A) -- (B) -- (C) -- cycle; % a closed polygon
      % or draw individual lines
      \draw (D) -- (A);
      \draw (D) -- (B);
      \draw (D) -- (C);
      \draw (D) -- (E);
      \draw (D) -- (F);
      \draw (E) -- (A);
      \draw (E) -- (B);
      \draw (E) -- (C);
      \draw (E) -- (F);
      \draw (F) -- (A);
      \draw (F) -- (B);
      \draw (F) -- (C);
      \fill (D) circle[radius=1.5pt];
      \fill (E) circle[radius=1.5pt];
      \fill (F) circle[radius=1.5pt];
      \node[left] at (D) {$v$};
      \node[right] at (E) {$w$};
      \node[below] at (F) {$w$};
    \end{scope}
    \end{tikzpicture}
    \caption{Left: The stellar subdivision at $u$, then $v$, then $w$. Middle: That at $v$, then $u$ and $w$ in either order. Right: The intersection of all stellar subdivisions at $u, v, w$ in any order. }
    \label{fig:orderedstellarsubdivisions}
\end{figure}
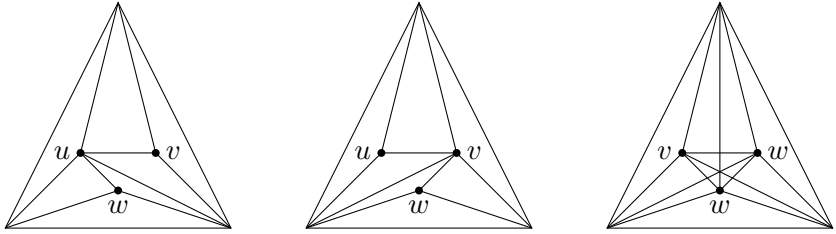
    
    Let $\Sigma'_n \to \Sigma_n$ be the root stack given by rescaling the lattice by $n!$. For $n \gg 0$, $\Sigma_n \to \sigma$ refines any log blowup and $\Sigma'_n \to \sigma$ refines any log alteration. This is because 
    As we assumed $\af{Y}$ was free, we get an induced log alteration of $\af{Y}$ which pulls back to define $Y'_n \to Y$. Check that it refines every possible log alteration of $Y$ or any $Y_0 \to Y$.
\end{proof}

We review formal schemes and logarithmic spaces to discuss representability of the firm locus.

\begin{definition}\label{def:weaklogspace}
    Let $X$ be a sheaf on the big strict \'etale site of (f.s.) log schemes. Recall that $X$ is a 
    \begin{itemize}
        \item \emph{logarithmic space} if it admits a log smooth cover $X' \to X$ by a log scheme \cite[Theorem C]{logpic}, and
        \item \emph{weak logarithmic space} if it is a root stack of a logarithmic space.
    \end{itemize}
\end{definition}

\begin{remark}
    A weak logarithmic space $X$ need not be a logarithmic space -- we are explicitly allowing roots divisible by the characteristic. However, locally in the log smooth topology near a point $x \to X$ of residue characteristic $p$, there must exist a covering $f \colon X' \to X$ by a log scheme such that $f$ is a $\mu_p^{\Lambda}$-torsor for some free commutative abelian group $\Lambda$. A fortiori, $X$ has a log flat cover by a log scheme.
\end{remark}

Let $T \subseteq Y$ be a locally constructible subset of a scheme $Y$. View it as a functor on schemes by 
\begin{equation}\label{eqn:constructibleisindformalscheme}
    T(X) \coloneqq \{f : X \to Y \, | \, f^{-1}(T) = X\}.
\end{equation}
We explain how to view this as a formal scheme. 

\begin{lemma}\label{lem:constructiblesetsasformalschemes}
    The functor $T$ in \eqref{eqn:constructibleisindformalscheme} is representable by a formal scheme which is an ind-subscheme of $Y$. 
\end{lemma}

\begin{proof}
    The claim is Zariski-local in $Y$, so we can assume $Y$ is affine and $T \subseteq Y$ is the intersection $U \cap D$ of an open $U$ and a closed subset $D$ of $|Y|$. It suffices to handle the cases $T = U$ and $T = D$ separately, and $T = U$ is an open subscheme of $X$. The functor $T = D$ is representable by the formal spectrum $\Spf (\OO_Y)^\wedge_{I_D}$ of the completion of $\OO_Y$ at the ideal $I_D$ of $D$, an ind-scheme. See \cite[\href{https://stacks.math.columbia.edu/tag/0AIZ}{Tag 0AIZ}]{stacks-project} for more details.
\end{proof}
\begin{corollary}\label{cor:firmlocusalteration}
    Let $f : X \to Y$ be a finitely presented morphism of log schemes with $Y$ quasicompact. There is a log alteration $Y' \to Y$ along which the firm locus pulls back to a locally constructible subset $Y(f) \times_Y^{\rm fs} Y' \subseteq Y'$, interpreted as a formal scheme and ind-subscheme as in \eqref{eqn:constructibleisindformalscheme}. If $f$ is log flat, the locally constructible subset of $Y'$ is an open subscheme. 
\end{corollary}

In other words, the firm locus $Y(f) \subseteq Y$ in this corollary is an ind-weak logarithmic space as in Definition \ref{def:weaklogspace}, where ``ind-'' refers to filtered colimits \cite[\href{https://stacks.math.columbia.edu/tag/05PW}{Tag 05PW}]{stacks-project}. It is a weak logarithmic space if $f$ is log flat. 

\begin{proof}
    Lemma \ref{lem:logaltnint+sat} provides a log alteration $Y' \to Y$ such that the pullback 
    \[
    \begin{tikzcd}
        X' \ar[r] \ar[d, "f'", swap] \lpb      &X \ar[d, "f"]      \\
        Y' \ar[r]      &Y
    \end{tikzcd}
    \]
    of $f$ is integral and saturated. By Lemma \ref{lem:firmlocuspullsback}, the firm locus pulls back:
    \[Y(f) \times_Y^{\rm fs} Y' = Y'(f').\]

    Replace $f$ by the integral, saturated $f'$. Theorem \ref{thm:firmlocus} identifies the set-theoretic image $f(X) \subseteq Y$ as the firm locus. By Chevalley's theorem \cite[\href{https://stacks.math.columbia.edu/tag/054K}{Tag 054K}]{stacks-project}, $f(X)$ is locally constructible. If $f$ is assumed log flat, it is flat (since it is also integral) by \cite[Theorem IV.4.3.5]{ogusloggeom} and $f(X) \subseteq Y$ is open \cite[\href{https://stacks.math.columbia.edu/tag/01UA}{Tag 01UA}]{stacks-project}.
\end{proof}

\begin{example}
    Let $f : X \to Y$ be the inclusion $\vec 0 \in \Aff^1$ of the origin. As $f$ is strict, Theorem \ref{thm:firmlocus} shows a map $S \to Y$ is $f$-firm if and only if it set-theoretically factors through the origin $\vec 0 \in \Aff^1$. But this includes all closed subschemes $Z \subseteq \Aff^1$ supported at the origin, for example $Z = V(x^n)$ for any $n \in \NN$. Firmness is defined in terms of geometric points, and those of $Z$ are the same as the geometric points of the origin $\vec 0 \in \Aff^1$. As in Lemma \ref{lem:constructiblesetsasformalschemes}, the $f$-firm locus is the formal scheme 
    \[
        \colim_{\sqrt{I} = (x)} V(I) = \Spf k \adj{x} \qquad \subseteq \Aff^1.
    \]
\end{example}

\begin{lemma}\label{lem:firmlocusfactorthroughblowupaltn}
    Let $f : X \to Y$ be a morphism locally of finite presentation which factors through a log alteration $Y' \to Y$. The firm locus $Y(f)$ also factors uniquely through $Y' \subseteq Y$.
\end{lemma}

\begin{proof}
    Let $S \to Y$ be a map firm along $f$. The problem is local in $S, Y, X$. Assume $S$ is strictly henselian by Lemma \ref{lem:strhenselianlifting} and then base change to assume $Y = S$. Replace $X$ by a strict \'etale cover to assume it is a disjoint union of atomics. 

    As $Y$ is strictly henselian, we have a (f.s.\ and ordinary) pullback square 
    \[
    \begin{tikzcd}
        Y' \ar[r] \ar[d] \lpbstrict     &\scr B \ar[d]         \\
        Y \ar[r]       &\scr C
    \end{tikzcd}
    \]
    with $\scr B \to \scr C$ a log alteration of Artin fans by Definition \ref{defn:logalteration}. We need to show that $S \to Y$ factors through $Y'$.

    The Artin fan is then the same as that of its closed point $\bar s \to S$ as in Example \ref{ex:strhenselianAF}: 
    \[\af{S} = \af{\bar s}.\]

    By Lemma \ref{lem:firmbasics}\ref{firmbasics iv}, $S$ being firm implies that the log geometric point $\bar s \to S \to Y$ is firm.
    By Proposition \ref{prop:logpointfirm=AFsection}, the log geometric point $\bar s \to S \to Y$ is firm if and only if the map $\af{X_{\bar s}} \to \af{\bar s}$ admits a section. We may have $\af{X_{\bar s}} \neq \af{X}$ as in Example \ref{ex:AFbasechangecounterex}. The strict map $X_{\bar s} \to X$ nevertheless induces $\af{X_{\bar s}} \to \af X$ by Lemma \ref{lem:strictAFfunctorial} and 
    the factorization through $\af{X_{\bar s}}$ induces a dashed arrow 
    \[
    \begin{tikzcd}
       &         &       &\scr B \ar[d]       \\
   S \ar[r] &  \af{\bar s} = \af{S} \ar[r,"="] \ar[urr, dashed, bend left=15]        &\af{Y} \ar[r]         &\scr C,
    \end{tikzcd}
    \]
    where the map $\af{Y} \to \scr C$ exists because $Y \to \scr C$ is strict, hence factors via $\af{Y}$. But then $S \to Y \to \scr C$ factors through $\scr B$ and so $S \to Y$ factors through $Y' \to Y$.
\end{proof}

\begin{example}[{\cite[Remark after Proposition 2.6]{logetale2nakayama}}]
    We give an example due to C.\ Nakayama of a sequence $X \to Y' \to Y$ where both $X \to Y'$ and $X \to Y$ are log blowups but $Y' \to Y$ is not a log modification in our sense. Lemma \ref{lem:firmlocusfactorthroughblowupaltn} fails for the sequence $X \to Y' \to Y$. 
    
    Let $k = \bar k$ be an algebraically closed field for simplicity. Let $Y = \Spec k [x, y]/(x^2, y^2)$ and endow it with the strict log structure from its natural inclusion in $\Aff^2$. Let $X \to Y$ be the blowup at $(x, y)$. Then $X \to Y$ factors through the reduced closed point 
    \[Y' \coloneqq \Spec k \subseteq \Spec k[x, y]/(x^2, y^2)\]
    because on $X$, the equations
    \[xy = x^2(y/x) = 0, \qquad xy = y^2(x/y) = 0\]
    on each chart of the blowup force $xy = 0$. The sequence $X \to Y' \to Y$ fits the above description. The identity $Y \longequals Y$ is firm along $p : Y' \to Y$ because $Y', Y$ have the same set of log geometric points, so the firm locus $Y(p) = Y$ does not factor through $Y'$.
\end{example}

Recall that a morphism $X \to Y$ of log schemes is \emph{log reduced} if $X \to \Log_Y$ has reduced geometric fibers.

\begin{remark}\label{rmk:relAFfirmness}
    Let $f : X \to Y$ be a log flat, log reduced morphism of finite presentation. By \cite{RdC}, there exists a \emph{relative Artin fan} $\af{X/Y}$ factoring $X \to \af{X/Y} \to Y$. 
    
    In good situations, the proof of Theorem \ref{thm:thibault} applies to show $X \to \af{X/Y}$ is surjective. Then $X$ and $\af{X/Y}$ have the same images in $Y$. If $f$ is integral and saturated, the firm locus can then equally be characterized as the image of $\af{X/Y} \to Y$ by Theorem \ref{thm:firmlocus}.
\end{remark}

Using the idea of the relative Artin fan, we prove a technical lemma for later. 

\begin{lemma}\label{lem:relAFbasechange}
    Let $X \to Y$ be a log flat, log reduced morphism of finite presentation and $\bar y \to Y$ a strict geometric point. Write $X_{\bar y} = X \times_Y^{\rm fs} \bar y$ for the f.s.\ fiber.
    Then the fibers of the Artin fans of $\af{X}$ and $\af{X_{\bar y}}$ over $\bar y$ coincide. That is, there exists an isomorphism $\af{X_{\bar y}} \times_{\af{\bar y}} \bar y \longsimeq \af{X} \times_{\af{Y}} \bar y$ which makes the triangle
    \[
    \begin{tikzcd}
        X_{\bar y} \arrow[d] \ar[dr, bend left = 15] & \\
        \af{X_{\bar y}} \times_{\af{\bar y}} \bar y \arrow[r, "\sim"] & \af X \times_{\af Y} \bar y
    \end{tikzcd}
    \]
    commute.
\end{lemma}

\begin{proof}
    Any finitely presented, flat morphism $W \to Z$ with reduced geometric fibers admits an initial factorization through an \'etale $Z$-algebraic space by \cite[Theorem 2.5.2]{romagnypi0}. The morphisms $X_{\bar y} \to \Log_{\bar y}$, $X \to \Log_Y$ meet these criteria. Write $\scr B, \scr C$ for their initial factorizations, which are the \emph{relative Artin fans} \cite{RdC} of $X_{\bar y} \to \bar y$ and $X \to Y$. 
    
    The maps $X_{\bar y} \to \scr B$, $X \to \scr C$ have connected geometric fibers by loc.\ cit. We have a diagram 
    \[
    \begin{tikzcd}
        X_{\bar y} \ar[r] \ar[d]      &X \ar[d]     \\
        \scr B \ar[r, dashed] \ar[d]      &\scr C \ar[d]       \\
        \Log_{\bar y}  \ar[r] \ar[d] \lpbstrict       &\Log_Y \ar[d]         \\
        \bar y \ar[r]          &Y,
    \end{tikzcd}    
    \]
    where the dashed arrow exists by the initialness of $\scr B$ and the fact that $\scr C_{\bar y} \coloneqq \scr C \times_{\Log_Y} \Log_{\bar y}$ is \'etale and representable over $\Log_{\bar y}$.

    The morphisms $X_{\bar y} \to \scr B$ and $X \to \scr C$ are flat of finite presentation since $X_{\bar y} \to \Log_{\bar y}$ and $X \to \Log_Y$ are. In particular, they are open and we may write $U \subseteq \scr B, W \subseteq \scr C$ for their open, nonempty images. Since $X \to \Log_{\bar y}$ factors through the \'etale map $U \to \Log_{\bar y}$, we have $U=\scr B$ by the initialness of $\scr B$, so $X_{\bar y} \to \scr B$ is surjective. Likewise, $X \to \scr C$ is surjective. The maps 
    \[X_{\bar y} \to \scr B, \qquad X_{\bar y} \to \scr B \to \scr C_{\bar y}\]
    are surjective morphisms with connected geometric fibers. By \cite[Proposition 3.13]{loghochschild}, the map $\scr B \to \scr C_{\bar y}$ is then an isomorphism.
\end{proof}

\section{Applications to finding rational points}\label{s:appnsrationalpoints}

This section connects our work with Abramovich's theory of firmaments of toroidal morphisms in \cite[\S2.4]{abramovich09ClayNotes}. 
We give some background on toroidal embeddings in Section~\ref{Toroidal-maps-log-smoothness}, and review the cone complex $\Sigma_X$ associated to a toroidal embedding $X$ in Section~\ref{ss:AFvsconecomplexes}.

A toroidal morphism of toroidal embeddings $f : (X,U_X) \to (Y,U_Y)$ induces functorially a compatible map of cone complexes $\Sigma_X \to \Sigma_Y$. We equip toroidal embeddings with their natural log structures as in Section~\ref{Toroidal-maps-log-smoothness}.

In that case, Abramovich attaches to $f$ a combinatorial object called the \emph{firmament} $\Gamma_f \subseteq \Sigma_Y(\NN)$ of $f$ consisting of integral points of the cone complex $\Sigma_Y$ of $Y$ (cf.\ Definition \ref{def:firmamentcontactorder}). Let $\phi : \spec R \to Y$ be a morphism from a discrete valuation ring $R$ that sends the open point of $\spec R$ to the dense open  $U_Y \subseteq Y$. Abramovich defines an element ``$n_\phi$'' in $\Sigma_Y(\NN)$ (cf.\ 
Proposition \ref{prop:firmamentssameasAbramovich}). He then claims that to lift an $R$-point along $f$, it is necessary and sometimes sufficient that the contact order $n_\phi$ of $\Spec R$ in $Y$ lies in the firmament $\Gamma_f \subseteq \Sigma_Y(\NN)$\footnote{Abramovich says a point $\Spec R \to Y$ is ``firm'' if its contact order lies in the firmament in $\Sigma_Y$. This inspired our use of the term, but we simply say the point ``lies in the firmament'' to avoid confusion.}\footnote{The firmament $\Gamma_f$ defines a unique set $M$ for the toroidal boundary such that a point lies in the firmament if an only if it is an M-point in the sense of \cite{moerman}.}. 

This section is organized as follows. Section \ref{ss:firmamentsvsabramovich} introduces the firmament and compares it with the notion of Abramovich \cite{abramovich09ClayNotes}. We compare the firmament with our notion of firmness in 
Section \ref{ss:abramovichrelation}. In Section~\ref{lifting_rat}, we prove the claim from \cite[\S2.4.16]{abramovich09ClayNotes} concerning the lifting of points lying in the firmament mentioned above using preliminary results established earlier in this paper. In Section \ref{ss:stablemaps} we compare the firmament and type for log stable maps.

The set  $\Sigma_X(\NN)$ of \emph{integral points of the cone complex} $\Sigma_X$ can be defined (Definition \ref{defn:integralpoints}) for a general log scheme $X$ as 
\[
\Sigma_X(\NN) \coloneqq \Hom(\af{}, \af{X}).
\]
Proposition \ref{prop:functoriallyidentifyingconecomplexes} shows this set $\Sigma_X(\NN)$ is functorial in $X$ for morphisms of log schemes $f : X \to Y$ in a way that reproduces the functoriality of cone complexes of toroidal embeddings. If $f : X \to Y$ is a log blowup, the sets of integral points of cone complexes are identified 
\[
\Sigma_f : \Sigma_X(\NN) \longsimeq \Sigma_Y(\NN)
\]
in Corollary \ref{cor:logblowupsameintegralconecomplex}. 

\begin{example}
    Let $b : B = Bl_{\vec 0} \Aff^2 \to \Aff^2$ be the blow-up of $\Aff^2$ at the origin. The cone complex of $\Aff^2$ consists of one cone $\RR_{\geq 0}^2$, with set of integral points $\NN^2$. The morphism $b$ is toric, and hence toroidal, corresponding to the subdivision of $\RR_{\geq 0}^2$ into two cones along the diagonal. The map of cone complexes and the induced identification of sets of integral points is depicted in Figure \ref{fig:blowupA2}.

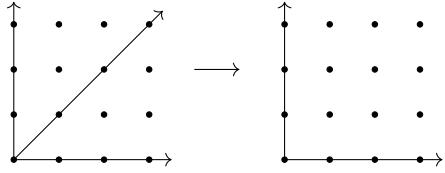
\begin{figure}[h]  
    \centering
    \begin{tikzpicture}[scale=.6]
        \draw[->] (0, 0) to (3.5, 0);
        \draw[->] (0, 0) to (0, 3.5);
        \foreach \x in {0,1,2,3}{
            \foreach \y in {0,1,2,3}{
                \draw[fill=black] (\x,\y) circle (0.06);
            }
        }

        \draw[->] (-2, 2) to (-1, 2);

        \begin{scope}[shift = {(-6, 0)}]
            \draw[->] (0, 0) to (3.5, 0);
            \draw[->] (0, 0) to (0, 3.5);
            \draw[->] (0, 0) to (3.3, 3.3);
            \foreach \x in {0,1,2,3}{
                \foreach \y in {0,1,2,3}{
                    \draw[fill=black] (\x,\y) circle (0.06);
                }
            }
        \end{scope}
    \end{tikzpicture}
    \caption{The map on cone complexes corresponding to the log blowup of $\Aff^2$ at the origin. Remark that source and target have the same set $\NN^2$ of integral points. }
    \label{fig:blowupA2}
\end{figure}

    Regarding $b$ as a morphism of log schemes, it is pulled back from the subdivision of Artin fans $\af{B} \to \af{\Aff^2} = \af{}^2$ which may be obtained by quotienting $b$ by $\GG_m^2$ depicted in Figure \ref{fig:blowupA2}. As $\af{B} \to \af{}^2$ is a subdivision, source and target have the same set of integral points 
    \[\Sigma_B(\NN) = \Sigma_{\Aff^2}(\NN) = \NN^2\]
    divided among a different set of cones.
\end{example}

\subsection{The firmaments of Abramovich}\label{ss:firmamentsvsabramovich}

\begin{definition}\label{def:firmamentcontactorder}
    For a morphism $f : X \to Y$ of log schemes (or log algebraic stacks), define the \emph{firmament} $\Gamma_f \subseteq \Sigma_Y(\NN)$ as the set-theoretic image $\Gamma_f \coloneqq \Sigma_f(\Sigma_X(\NN))$ of the morphism of cone complexes 
    \[
        \Sigma_f : \Sigma_X(\NN) \to \Sigma_Y(\NN).
    \]
    If $p : S \to Y$, $f : X \to Y$ are maps of log schemes and $S$ has log structure of rank at most one everywhere, we say that $p$ \emph{lies in the firmament of }$f$ if $\Gamma_p \subseteq \Gamma_f$ as subsets of $\Sigma_Y(\NN)$. 
\end{definition}

If $f : X \to Y$ is a flat morphism of toroidal embeddings, this is the definition of the ``base (toroidal) firmament'' in \cite[Definition 2.4.13]{abramovich09ClayNotes}; this is shown to coincide with the valuative definition of firmaments in the next section \cite[\S 2.4.14]{abramovich09ClayNotes}.

If $X' \to X$ is a morphism inducing a surjection $\Sigma_{X'}(\NN) \to \Sigma_X(\NN)$ on cone complexes, $f$ and the composite $X' \to X \to Y$ have the same firmament. We can thus localize freely in the source and target (cf.\ Proposition \ref{lem:artinconesfactorthroughartincones}). Firmaments are $\NN$-sets (they are closed under scaling by $\NN$) but they need not be sublattices or submonoids in the cones of $\Sigma_Y$. 

\begin{example}\label{Disjoint_union}
    For each $r \in \NN$, we get a map 
    \[\bra{r} : \Aff^1 \to \Aff^1; \qquad t \mapsto t^r.\]
    Let $f : X \to Y$ be the disjoint union of two of these maps, say for $r = 2$ and $r = 3$:
    \[\bra{2} \sqcup \bra{3} : \Aff^1 \sqcup \Aff^1 \to \Aff^1.\]
    The map on integral points of cone complexes is the disjoint union 
    \[\NN \sqcup \NN \to \NN\]
    of the maps defined by $1 \mapsto 2$ and $1 \mapsto 3$. The resulting firmament is the union 
    \[
    \Gamma_f = 2 \NN \cup 3 \NN \qquad \subseteq \NN = \Sigma_Y,
    \]
    which is closed under scaling by $\NN$ but lacks the element $5$ and so is not a submonoid of $\NN$. 
\end{example}

\begin{example}[{\cite[Example (9), \S 2.4.15]{abramovich09ClayNotes}}]\label{ex:firmamentrootcone}
The inclusion of monoids
\[Q_1 \coloneqq \left\{(a, b) \in \NN^2 \, \middle| \, 2 | (a + b)\right\} \qquad \subseteq \NN^2\]
in Figure \ref{fig:firmamentrootcone} leads to an equivariant morphism of affine toric varieties 
\[
f : \Spec \ZZ[s, t, \sqrt{st}] \to \Spec \ZZ[s, t].
\]
Its firmament is simply $Q_1 \subseteq \NN^2$. More generally, root stacks and Kummer maps $f : X \to Y$ have associated inclusions $\Sigma_X \subseteq \Sigma_Y$ of cone complexes and their firmaments are precisely $\Sigma_X \subseteq \Sigma_Y$.

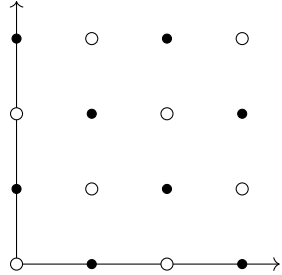
\begin{figure}[h]
    \centering
    \begin{tikzpicture}
        \draw[->] (0, 0) to (3.5, 0);
        \draw[->] (0, 0) to (0, 3.5);
        \foreach \x in {0,1,2,3}{
            \foreach \y in {0,1,2,3}{
                \pgfmathparse{int(\x+\y)}
                \ifodd\pgfmathresult
                    \draw[fill=black] (\x,\y) circle (0.06);
                \else
                    \draw[fill=white] (\x,\y) circle (0.08);
                \fi
            }
        }
    \end{tikzpicture}
    \caption{Depicted is the set $\Hom(\af{}, \af{}^2) = \Sigma_{\Aff^2}$ of possible multiplicites along the $x-$ and $y-$axes in $\Aff^2$. The white
    circles are the points of the firmament in Example \ref{ex:firmamentrootcone}. The white circles
    are also the monoid $Q_1$ in Example \ref{ex:blowupnotenoughtomakefirm}. 
    }
    \label{fig:firmamentrootcone}
\end{figure}
\end{example}

The firmament $\Gamma_f \subseteq \Sigma_Y$ is ``logarithmically birational.''

\begin{corollary}\label{cor:firmamentblowup}
    Let $f : X \to Y$ be a map of log schemes. Given a commutative diagram 
    \[
    \begin{tikzcd}
        \tilde X \ar[r, "{\tilde f}"] \ar[d]       &\tilde Y \ar[d]       \\
        X \ar[r, "f", swap]       &Y
    \end{tikzcd}
    \]
    with vertical maps log blowups, the firmaments $\Gamma_f = \Gamma_{\tilde f}$ are identified under the isomorphism $\Sigma_{\tilde Y}(\NN) = \Sigma_Y(\NN)$ of Corollary \ref{cor:logblowupsameintegralconecomplex}. 
\end{corollary}

\begin{proof}
    Omitted. 
\end{proof}

\begin{remark}
    We will mostly be interested in the case of maps $p : S \to Y$ from a log scheme $S$ with log structure of rank at most one everywhere.
    Examples include: 
    \begin{itemize}
        \item The spectrum of a discrete valuation ring $S = \spec R$ with log structure $M_S \coloneqq R \setminus \{0\}$. This is the divisorial log structure at the closed point of $S$. 
        \item Log points with rank-one log structure $S = \spec k$ with $\bar M_S = \NN$. 
        \item The spectrum $S = \spec R$ of Dedekind rings $R$ such as $\ZZ$ with its divisorial log structure at a finite set of marked points
        \footnote{
        One can also take $M_S \coloneqq R \setminus \{0\}$ here, which is the limit of all these log structures at finite sets of primes. This is not a f.s.\ or even quasicoherent log structure. Nevertheless, its firmament is well defined. 
        }.
    \end{itemize}
    
    Suppose $p : S \to Y$ is a morphism from a discrete valuation ring $S = \spec R$ with the above log structure $M_S = R \setminus \{0\}$. The inclusion of the closed point $s \in S$ induces an isomorphism of Artin fans $\af{s} = \af{S}$. The 
    firmaments of $p : S \to Y$ and of the composite $s \subseteq S \to Y$ are the same.
    
    More generally, suppose $S$ has a log structure of rank at most one everywhere and write $S_{=1} \subseteq S$ for the reduced closed subscheme on which the log structure is supported. The Artin fan $\af{S_{=1}} = \bigsqcup \af{}$ is a disjoint union of copies of $\af{}$ indexed by $\pi_0(S_{=1})$. The Artin fan $\af{S}$ of $S$ is the quotient of this disjoint union by identifying some of the open points, together with a disjoint summand $\bigsqcup \pt$ of points indexed by connected components of $S$ with trivial log structure. Unless $S_{=1} = \varnothing$ is empty, the maps of cone complexes associated to $p : S \to Y$ and $S_{=1} \subseteq S \to Y$ have the same image and hence the same 
    firmament.

    As $S_{=1}$ supports a locally constant, rank-one characteristic monoid $\bar M_{S_{=1}}$, the map $S_{=1} \to Y$ corresponds to a map to the evaluation stack $\pra{S_{=1}}^\circ \to \wedge Y$ \cite{evaluationspacefamiliesofstandardlogpoints}. By taking topological connected components $\pi_0(-)$ of the map 
    \[\pra{S_{=1}}^\circ \to \wedge Y \to \wedge \af{Y},\]
    we obtain another equivalent definition of the firmament in this case. See Section \ref{ss:stablemaps} for the example of marked points on log prestable curves. 
\end{remark}

Let $X$ be a toroidal embedding and $R$ a discrete valuation ring with a map $\phi : \spec R \to X$ which sends the generic point to the open interior $U_X \subseteq X$. To these data, Abramovich associates a point $n_\phi \in N_\sigma$ in the monoid $N_\sigma$ dual to a local chart of $X$ at the image of the closed point of $\spec R$ \cite[\S 2.4.9]{abramovich09ClayNotes}. By composing with the map $N_\sigma \to \Sigma_X$, we can view $n_\phi$ as an integral point in $\Sigma_X(\NN)$. 

Such a morphism $\phi$ can be uniquely promoted to a morphism of log schemes $\varphi : \spec R \to X$ with $\varphi^\circ = \phi$ by giving $\spec R$ the log structure 
\begin{equation}\label{eqn:DVRlogstr}
    M_{\spec R} \coloneqq R \setminus \{0\}.
\end{equation}
This is because the log structure of $X$ comes from a divisor which pulls back to the closed point of $\spec R$ by our assumption. (Any log map $\spec R \to X$ from a discrete valuation ring with log structure \eqref{eqn:DVRlogstr} must conversely send the generic point into the interior of $X$.)

\begin{proposition}\label{prop:firmamentssameasAbramovich}
    Use notation as above. The integral point $n_\phi \in \Sigma_X(\NN)$ defined by Abramovich \cite[\S 2.4.9]{abramovich09ClayNotes} is the unique generator of the firmament
    $\Gamma_\varphi \subseteq \Sigma_X(\NN)$ as an $\NN$-set. They determine each other uniquely. 
\end{proposition}
\begin{proof}
    Both can be defined \'etale locally in $X$ and $R$. We can assume $X$ is atomic and admits a strict \'etale morphism $X \to V$ to an affine toric variety $V = \spec \ZZ[P]$. The morphism $P \to P' \coloneqq \Gamma(X, \bar M_X)$ is then localization at a face, so $X \to V$ factors through an open immersion $\spec \ZZ[P'] \to V$. This factorization is also strict \'etale \cite[\href{https://stacks.math.columbia.edu/tag/02GW}{Tag 02GW}]{stacks-project}, so we can replace $P$ by $P'$ and assume $\af{X} = \af{V}$ so that $\Sigma_X(\NN) = \Sigma_V(\NN)$. 

    By the equality of cone complexes, it suffices to verify $n_\phi$ generates $\Gamma_\varphi$ for $X = V$. As $\phi$ sends the generic point of $\spec R$ to the dense torus, the map $\phi^\sharp : \ZZ[P] \to R$ is nonzero on elements $p \in P$ of the monoid: $\phi^\sharp(p) \in R \setminus \{0\}$. By definition, $N_\sigma = \Hom(P, \NN)$ and $n_\phi : P \to \NN$ is the morphism sending $p \in P$ to the valuation of $\phi^\sharp(p) \in R \setminus \{0\}$ in the discrete valuation ring $R$. This valuation is the quotient 
    \[R \setminus \{0\} \longrightarrow \dfrac{R \setminus \{0\}}{R^*} \simeq \NN.\]
    So the element $n_\phi \in \Hom(P, \NN)$ is the composite of the top row in the diagram 
    \[
    \begin{tikzcd}
        P \ar[r] \ar[d]       &R \setminus \{0\} \ar[r] \ar[d]      &\NN        \\
        \ZZ[P] \ar[r]      &R.
    \end{tikzcd}
    \]
    This corresponds to the point $\af{} \to \af{P} \in \Sigma_V(\NN)$ given by functoriality of the Artin fan for the map of atomic log schemes $\spec R \to V$, which was to be shown. 
\end{proof}

\subsection{Firmaments vs firmness}\label{ss:abramovichrelation}

Let $f:X\to Y$ be a toroidal morphism of toroidal embeddings, regarded as always as a morphism of log schemes.

\begin{proposition}\label{comparefirmlieinfirmament}
    Let $f : X \to Y$ be a surjective, log flat morphism locally of finite presentation of log schemes and $p : S \to Y$ a log point with rank-one log structure. Then $p$ is $f$-firm if and only if $p$ lies in the firmament of $f$.
\end{proposition}

\begin{proof}
    Because lying in the firmament and firmness are each local on $S$, we can replace $S$ by one of its strict geometric points and $X, Y$ by strict-\'etale covers to assume $Y$ is atomic with $S \to Y$ in its closed stratum and $X$ a disjoint union of atomics. 
    
    Now there are compatible maps of Artin fans between them 
    \[
    \begin{tikzcd}
        \af{X \times_Y^{\rm fs} S} \ar[r] \ar[dr]      &\af{X} \times^{\rm fs}_{\af{Y}} \af{S} \ar[r] \ar[d] \lpb      &\af{X} \ar[d]         \\
                &\af{S} \ar[r]         &\af{Y}.
    \end{tikzcd}
    \]
    Firmness is equivalent to finding a section 
    $\af{S} \dashrightarrow \af{X \times_Y^{\rm fs} S}$ by Proposition \ref{prop:logpointfirm=AFsection}. Such a map results in a lift $\af{S} \dashrightarrow \af{X}$ over $\af{Y}$, which means $p$ lies in the firmament of $f$. 

    Now suppose $p$ lies in the firmament of $f$. As in the proof of Theorem \ref{thm:thibault}, write $v \in \af{Y}$ for the origin and $v' \in \af{X}$ for its pullback. The map $X \to v'$ is open and hence surjective by Lemma \ref{lem:opendeepeststratasurjectiveAF}. 

    We assumed $S \to Y$ lied in the closed stratum, so $S \to Y \to \af{Y}$ factors through $v$. Lying in the firmament means there is a lift of $\af{S} = \af{} \to \af{Y}$ to $\af{X}$, so there is a lift of $S \to Y \to \af{Y}$ to $\af{X}$. This lift factors through the pullback $v'$:
    \[
    \begin{tikzcd}
                &v' \lpbstrict \ar[r] \ar[d]         &\af{X} \ar[d]         \\
        S \ar[r] \ar[ur, dashed]       &v \ar[r]          &\af{Y}.
    \end{tikzcd}
    \]
    As $X \to v'$ is surjective and locally of finite presentation, the geometric point $S \to v'$ lifts to a point of $X^\circ$ \cite[\href{https://stacks.math.columbia.edu/tag/0487}{Tag 0487}]{stacks-project}. As $X \to v'$ is strict, this also gives a log lift. 
\end{proof}

\begin{theorem}\label{thm:firm=firmament}
    Let $f : X \to Y$ be a log flat, log reduced, finite presentation morphism of log schemes and $\spec R \to Y$ a log morphism from a discrete valuation ring $R$ equipped with its valuative log structure $M_R = R \setminus \{0\}$. Let $s \in \Spec R$ be the closed point with induced log structure. 
    
    The following are equivalent: 
    \begin{itemize}
        \item The map $\spec R \to Y$ lies in the firmament $\Gamma_f \subseteq \Sigma_Y$. 
        \item The map $\spec R \to Y$ is $f$-firm. 
        \item The composite $s \to \spec R \to Y$ is $f$-firm. 
    \end{itemize}
   
\end{theorem}

\begin{proof}
Because the firmament of $\spec R$ is the same after localizing, we can use Lemma \ref{lem:coverbyatomics} and Lemma \ref{lem:localAFfunctoriality} to \'etale localize and assume there are compatible morphisms of Artin fans fitting in a commutative diagram of solid arrows
    \[
    \begin{tikzcd}
            &   &X \ar[dl] \ar[d]      \\
    s \ar[r] \ar[d]       &Y \ar[d]      &\af{X} \ar[dl]\\
    \af{} \ar[r] \ar[urr, crossing over, dashed]     &\af{Y}
    \end{tikzcd}
    \]
    and that $Y$ is affine. The last two bullet points are equivalent by Proposition \ref{prop:generizationfirmlocus}. 
    
    The map $\spec R \to Y$ lies in the firmament when there exists a dashed factorization. Equivalently, we need a dashed arrow $s \dashrightarrow \af{X}$. 
    
    Replace $s$ by a strict geometric point. By Lemma \ref{lem:relAFbasechange}, the two pullbacks give the same $\scr B$ in the diagram: 
    \[
    \begin{tikzcd}
        \af{X_s} \ar[d]        &\scr B \ar[r] \ar[d] \ar[l] \lpbstrict \ar[dl, very near start, phantom, "\msout{\ell}\urcorner"]      &\af{X}  \ar[d]        \\
        \af{s}      &s \ar[r] \ar[l] \ar[ur, dashed] \ar[ul, dashed]      &\af{Y}
    \end{tikzcd}
    \]
    so finding the dashed lifts in the two squares are equivalent. But the dashed factorization in the left square is equivalent to firmness of $s \to \spec R \to Y$ along $f$ by Proposition \ref{prop:logpointfirm=AFsection}.
\end{proof}

Theorem \ref{thm:firm=firmament} shows a rank-one log geometric point $s \to Y$ coming from a discrete valuation lies in the firmament of $X \to Y$ if and only if it is firm, i.e., it lifts to a log point of $X$. Up to log alterations of $Y$, we now show that firmness can also be checked on log points.

\begin{corollary}\label{cor:firmiffNpoints}
    Let $f : X \to Y$ be a morphism of finite presentation with $Y$ quasicompact. Suppose each log geometric point $\bar s \to Y$ with either 
    \begin{itemize}
        \item rank-one log structure $\bar M_{\bar s} = \NN$ or 
        \item divisible rank-one log structure $\bar M_{\bar s} = \QQ_{\geq 0}$\footnote{Beware that these log points are not f.s.\ log schemes, but integral and saturated. }
    \end{itemize}
    is $f$-firm. Then there is a log alteration $Y' \to Y$ which is $f$-firm. 
\end{corollary}

\begin{proof}
    If all rank-one log geometric points $s \to Y$ lift to $X$, so do all divisible rank-one log points $t \to Y$. This is because any divisible rank-one log point $t \to Y$ factors through a rank-one log geometric point corresponding to some inclusion $\dfrac{1}{n} \NN \subseteq \QQ_{\geq 0}$. So we have reduced to the divisible case. 
    
    All divisible rank-one log geometric points factor uniquely through any log alteration. So we can replace $f$ by its f.s.\ pullback along a log alteration $Y' \to Y$ to assume $f$ is integral and saturated by Lemma \ref{lem:logaltnint+sat}. Then Theorem \ref{thm:firmlocus} shows the firm locus is the set-theoretic image, but $f$ is surjective by the lifting condition.
\end{proof}

If all rank-one log geometric points $\bar s \to Y$ with $\bar M_{\bar s}$ lift to $X$, we would like to find a log blowup of $Y$ which makes $f$ firm. Unfortunately, this is not the case. 

\begin{example}\label{ex:blowupnotenoughtomakefirm}
    Let $Q_1, Q_2, Q_3 \subseteq \NN^2$ be the submonoids 
    \[
    Q_1 = \{(a, b) \in \NN^2 \, | \, 2\mid a+b\},
    \qquad
    Q_2 = 2 \NN \times \NN, \qquad 
    Q_3 = \NN \times 2 \NN.
    \]
    The monoid $Q_1 \subseteq \NN^2$ is depicted in Figure \ref{fig:firmamentrootcone}. Then $\NN^2 = Q_1 \cup Q_2 \cup Q_3$, as for $(a, b) \in \NN^2$, if either $a$ or $b$ is even, it lies in $Q_2$ or $Q_3$. Otherwise, both $a, b$ are odd, and it lies in $Q_1$. 

    Let $k$ be a geometric point. Define log schemes $\bar x_i, Y$ by 
    \[\bar x_i^\circ = Y^\circ = \Spec k\]
    and 
    \[
    \bar M_{\bar x_i} = \Hom(Q_i, \NN), \qquad
    \bar M_Y = \NN^2.
    \]
    The inclusions $Q_i \subseteq \NN^2$ dualize to give maps $\bar M_Y \to \bar M_{\bar x_i}$ inducing maps of log schemes $\bar x_i \to Y$. Let $X = \bar x_1 \sqcup \bar x_2\sqcup \bar x_3$ be their disjoint union. 

    All rank-one log geometric points $\bar s \to Y$ lift to $X$, as $\NN^2 = Q_1 \cup Q_2 \cup Q_3$. But no log blowup of $Y$ will be $f$-firm. This is because no submonoid $P \subseteq \NN^2$ with the same associated group $\gp P = \gp{\pra{\NN^2}} = \ZZ^2$ lifts to any of the $Q_i$ even though all elements $p \in P$ land in some $Q_i$.
\end{example}

\subsection{Application: lifting rational points}\label{lifting_rat}
In \cite[paragraph following Theorem 2.4.18]{abramovich09ClayNotes}, a property of lifting rational points through toroidal dominant maps is stated without proof. We prove the claim in Theorem \ref{LiftAbr} using Theorem \ref{thm:thibault}. 

We use a log version of Hensel's lemma which is known \cite[Lemma 3.2]{JanDenef}, \cite[Proposition 5.13]{LSS20pseudo-split}. We give a proof using Artin fans.

\begin{lemma}(\textbf{Log Hensel Lemma})\label{LogHensel} Let $f \colon X \to Y$ be a log smooth morphism of log schemes. Let $p \colon s \to S$ be a strict log point such that the residue field of $s$ is algebraic over that of $p(s)$ and contains a separable closure of it, for example $s$ a log geometric point. For any commutative square of log schemes 
    \begin{equation} \label{log_hensel}
    \begin{tikzcd}
        s \arrow[r] \arrow[d] & X \arrow[d] \\
        S \arrow[r,"p"] & Y,
    \end{tikzcd}
    \end{equation}
    there exists a strict \'etale neighborhood $V \to S$ of $s$ and a factorization 
    \begin{equation} \label{log_hensel_2}
    \begin{tikzcd}
        s \arrow[rd] \arrow[r] & V \arrow[r] \arrow[d] & X \arrow[d] \\
        & S \arrow[r] & Y.
    \end{tikzcd}
    \end{equation}
\end{lemma}

\begin{proof}
    Let $s'$ be the separable closure of $p(s)$ in $s$. Replace $X$ by $X \times_Y^{\rm fs} S$ to assume $S = Y$. The claim is strict-\'etale local in $S = Y$, so Lemma \ref{lem:strhenselianlifting} lets us assume the underlying scheme $S^\circ$ is strictly henselian with closed point $s'$. The statement is also strict-\'etale local in $X$, so we can assume $X$ is a disjoint union of atomics using Lemma \ref{lem:coverbyatomics}. 

    The Artin fan is functorial for morphisms between disjoint unions of atomics, so we have a diagram 
    \[
    \begin{tikzcd}
                &         &X \ar[d]         \\
                &         &Y \times_{\af{Y}}^{\rm fs} \af{X} \lpbstrict \ar[r] \ar[d]      &\af{X} \ar[d]             \\
        s \ar[r] \ar[uurr, bend left=25]       &S \ar[r] \ar[ur, dashed]        &Y \ar[r]       &\af{Y}.
    \end{tikzcd}
    \]
    Because the Artin fans $\af{s} = \af{S}$ are the same, we have a dashed arrow in the diagram $S \dashrightarrow Y \times_{\af{Y}}^{\rm fs} \af{X}$. Because $X \to Y$ is log smooth, $X \to Y \times_{\af{Y}}^{\rm fs} \af{X}$ is smooth.
    Pull back $X$ along this dashed arrow to obtain a strict smooth morphism
    \[W \coloneqq X \times_{Y \times_{\af{Y}}^{\rm fs} \af{X}} S \to S.\]
    Since $W \to S$ is a smooth morphism of algebraic spaces, the point $s \to W$ factors through a point $s' \to W$ and we are reduced to the classical Hensel's lemma \cite[\S 10 Exercise 9]{atiyah-macdonald}, which concludes the proof.
\end{proof}

\begin{corollary}\label{corollary:etalelocallifts}
    Let $f \colon X \to Y$ and $p \colon S \to Y$ be morphisms of log schemes and $s \to S$ a strict log geometric point. Suppose $f$ is log smooth. Then,
    \begin{enumerate}
        \item There exists a strict \'etale neighbourhood $V \to S$ of $s$ factoring through $X$ if and only if $s \to Y$ is $f$-firm.
        \item There exists a strict \'etale cover $V \to S$ such that $V \to Y$ factors through $X \to Y$ if and only if $p$ is $f$-firm.
    \end{enumerate}
\end{corollary}

\begin{proof}
    (1) implies (2) since $p$ is $f$-firm if and only if the composite
    \[
    p_s \colon s \to S \to Y
    \]
    is $f$-firm for all strict log geometric points $s$ of $S$. For (1), pick a strict geometric point $p \colon s \to S$ and let $s'$ be the algebraic closure of $p(s)$ in $s$. The corollary applied to $s' \to S$ implies the corollary applied to $s \to S$, so we may assume $s=s'$. By Theorem \ref{thm:thibault}, $p_s$ is $f$-firm if and only if it factors through $f$. By Lemma \ref{LogHensel}, $p_s$ factors through $f$ if and only if there is a strict \'etale neighbourhood $V \to S$ of $s$ in $S$ such that $V \to S \to Y$ factors through $f$.
\end{proof}

\begin{remark}\label{remark:etalelocallifts_for_divisorial_case}
    In the setting of Corollary \ref{corollary:etalelocallifts}, suppose that $X \to Y$ is set-theoretically surjective and that $S$ has the divisorial log structure coming from a divisor $D$. Any strict geometric point of $S \setminus D$ has the trivial log structure, so it is automatically $f$-firm. Hence, $p$ lifts \'etale-locally along $f$ if and only if the strict geometric points of $D$ are $f$-firm.
\end{remark}

Let $S=\spec A$ where $A$ is a Dedekind domain with fraction field $K$ of characteristic 0. For example, $A$ could be the ring of integers of a number field $K$. 
If $\mathscr Y$ is an $A$-scheme endowed with a divisorial log structure given by a reduced divisor $\mathscr D\subseteq \mathscr Y$, and $\mathcal Q: S\to\mathscr Y^\circ$ is a morphism of schemes such that the generic point of $S$ maps to $\mathscr Y^\circ \backslash \mathscr D$, we endow $S$ with the divisorial log structure defined by the divisor $\mathcal Q^{-1}(\mathscr D)$ and we call it the \emph{divisorial log structure on $S$ induced by $\mathcal Q$}.
\begin{theorem}\label{LiftAbr}
   Let $f: X \to Y$ be a proper, dominant map of integral proper $K$-varieties and $S = \Spec A$ a Dedekind domain as above. After replacing $S$ by a sufficiently small nonempty open subset, there exists a proper log smooth $K$-birational model $f':X' \to Y'$ of $f$ and 
   a proper log smooth $S$-model $g:\mathscr X'\to\mathscr Y'$ of $f'$ such that every $K$-point on the locus $X'_0\subseteq X'$ where the log structure is trivial induces an $S$-point $\mathcal Q$ on $\mathscr{Y}'$ that lies in the firmament of $g$, and conversely, every $S$-point $\mathcal Q$ on $\mathscr{Y}'$ 
   that intersects the locus where the log structure of $\mathscr Y'$ is trivial and lies in the firmament of $g$ 
   lifts étale locally on $S$ to a rational point on $X'$.
\end{theorem}

In the statement of Theorem \ref{LiftAbr}, for each $S$-point $\mathcal Q$, the log structure on $S$ is the divisorial log structure induced by $\mathcal Q$.

\begin{proof} 
According to \cite[Theorem 1.1]{WeakToroidalization}, there exists a toroidal model $f': X' \to Y'$ of $f$ such that 
\begin{itemize}
    \item $X', Y'$ are smooth proper toroidal embeddings which admit charts Zariski-locally, 
    \item $f'$ is proper and dominant, 
    \item the toroidal boundary divisors of $X', Y'$ are strict normal crossings. 
\end{itemize}
Endow $X'$ and $Y'$ with the divisorial log structures defined by the toroidal boundaries. As $\mathrm{char}(K)=0$, $f'$ is log smooth by Remark \ref{remark:toroidal_implies_pseudotoroidal}, Lemma \ref{Dominant=finite kernel} and Lemma \ref{lemma:pseudotoroidal+dominant+good characteristic => log smooth}. (Alternatively, the existence of $f'$ can be deduced by \cite[Theorem 3.9]{IllusieTemkin14}.)

After inverting finitely many primes of $A$, \cite[Theorem 3.2.1]{Poonen17} provides an $S$-model $g:\mathscr{X}' \to \mathscr{Y}'$ such that 
\begin{itemize}
    \item $\scr X', \scr Y'$ are regular and proper over $S$, 
    \item the closure of the s.n.c.\ divisors of $X', Y'$ are s.n.c.\ on $\scr X', \scr Y'$, 
    \item $g$ is log smooth with the divisorial log structures on $\scr X', \scr Y'$. 
\end{itemize}

    The direct implication is immediate. For the converse, let $\cal Q : S \to \scr Y'$ be such that the generic point of $S$ lands in the locus where the log structure of $\mathscr Y'$ is trivial.
    Equip $S$ with the divisorial log structure induced by $\cal Q$. This log structure is necessarily divisorial at finitely many points of $S$, as the generic point $\Spec K \in S$ does not map into the log structure of $\scr Y'$.
    
    Assume that $\mathcal{Q}$ lies in the firmament of $g$. Let $\bar s\to S$ be any strict geometric point with residue field that is algebraic over the residue field of its image in $S$. Then there exists a prime $ \mathfrak{p} \in S$ such that $\bar s \to \Spec A$ factors through the strict morphism $\Spec A_{\mathfrak{p}} \to S$ as a log map. 
    The induced point $\bar s \to \mathscr{Y}'$ is $g$-firm. Indeed, if $S$ has trivial log structure at $\mathfrak p$ it follows directly from the definition of firmness (cf.~Example \ref{eg:trivial}), while if $S$ has nontrivial log structure at $\mathfrak p$ it follows from Theorem \ref{thm:firm=firmament} and Lemma \ref{lem:firmbasics}\ref{it:firmpointsatgeompoints}, as $\Spec A_{\mathfrak{p}} \to \mathscr{Y}'$ lies in the firmament of $g$. In particular, $\bar s$ lifts to $\mathscr X'$ by Theorem \ref{thm:thibault}, and by Lemma \ref{LogHensel} there is a strict \'etale neighborhood $V_{\bar s}\to S$ of $\bar s$ and a lift $V_{\bar s}\to \mathscr X'$ of $\bar s\to\mathscr X'$. Since $S$ is Noetherian, finitely many such \'etale neighborhoods give a covering of $S$. Thus $\mathcal Q: S\to\mathscr Y'$ lifts \'etale locally on $S$ to $\mathscr X'$.
\end{proof}

\begin{remark}\label{primestoinvert} 
\phantom{a}
\begin{enumerate}[label=(\roman*), ref=(\roman*)]
    \item \label{item:same firmament}
    By shrinking $S$ further, one can assume that $Y'$ and $\mathscr Y'$ have the same cone complex and that $f'$ and $g$ define the same firmament as in Definition \ref{def:firmamentcontactorder}.

    \item A variant of Theorem \ref{LiftAbr} can be formulated without the assumption of dominance. If $f: X \to Z$ is a proper map of $K$-varieties, and $Y \subseteq Z$ is the image of $f$ with reduced scheme structure, then the same statement as above applies to $f: X \to Y$, which is now dominant.
    
    \item In the previous proof, many steps involve inverting primes in $S$ to construct a log smooth model $\mathscr{X}' \to \mathscr{Y}'$:
    \begin{enumerate}
        \item Choosing the models $\mathscr{X}'$ and $\mathscr{Y}'$ of $X'$ and $Y'$ to be regular and proper.
        \item After taking the Zariski closure of the snc toroidal divisor of $X'$ (resp. $Y'$) in $\mathscr{X}'$ (resp. $\mathscr{Y}'$), we can assume that it is an snc divisor on $\mathscr{X}'$ (resp. $\mathscr{Y}'$) after inverting enough primes in $S$. We endow both $\mathscr{X}'$ (resp. $\mathscr{Y}'$) with the induced divisorial log structure, obtaining a map of log schemes $\mathscr{X}' \to \mathscr{Y}'$.

        \item
        We can Zariski-locally obtain a chart 
        \[
        \begin{tikzcd}
        U \arrow{r}{} \arrow[swap]{d}{} & V \arrow{d}{} \\
        \Spec A[Q] \arrow{r}{} & \Spec A[P]
        \end{tikzcd}
        \]
        for the morphism $g : \scr X' \to \scr Y'$. Finitely many such charts are needed as $\scr X', \scr Y'$ are quasicompact. For each such chart,
         we need to invert primes of $S$ to make $U \to V \times_{\Spec A[P]} \Spec A[Q]$ smooth. The map $g : \scr X' \to \scr Y'$ is then pseudo-toroidal (Definition \ref{pseudo-tor}).
        \item 
       Dominant pseudo-toroidal maps such as $g$ become log smooth after inverting finitely many primes on $S$. The primes we must invert are those dividing the orders of the kernel and of the torsion part of the cokernel of the map $\gp P \to \gp Q$. 
    \end{enumerate}
  
\item By Corollary \ref{corollary:etalelocallifts} a point $\mathcal Q:S\to \mathscr Y'$ lies in the firmament of $g$ if and only if it lifts to $\mathscr X'$ \'etale locally on $S$. 
    Thus if the point does not lie in the firmament, there is at least one prime of $S$ that ramifies in all finite extension $L$ of $K$ such that $X'_{\mathcal Q}(L)\neq\emptyset$.
\end{enumerate}
    
\end{remark}

\begin{corollary}\label{cor:abramovich}
    Let $A$ be a Dedekind domain with fraction field $K$ of characteristic 0. Let $X\to Y$ be a proper dominant morphism of integral proper $K$-varieties. Then up to inverting finitely many primes of $A$, there is a toroidal birational model $f':X'\to Y'$, an open subset $U\subseteq Y$ where $Y'\to Y$ is an isomorphism, and a proper $A$-model $\mathscr Y'$ of $Y'$ such that 
    for every $y\in U(K)$ the following are equivalent:
    \begin{enumerate}[label=(\roman*),ref=(\roman*)]
        \item $y$ lies in the firmament of $f'$;
        \item for every prime $\mathfrak p$ of $A$ there is a finite extension $L$ of $K$ unramified at $\mathfrak p$ such that $X_y(L)\neq\emptyset$;
        \item there are finitely many finite extensions $L_1,\dots,L_r$ of $K$ such that $X_y(L_i)\neq\emptyset$ for all $i\in\{1,\dots,r\}$ and for every prime $\mathfrak p$ of $A$ there is $i\in\{1,\dots,r\}$ such that $L_i/K$ is unramified at $\mathfrak p$.
    \end{enumerate}
\end{corollary}
\begin{proof}
    Let $U\subseteq Y$ be an open subset where the birational morphism $Y'\to Y$ is an isomorphism.
    From Remark \ref{primestoinvert}\ref{item:same firmament} and Corollary \ref{corollary:etalelocallifts} it follows that a point $y\in U(K)$ lies in the firmament of $f'$ if and only if for every prime $\mathfrak p$ of $A$ there is an \'etale neighborhood $V\to\Spec A$ of $\mathfrak p$ such that $X'_y(L)\neq\emptyset$ where $L$ is the function field of $V$. Since $A$ is Noetherian, finitely many such $V$ cover $\spec A$.
\end{proof}

\subsubsection{Examples}\label{ExamplesAbr}
To illustrate Theorem~\ref{LiftAbr}, we compute the étale local lift of a given firm point along a general toric morphism of affine spaces in Example~\ref{Eg.affine space}, and provide concrete instances of this construction in Examples~\ref{ex:23firmlocus} and~\ref{Inverte_cokernel}. Example~\ref{Diag} describes the firm locus in a case of a non-dominant map (a diagonal map).
\begin{example}\label{Eg.affine space}
When the toroidal embeddings are as simple as affine space, one can explicitly construct the étale base change giving the lift for a given firm point. Let $R$ be a discrete valuation ring with uniformizer $\pi$ and fraction field $K$ of characteristic $0$. Consider a dominant toric map\footnote{Note that the schemes in these examples are not proper, unlike the ones we work with in Theorem \ref{LiftAbr}. Properness only ensures that $K$-points extend to $R$-points on the models, which is not required here since we are directly studying the lifting property of $R$-points on the models themselves.}
\begin{align*}
    \bb A^m_R=\Spec R[x_1,\dots,x_m] &\xrightarrow{f} \bb A^n_R=\Spec R[y_1,\dots,y_n]\\
    (x_1,\dots,x_m) & \mapsto (x_1^{a_{1,j}} \cdots x_m^{a_{m,j}})_{1\leq j\leq n}
\end{align*}
with $a_{i,j}\in\NN$ for all $i$ and $j$. Then, the firmament is given by
$$ \Gamma_f:= \{f_{\Sigma}(\bb N^m)
\subset \bb N^n\},$$
where $f_{\Sigma}$ is the induced map on fans. 
Assume that $Q:\Spec R \to \bb A^n_R$ lies on the firmament $\Gamma_f$. 
If $\alpha: R[y_1,\dots,y_n] \to R, y_j \mapsto u_{y_j}\pi^{e_{y_j}}$ is the morphism defining $Q$, with $u_{y_j} \in R^{*}$ for all $j \in \{1,...,n\}$, this implies that $$(e_{y_1},\dots,e_{y_n}) \in \Gamma_f.$$
In particular, there exists $(e_{x_{1}},\dots,e_{x_{m}}) \in \bb N^m$ such that $f_{\Sigma}(e_{x_{1}},\dots,e_{x_{m}})=(e_{y_1},\dots,e_{y_n})$. We have
\begin{align}\label{A}
    \alpha(y_j) = u_{y_j}\pi^{e_{y_j}}=u_{y_j}\pi^{a_{1,j}e_{x_1}+\dots+a_{m,j}e_{x_{m}}}.\
\end{align}
To lift $Q$ to $\bb A^m_R$, it is enough to construct a map $\beta: R[x_1,\dots,x_m] \to R$ such that $\alpha=\beta \circ g$, where $g:R[y_1,\dots,y_n]\to R[x_1,\dots,x_m]$ is the morphism corresponding to $f$, i.e.,
\begin{align}\label{B}
 \alpha(y_j) = \beta (x_{1}^{a_{1,j}}\cdots x_m^{a_{m,j}}) 
    = \beta(x_1)^{a_{1,j}}\cdots\beta (x_m)^{a_{m,j}}
\end{align}
for all $j\in\{1,\dots,n\}$.
Comparing \eqref{A} and \eqref{B}, if we solve the system
\begin{equation}\label{system}
       \left\{
    \begin{array}{ll}
        u_{y_1} &= u_{x_1}^{a_{1,1}}\cdots u_{x_m}^{a_{m,1}} \\
        \vdots\\
        u_{y_n}&= u_{x_1}^{a_{1,n}}\cdots u_{x_m}^{a_{m,n}} ,
    \end{array}
\right.
\end{equation}
for $u_{x_1}, \dots,u_{x_m}$ units in some extension $R'$ of $R$,
and then set $\beta(x_i):=u_{x_i}\pi^{e_{x_i}}$, the point $Q$ lifts to $\bb A^m$ over $\spec R'$. Since $f$ is dominant, $m\geq n$ and the map on coordinate rings of $f$ is injective. In particular, the matrix $$(a_{i,j})_{\substack{1 \leq i \leq m \\ 1 \leq j \leq n}}$$ has full rank $n$, and therefore, the system $\eqref{system}$ is consistent. Since $m \geq n$, it has at least one solution in $\overline{K}$.
If $(u_{x_1},...,u_{x_m})$ is a solution of \ref{system}, we set $R':=R[u_{x_1},\dots,u_{x_m}]$. Therefore $Q$ lifts locally on $\Spec R$ to $\bb A^m$. The condition that $R'$ is \'etale over $R$ depends only on the matrix $(a_{i,j})_{1\leq i\leq m,1\leq j\leq n}$. We work out some concrete examples below.
\end{example}
\begin{remark}
A similar computation can be carried out when $R$ is a principal ideal domain rather than a discrete valuation ring. In this case, the local lift can be chosen to be étale after possibly shrinking $\Spec R$ to a smaller open subset $\Spec R'$, where $R'$ is obtained from $R$ by inverting a suitable finite set of primes depending only on $f$. As a consequence, if $f':X'\to Y'$ is a toric morphism of split toric varieties, one can take $r=1$ in Corollary \ref{cor:abramovich}.
\end{remark}

\begin{example} \label{ex:23firmlocus}
Let $R$ be a discrete valuation ring with uniformizer $\pi$, and where the prime $3$ is invertible. Then 
$$f: X=\Spec R[x,y] \to Y=\Spec R[s,t], \quad
(x,y)\mapsto (x^2y^3,x)
$$ has a global chart $\bb N^2 \to \bb N^2, (a,b) \mapsto (2a+b,3a)$. The kernel of the induced map of groups $\bb Z^2 \to \bb Z^2$ is trivial, and its cokernel is computed through the Normal Smith Form to be $\bb Z/3 \bb Z$. In particular $f$ is log smooth if and only if $3$ is invertible in $R$. Let $Q: \Spec R \to Y$ be  the point $(u_s\pi^{e_s},u_t\pi^{e_t})$, with $u_s,u_t$ units in $R$. Then $Q$ lies in the firmament of $f$ if and only if $(e_s,e_t)=(2a+3b,a)$ for some $(a,b) \in \bb N^2$. To lift $Q$ étale locally to $X$, it is enough to solve 
$$
   \left\{
    \begin{array}{ll}
        u_{s} &= u_{x}^{2}u_{y}^{3} \\
        u_{t}&= u_{x}.
    \end{array}
\right.
$$
One of the possible solutions is $u_x=u_t, u_y= \sqrt[3]{u_su_t^{-2}}$. We then have a lift $ \Spec R[u_x,u_y]=\Spec R[u_y] \to X$ given by the point $(u_x\pi^a,u_y \pi^b)$, and $\Spec R[u_x,u_y] \to \Spec R$ is étale because $3$ is assumed to be invertible. In other words, $Q$ lifts to $X$ after an \'etale extension of $R$.
\end{example}

\begin{example}\label{Inverte_cokernel}
   Let $R$ be a discrete valuation ring with uniformizer $\pi$. Then 
   \[f: X= \Spec R[x,y] \to Y=\Spec R[s,t], (x,y) \mapsto (x^2y,xy^2)\]
   has a global chart $\bb N^2 \to \bb N^2, (a,b) \mapsto (2a+b,a+2b)$. The kernel of the induced map of groups is trivial and its cokernel is $\bb Z/3\bb Z$. Let $Q: \Spec R \to Y, (s,t) \mapsto (u_s\pi^{e_s},u_t\pi^{e_t})$, with $u_s,u_t$ units in $R$. Then $Q$ lies in the firmament of $f$ if and only if $(e_s,e_t)=(2a+b,a+2b)$ for some $(a,b) \in \bb N^2$. To lift $Q$ \'etale locally to $X$, it is enough to solve 
   $$
   \left\{
    \begin{array}{ll}
        u_{s} &= u_{x}^{2}u_{y}^{} \\
        u_{t}&= u_{x}u_y^2.
    \end{array}
\right.
$$
Note that all the solutions require to take a $3$-th root of a unit. 
In particular, if $f$ is log smooth, i.e., if $3$ is invertible in $R$, then $Q$ lifts \'etale locally to $X$ (for example, choosing $(u_x,u_y)=(\sqrt[3]{u_s^2u_t^{-1}},\sqrt[3]{u_s^{-1}u_t^{2}})$). As this example shows, the order of the torsion of the cokernel doesn't have to be among the $a_{i,j}$.
\end{example}

\begin{example}
In Example \ref{Disjoint_union}, given two distinct primes $q_1$ and $q_2$, and a map $\Spec \bb Z \to \bb A^1$ given by $t=q_1^2q_2^3$, the contact order of the restriction to $\Spec \bb Z_{(p)} \to \bb A^1$ is $2$ if $p=q_1$, $3$ if $p=q_2$ and $0$ otherwise. In particular, every such restriction lies in the firmament $2\bb N \cup 3 \bb N$. Note that the point doesn't lift under $\bra{2} \sqcup \bra{3}: \Aff^1 \sqcup \Aff^1 \to \Aff^1$ as a $\bb Z$-point to $\Aff^1 \sqcup \Aff^1$, though it lifts locally.
\end{example}

\begin{example}\label{Diag}
    Consider the diagonal map $\Delta : \bb A^1 \to \bb A^2$. It is toric but not dominant. For example, the point $(\pi,-\pi)$ lies in the firmament, which is the diagonal in $\bb N \times \bb N$, but doesn't lift to $\bb A^1$. 
    
    The map $\Delta : X \coloneqq \Aff^1 \to Y \coloneqq \Aff^2$ factors through the log blowup $B \coloneqq Bl_{\vec 0} \Aff^2$ at the monoidal ideal $(x, y)$. See Figures \ref{fig:blowupA2}, \ref{fig:ex:Diag}. 

    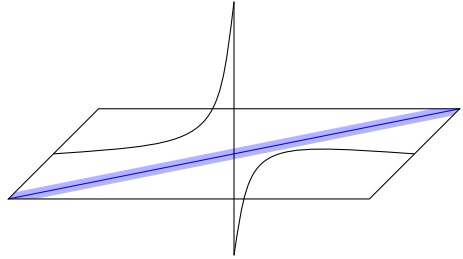
\begin{figure}[h]
    \centering
    \begin{tikzpicture}[scale=1.5]

    \draw[blue!30, line width = 4] (-2, 0) -- (2, 0.8);
    \fill[white] (-2, 0) rectangle (-1.5, -.5); 
    \fill[white] (-2, 0) -- (-1.2,0.8) -- (-1.6, 0.8) -- (-2.4, 0) -- cycle; 
    \fill[white] (2, 0.8) rectangle (1.5, 1.5); 
    \fill[white] (2, 0.8) -- (1.2,0) -- (1.2, 0) -- (2.4, 0.8) -- cycle; 
    \draw[blue] (-2, 0) -- (2, 0.8);
    
    % affine plane (parallelogram shape)
    \draw (-2,0) -- (-1.2,0.8) -- (2,0.8) -- (1.2,0) -- cycle;
    
    % exceptional divisor (vertical line)
    \draw (0,-0.5) -- (0,1.75);
    
    \draw (-1.6, 0.4) .. controls (-.15, .5) .. (0, 1.75);
    \draw (1.6, 0.4) .. controls (.15, .5) .. (0, -0.5);

    \end{tikzpicture}
    \caption{The firm locus $Y(\Delta)$ for the map $X \to Y$ given by the diagonal $\Delta : \Aff^1 \to \Aff^2$ in Example \ref{Diag} is the completion of the strict transform of the diagonal $\Delta$ in the blowup $B = Bl_{\vec 0} \Aff^2$. }
    \label{fig:ex:Diag}
\end{figure}

    By Lemma \ref{lem:firmlocusfactorthroughblowupaltn}, the firm locus $Y(\Delta) \subseteq Y$ factors through $B$ as well. As the factorization $f: \Aff^1 \to B$ is strict, the firm locus of $f$ is the set-theoretic image by Theorem \ref{thm:firmlocus}. The set theoretic image is represented by the formal scheme given by the completion of $B$ at $\Aff^1$, which is $Y(\Delta)$. 
    
\end{example}

\subsection{Stable maps}\label{ss:stablemaps}

Our applications so far have mainly concerned lifting points valued in Dedekind rings like number fields or in discrete valuation rings. Nothing about the definition of firmness or our results requires us to work in the arithmetic or mixed characteristic setting.

We now consider firmness for maps $f : X \to Y$ of log smooth, projective schemes over $\CC$ and $S$ a family of log smooth proper, connected curves over a base log scheme $T$. See the original 
\cite{fkatomodulilogcurves, gross2013logarithmic, abramovich2014stable, chen2014stable} or see \cite{RdC} for an expository account.

\begin{definition}
    A \emph{log prestable curve} (over a base log scheme $T$) is a morphism $\pi : C \to T$ of log schemes which is integral, saturated, and log smooth on which the fibers are geometrically connected of dimension one. A \emph{log prestable map} is a map $f_C : C \to X$ of log schemes with source a log prestable curve $C \to T$. A log prestable map is \emph{stable} if, for each geometric point $\bar t \to T$, the map $f_C$ restricts to a stable map on underlying schemes on the fiber $f_C : C_{\bar t} \to C \to X$. See \cite{fkatomodulilogcurves} for discussion of the genera and number of marked points of such curves, which are locally constant functions on $T$.

 Write $\Mpl(X)$ for the log algebraic stack parameterizing all log prestable maps to $X$ of genus $g$ with $n$ marked points. Let $\Msl(X) \subseteq \Mpl(X)$ be the strict open substack of log stable maps \cite{gross2013logarithmic}. These stacks are highly disconnected, with connected components indexed by ``contact orders'' of the marked points with the log structure of $X$. 
\end{definition}

The na\"ive ``evaluation map'' which restricts a log stable map $f_C \colon C \to X$ to its $n$ marked points does not give a log map 
\[ev : \Mpl(X) \dashrightarrow X^n.\]
Instead, there is an evaluation map to the log evaluation stack \cite{evaluationspacefamiliesofstandardlogpoints}:
\[
ev : \Mpl(X) \to \left(\wedge X\right)^n,
\]
which we now recall.

Consider the quotient $\af{} \to B\GG_m$ of the map $\Aff^1 \to \pt$ by $\GG_m$. Write $\scr Q \in \af{}$ for the stacky origin of Remark \ref{rmk:originclosedpoint}, so that the composite 
\begin{equation}\label{eqn:universalfamilyofstandardlogpoints}
    \scr Q \in \af{} \to B\GG_m
\end{equation}
is an isomorphism on underlying stacks. This composite \eqref{eqn:universalfamilyofstandardlogpoints} is  the universal ``family of standard log points'' as in \cite{evaluationspacefamiliesofstandardlogpoints}.

% Let $\scr Q \in \af{}$ be the quotient of the origin. It is a stacky point $\scr Q^\circ = B\GG_m$ with rank-one log structure. The map $\scr Q \to B\GG_m$ is the universal ``family of standard log points'' as in \cite{evaluationspacefamiliesofstandardlogpoints}. 

\begin{definition}[{The log evaluation stack \cite{evaluationspacefamiliesofstandardlogpoints}}]
    For a log scheme, log algebraic stack, or any (pseudo-)functor on f.s.\ log schemes $Z$, define its log evaluation stack $\wedge Z$ over $B\GG_m$ by
    \begin{equation}\label{eqn:logevalstack}
        \left\{
        \begin{tikzcd}
                    &\wedge Z  \ar[d]      \\
            T \ar[r] \ar[ur, dashed]      &B\GG_m
        \end{tikzcd}
        \right\} \coloneqq Z(T \times_{B\GG_m} \scr Q).
        %\Hom(T \times_{B\GG_m} \scr Q, Z).
    \end{equation}
    The evaluation stack of an Artin cone $\scr B = \af{P}$ splits into components indexed by the set $\Hom(\af{}, \scr B) = \Hom(P, \NN)$ \cite[Example 2.47]{logjets}. 
\end{definition}

Given a map $f : X \to Y$, there is an induced commutative diagram
\[
\begin{tikzcd}
    \Mpl(X) \ar[r] \ar[d]         &\left(\wedge X\right)^n \ar[d]        \\
    \Mpl(Y) \ar[r]         &\left(\wedge Y\right)^n.
\end{tikzcd}
\]

In this way, the log evaluation map defined on log prestable maps is functorial.

\begin{remark}
    Beware that the analogous diagram for log \emph{stable} maps is not commutative: 
    \[
    \begin{tikzcd}
        \Msl(X) \ar[r] \ar[d] \ar[dr, phantom, "\times"]         &\left(\wedge X\right)^n \ar[d]        \\
        \Msl(Y) \ar[r]         &\left(\wedge Y\right)^n.
    \end{tikzcd}
    \]
    This can be seen even when $f : X \to Y$ is a morphism of smooth, projective schemes with trivial log structures, in which case $\Msl(X) = \Ms(X)$ and $\Msl(Y) = \Ms(Y)$. This is because each evaluation map $\Ms(X) \to \wedge X$ factors the $\psi$-class map $\Ms(X) \to B\GG_m$ encoding the cotangent line bundle at the corresponding marked point. These $\psi$-classes are incompatible with stabilization. 

    If $\pi : \Ms[g, n + 1] \to \Ms$ forgets the $(n+1)$st marked point for example, then \cite[Lemma 1.3.1]{kock-psi-notes}
    \[\pi^* \psi_i = \psi_i + D.\]
    Here $D \subseteq \Ms[g, n+1]$ is the divisor whose generic point has a single node joining a genus-$g$ smooth curve and a rational component which contains the points $i$ and $n+1$. The map $\pi$ factors by first forgetting the $(n+1)$st marked point and then stabilizing 
    \[\Ms[g, n+1] \to \Mp \to \Ms,\]
    and the first map identifies universal curves and hence $\psi$-classes. 

    To see this arise in the context of a stabilization morphism, take $X = \PP^1$ and consider degree-one maps $\Ms(\PP^1, 1)$. These all consist of a curve $C = C' \cup \PP^1$ with a rational component mapping isomorphically onto $X$ and the rest $C'$ being contracted. Take $f$ to be the map $X \to \pt$ so that a stable map $C \to X$ is sent to the stabilization of $C$. 
\end{remark}

Let $f : X \to Y$ be a morphism of projective, log smooth log schemes and $Y(f) \subseteq Y$ be its firm locus. The factorization $X \to Y(f) \subseteq Y$ results in a factorization of log evaluation stacks 
\[\wedge X \to \wedge Y(f) \to \wedge Y.\]
We define $\wedge Y(f)$ by the functor of points, though it may not be representable in a nice category. It still constrains the set of connected components of $\wedge Y$ in which the image of $\wedge X$ can lie.

\begin{proposition}
    We have a commutative diagram
    \[
    \begin{tikzcd}
        \Mpl(X) \ar[r] \ar[dd]         &\pra{\wedge X}^n \ar[d]       \\
                &\pra{\wedge Y(f)}^n \ar[d]        \\
        \Mpl(Y) \ar[r]         &\pra{\wedge Y}^n.
    \end{tikzcd}
    \]
    In order for a log prestable map $C \to Y$ to lift to $X$, it is necessary that each of the $n$ marked points of $C$ lands in the firm locus of the map $f$. 
\end{proposition}

The following example shows that the type of a log prestable map doesn't in general determine its firmament.
\begin{example}\cite[Examples]{gross2013logarithmic}\label{gross2013logarithmic}
Let $X=\mathbb{P}^1$ endowed with the divisorial log structure induced by the boundary divisor $D=\{0\}$. Consider two scheme-theoretic maps from smooth source curves $f_i : C_i^\circ \to X^\circ$ for $i = 1, 2$. Suppose $f_1^{-1}(D) = \{x_1, x_2, \cdots, x_n\}$ is a finite, nonempty set of points and $f_2^{-1}(D) = C_2^\circ$ is the whole curve. We will promote $f_1, f_2$ to log prestable map. By “type” below, we mean the type of the object in 
$\underline{GS}(\overline{\mathcal{M}})$ \cite[Definitions 1.7 and 1.10]{gross2013logarithmic} that corresponds to a prestable log map. 

Endow $C_1$ with the divisorial log structure at $\{x_1, \cdots, x_n, y_1, \cdots, y_s\}$, where $y_i \in C_1^\circ$ are points mapping away from $D \subseteq X$. Then $f_1$ is a log prestable map whose type consists of $n + s$ maps $u_{x_i}, u_{y_j} : \NN \to \NN$ of which only the $u_{y_j}$'s are the zero map. The firmament of $f_1$ is the union of the images of the $u_{x_i}$, as taking the dual gives the same map $u_{x_i}$. The type determines the firmament of $f_1$ and not conversely. 

There is no way to promote $f_2$ to a log prestable map without generic log structure coming from a base log scheme, as points with trivial log structure cannot map to points $D \subseteq X$ with nontrivial log structure. Let $Q$ be a sharp, nonzero, f.s.\ monoid and $S = \spec k$ the spectrum of a field with constant log structure $\bar M_S = \underline{Q}$ and $M_S = k^* \times \bar M_S$. 

Suppose $C_2$ has $n$ marked points $x_i$. Any choice of maps 
\[h : \NN \to Q, \qquad u_{x_i} : \NN \to \NN\]
with $h$ nonzero promotes $f_2$ to a log prestable map $C_2 \to X$ over $S$, where $C_2$ has generic log structure $Q$. The firmament is the union of the duals $h^\vee : Q^\vee \to \NN$ and $u_{x_i}^\vee = u_{x_i} : \NN \to \NN$ inside $\NN = \Sigma_X(\NN)$. The type of $f_2$ is merely the set $\{u_{x_i}\}$ of maps and does not include the information of the map $h$, so the type need not determine the firmament. 
 
\end{example}

\appendix

\section{Background}
\label{appendixLog}

We give an idiosyncratic introduction to the circle of ideas surrounding log schemes, Artin fans, and toroidal embeddings designed for brevity and relevance. This should not be the introduction of any reader to log schemes, for which 
\cite{katooriginal, ogusloggeom, loggeometryhandbook, RdC} are better-suited.

\subsection{Basics on Artin fans}
Any log scheme $X$ admits a natural map to a combinatorial object, called the \emph{Artin fan of $X$} and denoted by $\af X$, obtained by gluing local charts for $X$ in an appropriate sense. Here, in \ref{Artin fan of a log alg stack} and in \ref{ss:AFvsconecomplexes}, we discuss Artin fans and show that they are equivalent to the cone complexes used in \cite{abramovich09ClayNotes}.

\begin{definition}
    A \emph{monoid} for us is an integral, saturated, finitely generated, commutative, unital semigroup. A monoid $P$ is sharp if its group of units is zero $P^* = 0$.  
    
    Define the category of \emph{cones}%
    \footnote{This name is justified by the dual cone, traditionally viewed as a submonoid of $\RR^n$.}
    to be the opposite category of sharp monoids 
    \[\Cones \coloneqq \Mons^{\# \rm op}.\]
    Write $\Cone P$, $\bar M_\sigma$ for the cone and monoid corresponding to a monoid $P$ and a cone $\sigma$. We view cones as functors on monoids
    \[(\Cone P) (Q) \coloneqq \Hom(P, Q).\]
\end{definition}

% \begin{remark}
%     The assignment $P \mapsto \Hom(P, \NN)$ yields a self-duality on the category of sharp monoids \cite[Theorem I.2.2.3]{ogusloggeom}: 
%     \[
%     \Mons^\# \longsimeq \Mons^{\# \rm op}; \qquad P \mapsto \Hom(P, \NN).
%     \]
%     Cones can then also be identified with sharp monoids 
%     \[\Cones \longsimeq \Mons^\#; \qquad \Cone P = \Hom(P, \NN).\]
%     The importance of the distinction is that we will think of cones as affine schemes and glue them together along faces, while monoids form a sort of structure sheaf for the resulting monoidal schemes. 
% \end{remark}

\begin{example}\label{ex:affinetoricvars}
    For any integrally closed ring $R$, (normal) affine toric varieties over $R$ are precisely the schemes of the form $X = \Spec R[P]$ for some torsion-free monoid $P$. 

    Let $T \subseteq X$ be the dense torus $\Spec R [\gp P]$. Toric varieties admit sheaves of monoids
    \[M_X \coloneqq \{f \in \OO_X \, | \, f|_T \in \OO_T^*\}.\]
    Under the inclusion $M_X \subseteq \OO_X$, the units are identified $M_X^* = \OO_X^*$. Set $\bar M_X\coloneqq M_X/M_X^*$. All (f.s.) log structures are locally based on these sheaves. 
\end{example}

\begin{definition}
    A \emph{face} of a monoid $P$ is a submonoid $Q\subset P$ such that for all $a,b \in P$, if $a+b\in Q$, then $a, b \in Q$. They are exactly the kernels of maps between sharp f.s.\ monoids. 
    
    A \emph{face localization} is a quotient $P \to P/Q$ under which the face $Q$ is the preimage of $0 \in P$. Face localizations are dual to \emph{face inclusions of cones} under the identification $\Cone P = \Hom(P, \NN)$. 
    
    An \emph{cone space} is a ``space obtained by gluing cones along faces", i.e. the colimit of a diagram whose objects are cones and whose arrows are face inclusions. This colimit is in the category of presheaves on cones. 
\end{definition}

\begin{definition}\label{def:artin}
    An \emph{Artin cone} is the stack quotient of a (normal) affine toric variety over $\bb Z$ by the action of its dense torus. In other words, Artin cones are stacks on schemes of the form
    \[
    \af P := \left[\Spec\bb Z[P]/\Spec\bb Z[\gp P]\right]
    \]
    where $P$ is a cone. Write $$\af{} = \af{\NN} = \bra{\Aff^1/\GG_m}.$$ An \emph{Artin fan} $\scr B$ is a log algebraic stack which admits a strict \'etale cover by Artin cones $\{\af{P_i} \to \scr B\}$. 
    % such that the map $\scr B \to \Log$ is \'etale. 
    % which admits an \'etale cover $\{\sigma_i \to \scr B\}$ by Artin cones \emph{glued along faces}. This means the two projection maps from each fiber product 
    % \[
    % \sigma_{i_1} \times_{\scr B} \sigma_{i_2} \to \sigma_{i_j}
    % \]
    % are face inclusions of cones, including the case where $i_1 = i_2$. 
\end{definition}

For example, the log algebraic stack $\af{} \times B\ZZ/3$ is an Artin fan in our sense. If $\scr B$ is an Artin fan, the map $\scr B \to \Log$ is \'etale and in particular of DM type.

If $\scr B \to \Log$ is an \'etale, representable map from an algebraic stack over a separably closed field, then $\scr B$ is conversely an Artin fan \cite[Lemma 2.3.1]{birationalinvarianceabramovichwise}. The statement is false without representability \cite[Remark 2.3.10]{birationalinvarianceabramovichwise}.

% \begin{example}
%     The log algebraic stack $Y \coloneqq \af{} \times B\mu_2$ admits an \'etale cover by $X \coloneqq \af{}$, but it is not an Artin fan. The fiber product 
%     \[X \times_Y X = \af{} \times_{\pra{\af{} \times B \mu_2}} \af{} = X \times \mu_2\]
%     is not identified with a face via either map $X \times \mu_2 \rightrightarrows X$, so the Artin cone $X$ is not glued along faces to produce $Y$. 
% \end{example}

\begin{remark}
    The sheaf of monoids $\bar M_X$ defined in Example \ref{ex:affinetoricvars} descends to the stack quotient $\bra{X/T}$. It also glues along faces to provide a sheaf $\bar M_{\scr B}$ on any Artin fan $\scr B$. 
\end{remark}

Artin cones are equivalent to cones, which allows us to see cones and cone stacks as stacks on schemes.

\begin{theorem}[{\cite[Theorem 6.11]{tropicalcurvesmodulicones}}]
    The functor
    \[\Cones \longsimeq \cat{Artin Cones}; \qquad \sigma \mapsto \af{\bar M_\sigma}\]
    is an equivalence.
\end{theorem}

\subsection{Log schemes}\label{Artin fan of a log alg stack}

We assume all schemes and algebraic stacks are locally of finite type and all log structures are f.s.\ unless otherwise stated. As a result, a log scheme $X$ is essentially the same as a scheme $X^\circ$ together with a map $X \to \scr B$ to an Artin fan, viewed as an algebraic stack. Different maps $X \to \scr B_1$, $X \to \scr B_2$ can result in the same log structure on $X$, but there is an initial one.

\begin{theorem}\label{thm:deflogstructures}
    Let $X$ be a (locally noetherian) algebraic stack. The following are equivalent: 
    \begin{enumerate}
        \item\label{it:Kato} A sheaf of monoids $M_X$ on the lisse-\'etale site of $X$ with a map $\varepsilon : M_X \to \OO_X$ to the structure sheaf $\OO_X$ (viewed as a multiplicative monoid) which restricts to an isomorphim $\varepsilon : \varepsilon^{-1} \OO^*_X \longsimeq \OO_X^*$ on the units and locally is pulled back from Example \ref{ex:affinetoricvars}. 

        \item\label{it:BV} A sheaf of sharp monoids $\bar M_X$ on the lisse-\'etale site of $X$ together with a monoidal map 
        \[
        \bar M_X \to \af{}|_X
        \]
        to the stack $\af{} = \bra{\Aff^1/\GG_m}$ of pairs $(L, s)$ of line bundles and sections restricted to the lisse-\'etale site of $X$. 

        \item\label{it:Olsson} A morphism $X \to \Log$ to M.\ Olsson's stack of (f.s.) log structures \cite{logstacks}. 

        \item\label{it:AF} (over $\CC$) A morphism $X \to \scr B$ to an Artin fan which is initial among all factorizations through Artin fans
        \[X \to \scr B \to \Log\]
        with $\scr B \to \Log$  representable by algebraic spaces. 
    \end{enumerate}
\end{theorem}

\begin{proof}
    M.\ Olsson showed \eqref{it:Kato} is equivalent to \eqref{it:Olsson} in \cite{logstacks}. The second interpretation \eqref{it:BV} is due to \cite{bornevistoliparabolic}. The existence of an initial factorization $X \to \scr B \to \Log$ can be found in \cite[Proposition 3.2.1]{wisebounded}. 
\end{proof}

\begin{definition}\label{Artinfan}
    Any of the equivalent data in Theorem \ref{thm:deflogstructures} defines a \emph{log structure} on the algebraic stack $X$. Log schemes and log algebraic stacks are schemes and algebraic stacks equipped with log structures. Write $X^\circ$ for the underlying algebraic stack of a log algebraic stack, forgetting the log structure. 
\end{definition}

The initial factorization in Theorem \ref{thm:deflogstructures}\eqref{it:AF} defined over $\CC$ is called \emph{the Artin fan} of $X$ and written $\af{X} \coloneqq \scr B$. We generalize this construction to mixed characteristic in Definition \ref{def:AF} below.

We regard schemes and algebraic stacks as subcategories of the category of log algebraic stacks by endowing them with the initial log structure
\[M_X \coloneqq \OO^*_X.\]
Regarding the underlying stack $X^\circ$ as a log algebraic stack in this way, we have a canonical log morphism 
\[X \to X^\circ\]
forgetting all log structure. 

\begin{definition}\label{def:logstacks}
    By Definition \ref{Artinfan}, the algebraic stack $\Log$ of \cite{logstacks} parameterizes log structures: 
    \[
    \Log(T) \coloneqq \{\text{log structures on }T\}.
    \]
    For any log algebraic stack $Y$, there is an algebraic stack $\Log_Y$ parameterizing log structures together with a log map to $Y$: 
    \[
    \Log_Y(T) \coloneqq \{\text{a log structure on }T \text{ and a log map }T \to Y\}.
    \]
\end{definition}

\begin{definition}[{\cite[Theorem 4.6]{logstacks}}]\label{properties_logschemes}
    A log algebraic stack $X$ is \emph{log smooth}, \emph{log flat}, \emph{log \'etale}, or \emph{log reduced} if the map $X \to \Log$ is smooth, flat, \'etale, or has reduced geometric fibers. 
\end{definition}

These definitions can be made for a morphism $f \colon X \to Y$ of log schemes, using the map to the relative stack of log structures $X \to \Log_Y$. 

Unless specified otherwise, we assume all our schemes and algebraic stacks are locally of finite type. A log scheme or log algebraic stack $X$ (locally of finite type) has a stratification $X = \bigsqcup X_\alpha$ into locally closed subsets $X_\alpha \subseteq X$ given by finite intersections and complements of closed subschemes defined by monoidal ideals $I_X \subseteq M_X$ of the log structure \cite[\S 2.2.2]{logpic}. 

\begin{definition}[{\cite[Definition 2.2.2.2]{logpic}}]\label{def:atomic}
    A log scheme $X$ is \emph{atomic} if it has a unique stratum $X_\alpha \subseteq X$ which is closed, connected and, for all geometric points $\bar z \to X_\alpha$, the restriction map 
    \[\Gamma(X, \bar M_X) \to \bar M_{X, \bar z}\]
    is an isomorphism. 
\end{definition}

We very frequently localize log schemes and log algebraic stacks $X$ to reduce to the case where they are atomic, implicitly using the next lemma. 

\begin{lemma}[{\cite[Proposition 2.2.2.5]{logpic}}] \label{lem:coverbyatomics}
    Any locally Noetherian log scheme or log algebraic stack $X$ admits a strict \'etale cover $\{Y_i \to X\}$ by atomic log schemes $Y_i$. 
\end{lemma}

We define the Artin fan $\af{X}$ of a locally Noetherian log algebraic stack $X$  in mixed characteristic. Let $\{X_i \to X\}$ be a strict smooth cover by atomic log schemes $X$. Choose strict smooth covers $X_{ij\alpha} \to X_i \times_X X_j$ of their intersections by atomic log schemes $X_{ij \alpha}$. Write $Y_0 = \bigsqcup X_i$ and $Y_1 = \bigsqcup X_{ij\alpha}$, so we have maps 
\[
    Y_1 \rightrightarrows Y_0 \to X
\]
that form a truncated strict smooth hypercover. The Artin fans of each $X_i, X_{ij\alpha}$ are defined as the Artin cone of their characteristic monoids and similarly for $Y_k$:
\[
    \af{X_i} = \af{\Gamma(X_i, \bar M_{X_i})}, \quad
    \af{X_{ij\alpha}} = \af{\Gamma(X_{ij\alpha}, \bar M_{X_{ij\alpha}})}, 
\]
\[
    \af{Y_0} = \bigsqcup \af{X_i}, \quad 
    \af{Y_1} = \bigsqcup \af{X_{ij\alpha}}. 
\]

\begin{definition}\label{def:AF}
    Let $X$ be a log algebraic stack locally of finite type and construct $X_i, X_{ij\alpha}, Y_0, Y_1$ as above. The \emph{Artin fan} of $X$ is the coequalizer 
    \[
        \af{X} \coloneqq {\rm coeq}(\af{Y_1} \rightrightarrows \af{Y_0}),
    \]
    computed in the category of \'etale sheaves on the stack $\Log$. 
\end{definition}

We check that this construction of the Artin fan taken from \cite[Proposition 3.2.1]{wisebounded} is independent of the choice of hypercover $X_i, X_{ij\alpha}, Y_0, Y_1$. It need not have the universal property of Theorem \ref{thm:deflogstructures}\eqref{it:AF} in mixed characteristic:

\begin{example}
    Let $X = \spec \ZZ[1/5]$ have trivial log structure. As $X$ is atomic, its Artin fan as constructed above is just $\af{\Gamma(X, \bar M_X)} = \pt = \Spec \ZZ$. As the map $X \to \Log$ is \'etale, it is its own initial \'etale factorization. But $X \neq \Spec \ZZ$. 
\end{example}

% , start with a strict \'etale hypercover by atomics $\{X_i \to X\}$. The Artin fan of each $X_i$ is simply $\af{\Gamma(X, \bar M_{X_i})}$. Then the Artin fan of $X$ is the colimit 
% \[\af{X} = \colim_i \af{X_i}.\]
% This colimit is taken in the category of \'etale sheaves over the stack $\Log$. 

\begin{remark}\label{rmk:repabilityofAF}
The map $\af{X} \to \Log$ is always representable by algebraic spaces, as it is the \'espace \'etal\'e of an \'etale sheaf on $\Log$. General Artin fans $\scr B$ need not be representable over $\Log$. 
\end{remark}

\begin{example}\label{ex:strhenselianAF}
    Let $R$ be a strictly henselian local ring and $X$ a log scheme with underlying scheme $X^\circ = \spec R$ and closed point $x \in X$. Then $\af{X} = \af{x}$ because there are no nontrivial strict-\'etale covers of $X$. 
\end{example}

\begin{example}\label{ex:dvrAF}
    Let $R$ be a discrete valuation ring and equip it with log structure 
    \[M_R \coloneqq R \setminus \{0\}.\]
    Then $\bar M_R = \NN$ and a section $\bar M_R \to M_R$ is a choice of uniformizer. It is atomic, with Artin fan $\af{\spec R} = \af{}$. If $s \in \spec R$ is the closed point, then $\af{s} = \af{}$ also. 
\end{example}

The assignment $X \mapsto \af{X}$ is not functorial in that there are morphisms $f : X \to Y$ for which there cannot exist a commutative square \cite[\S 5.4.1]{skeletonsandfans}: 
\[
\begin{tikzcd}
    X \ar[r] \ar[d] \ar[dr, phantom, "\times"]           &Y \ar[d]      \\
    \af{X} \ar[r]     &\af{Y}.
\end{tikzcd}
\]

For an exact description of when such a square is possible, see \cite{RdC}. We will only need some weaker criteria. 

\begin{lemma}\label{lem:strictAFfunctorial}
    If $f : X \to Y$ is strict, there is a commutative square 
    \[
    \begin{tikzcd}
        X \ar[r] \ar[d]       &\af{X} \ar[d]         \\
        Y \ar[r]       &\af{Y}.
    \end{tikzcd}
    \]
\end{lemma}

\begin{proof}
    Over $\CC$, the universal property \eqref{it:AF} gives functoriality in this case as the morphism $X \to Y$ lies over $\Log$. In mixed characteristic, one can refine a hypercover of $Y$ by a hypercover of $X$ and use those to define the Artin fans. 
    % There is a commutative triangle
    % \[
    % \begin{tikzcd}
    %     X \ar[rr] \ar[dr]       &       &Y \ar[dl]      \\
    %             &\Log
    % \end{tikzcd}
    % \]
    % because $f$ is strict. Apply the universal property of the Artin fan to obtain the map $\af{X} \to \af{Y}$. 
\end{proof}

The Artin fan is ``locally functorial'' in the following sense:

\begin{lemma}\label{lem:localAFfunctoriality}
    If $X, Y$ are atomic and $f : X \to Y$ is a morphism of log schemes, there is a commutative square 
    \[
    \begin{tikzcd}
        X \ar[r] \ar[d]       &Y \ar[d]      \\
        \af{X} \ar[r]      &\af{Y}.
    \end{tikzcd}
    \]
    In particular, we can find such a commutative square after localizing any morphism $f : X \to Y$ by Lemma \ref{lem:coverbyatomics}. 
\end{lemma}

\begin{proof}
    Write $P \coloneqq \Gamma(X, \bar M_X)$, $Q \coloneqq \Gamma(Y, \bar M_Y)$. As $X, Y$ are atomic, their Artin fans are Artin cones
    \[
    \af{X} = \af{P}, \qquad \af{Y} = \af{Q}. 
    \]
    The morphisms $X \to \af{X}, Y \to \af{Y}$ correspond to the identification of $P, Q$ with the global sections of $\bar M_X, \bar M_Y$. The desired commutative square comes from functoriality of global sections and the map of sheaves $\bar M_Y|_X \to \bar M_X$. 
\end{proof}

For functoriality of the Artin fan $f : X \to Y$, it does not suffice that $X, Y$ have Zariski charts:

\begin{example}[{\cite[Example 4.8]{Ulirsch17}}]
    Let $X = C_1 \cup C_2$ be the union of two copies of $C_i = \PP^1$ meeting at two nodes, $p$ and $q$. Write $U_p = X \setminus \{q\}$ and $U_q = X \setminus \{p\}$ and order the two components of the intersection $V_1 \sqcup V_2 = U_p \cap U_q = X \setminus \{p, q\}$. 

    Equip $U_p$ with the log structure associated to
    \[\NN^2 \to \ZZ[x, y]/(xy); \qquad (a, b) \mapsto 0^a \cdot 0^b\]
    and likewise for $U_q$. Our convention is that $0^0 \coloneqq 1$ and $0^a \coloneqq 0$ if $a \neq 0$. 
    
    Descend to $X$ by making \emph{distinct} identifications of 
    \[\bar M_{U_p}|_{V_1} \simeq \bar M_{U_q}|_{V_1} \qquad \text{and} \qquad 
    \bar M_{U_p}|_{V_2} \simeq \bar M_{U_q}|_{V_2}.\]
    Then $\bar M_X$ is not a constant sheaf. 

    The Artin fan $\af{X}$ is $\bra{\af{}^2/\ZZ/2}$. Then we have the same problem with functoriality as in the famous counterexample  \cite[\S 5.4.1]{skeletonsandfans}. We spell this out.

    Consider the subdivision $\scr B \to \af{X}$ given by the $\ZZ/2$-quotient of the subdivision at the ray $(1, 1)$. I.e., 
    \[\scr B = \bra{\bra{Bl_{\vec 0} \Aff^2/\GG_m^2}/\ZZ/2}.\]
    The pullback to $X$ is a log scheme $Z$ which we describe. 
    
    Let $E$ be the fiber over $\vec 0 \in \Aff^2$ of the log blowup $Bl_{\vec 0} \Aff^2$. As a scheme, $E^\circ = \PP^1$. It has generic log structure $\NN$ and two points at which the log structure is $\NN^2$. 

    Let $C'_i = C_i \times E$. Form $Z$ as the pushout of $C'_1 \sqcup C'_2$ given by identifying the two fibers $\{p\} \times E \subseteq C'_i$ and also, separately, the two fibers over $q$. Identify the fibers in two opposite ways, where one is via the identity $E \longequals E$ and the other by the automorphism swapping the coordinates on $\PP^1$. 

    \textbf{Claim:}
    The Artin fan $\af{Z}$ is not $\scr B$, but $\af{}^2$. In fact, the Artin fan of $\scr B$ is $\af{}^2$. 

    This is shown the same way as in \cite[\S 5.4.1]{skeletonsandfans}. The stack $\Log$ has a stacky point $B\GG_m$ parameterizing rank-one log structures, which is the image of the point $[B\GG_m/\ZZ/2] \in \scr B$ corresponding to the ray $\ZZ_{>0} \cdot (1, 1)$. But the $\ZZ/2$ in the stabilizer of the point of $\scr B$ maps to zero in the stabilizer of $\Log$. So this map $\scr B \to \Log$ is not representable. By Remark \ref{rmk:repabilityofAF}, $\scr B$ is not the Artin fan of any log scheme.

    % There's no way for the $\ZZ/2$ stabilizer along the diagonal to live in the Artin fan of a log scheme by Remark \ref{rmk:repabilityofAF} -- the resulting map to $\Log$ is not representable. 

    We need to show we can't have a dashed arrow 
    \[
    \begin{tikzcd}
        Z \ar[r] \ar[d]       &\scr B \ar[r]          &\af{Z} = \af{}^2 \ar[dl, dashed]       \\
        X \ar[r]       &\af{X}
    \end{tikzcd}
    \]
    making the diagram commute. If we restrict to the fiber over the closed point in $\af{X}$, we see the dashed arrow would imply that the $\ZZ/2$-bundle over $Z$ ordering the coordinates of $\PP^1 = E$ is trivial on $Z$, which is false. 
\end{example}

We need one more case of functoriality for the Artin fan. 

\begin{lemma}\label{lem:AFmapstocone}
    Let $P$ be a sharp f.s.\ monoid and $X$ a log algebraic stack. Any morphism $X \to \af{P}$ factors uniquely through $X \to \af{X}$. 
\end{lemma}

\begin{proof}
    Present $P$ as a coequalizer
    \[\NN^s \rightrightarrows \NN^r \to P.\]
    The diagram of Artin fans is then an equalizer
    \[\af{P} \to \af{}^r \rightrightarrows \af{}^s,\]
    so we can reduce to the case $P = \NN$. We need to identify the sections of the characteristic monoids 
    \[\Gamma(X, \bar M_X) \overset{?}{=} \Gamma(\af{X}, \bar M_{\af{X}}).\]

    Find a strict \'etale cover $U = \bigsqcup_\alpha U_\alpha \to X$ by a disjoint union of atomics $U_\alpha$. Do the same for the fiber products $V_{\alpha \beta \gamma} \to U_\alpha \times_X U_\beta$ and write $V \coloneqq \bigsqcup_{\alpha, \beta, \gamma} \af{V_{\alpha \beta \gamma}}$. We get a strict \'etale hypercover of the Artin fan of $X$
    \[
    \bigsqcup_{\alpha \beta \gamma} \af{V_{\alpha \beta \gamma}} \rightrightarrows 
    \bigsqcup_{\alpha} \af{U_\alpha} \to 
    \af{X}.
    \]

    Because $\bar M_X, \bar M_{\af{X}}$ are sheaves, their sections on $X, \af{X}$ are the equalizers of the sequences
    \[
    \Gamma(X, \bar M_X) \to \prod \Gamma(U_\alpha, \bar M_{U_\alpha}) \rightrightarrows \prod \Gamma(V_{\alpha \beta \gamma}, \bar M_{V_{\alpha \beta \gamma}})
    \]
    \[
    \Gamma(\af{X}, \bar M_{\af{X}}) \to \prod \Gamma(\af{U_\alpha}, \bar M_{\af{U_\alpha}}) \rightrightarrows \prod \Gamma(\af{V_{\alpha \beta \gamma}}, \bar M_{\af{V_{\alpha \beta \gamma}}}). 
    \]
    These sequences may be identified, and so can their equalizers. 
\end{proof}

\subsection{Log alterations, modifications, and blowups} 

\begin{definition}
\label{defn:logalteration}
    A map of Artin fans $\pi : \scr B \to \scr C$ is a \emph{log alteration} if it is of DM type, proper, and birational. A map of log algebraic stacks $X \to Y$ is a log alteration if, strict \'etale locally in $Y$, there is an f.s.\ pullback square 
    \[
    \begin{tikzcd}
        X \ar[r] \ar[d] \lpb       &\scr B \ar[d]         \\
        Y \ar[r]       &\scr C.
    \end{tikzcd}
    \]
    One can assume $Y \to \scr C$ is strict, making the square also an ordinary pullback square $\ulcorner \msout{\rm fs}$. 

    A log alteration $f : X \to Y$ of log algebraic stacks is a
    \begin{itemize}
        \item \emph{Log modification} if $f$ is representable. 
        \item \emph{Log blowup} if $f$ is representable and projective. 
        \item \emph{Root stack} if $f$ is $\QQ$-integral
        \footnote{Other authors have required root stacks to be integral, as these form a cofinal system of root stacks.}. 
    \end{itemize}
    \emph{Log subalterations and submodifications} are strict open substacks of log alterations and modifications. 
\end{definition}

See \cite[\S 2.4, 2.6]{birationalinvarianceabramovichwise} for discussion of what we call log modifications and the forthcoming \cite{RdC} for an expository account.

\begin{remark}\label{rmk:logaltnrmk}
    A map $f : X \to Y$ of log algebraic stacks is a log alteration if, strict-\'etale locally in $Y$, there is an f.s.\ pullback square 
    \[
    \begin{tikzcd}
        X \ar[r] \ar[d] \lpb       &\bra{V_1/T} \ar[d]        \\
        Y \ar[r]       &\bra{V_2/T}
    \end{tikzcd}
    \]
    for a proper, birational, equivariant morphism $V_1 \to V_2$ of $T$-toric varieties. Write $\pi : \Sigma_1 \to \Sigma_2$ for the map of fans associated to $V_1 \to V_2$. Then $f$ is a log modification if $\pi$ is a subdivision, a log blowup if $\pi$ is a subdivision which is the bend locus of a piecewise linear function, and a root stack if $\pi$ is the inclusion of a sublattice on each cone. 
    
    In particular, log alterations are log monomorphisms, or monomorphisms in the category of f.s.\ log schemes. This reduces to the case of $\bra{V_1/T} \to \bra{V_2/T}$ as above, which results because subdivisions and sublattices are subfunctors of the ambient fan. 
\end{remark}

\subsection{Toroidal maps v.s.\ log smoothness}\label{Toroidal-maps-log-smoothness}

\begin{definition}\label{defn:toroidalembedding}
    Let $S$ be a scheme. We call a \emph{toroidal embedding}\footnote{When $S$ is a field of characteristic $0$, this definition is equivalent to that in \cite{WeakToroidalization} as proved in \cite[\S 2]{ToroidalMorphisms}.} over $S$ a pair $(f \colon X \to S,U_X)$ where
    \begin{itemize}
        \item $f \colon X \to S$ is a morphism of schemes.
        \item $X \setminus U_X$ is a Weil divisor on $X$.
        \item Any geometric point $x \to X$ has an \'etale neighbourhood $x \to V_x \to X$ such that $V_x$ is \'etale over an affine toric variety $\Aff_P \times S$ over $S$, and $V_x \times_X U_X$ is the preimage of the dense torus $\Aff_{\gp P}$ of $\Aff_P$. We call $V_x \to \Aff_P \times S$ a \emph{toric chart} of $X$ at $x$.
    \end{itemize}
   
    We call $D=X \setminus U_X$ the \emph{boundary divisor} of the toroidal embedding $(f,U_X)$. When there is no ambiguity, we will denote $(f,U_X)$ by $(X,U_X)$, or $(X,D)$ or even $X$. 

    A \emph{toroidal} morphism $f \colon X \to Y$ between toroidal embeddings is a map of $S$-schemes such that, \'etale-locally on $X$ and $Y$, there exist commutative diagrams
    \begin{equation}\label{toroidal-map-loc}
    \begin{tikzcd}
        X \arrow[r] \arrow[d] & Y \arrow[d] \\
        \Aff_Q \times S \arrow[r] & \Aff_P \times S.
    \end{tikzcd}
    \end{equation}
    where the vertical maps are the toric charts and the bottom map is toric. This does not depend on the choice of toric charts since toric charts are unique modulo localization and the torus action. 
\end{definition}

\begin{remark}\label{toroidal_vs_logsmooth}
    The category of toroidal embeddings with toroidal morphisms admits a natural faithful embedding into log schemes, by equipping them with the divisorial log structure given by their toroidal boundary divisor. In particular, a toroidal morphism of toroidal embeddings is a morphism of the corresponding log schemes.
    We consider toroidal embeddings as log schemes without further mention.
\end{remark}

The following definition is designed to encompass both log smooth maps and toroidal maps.
\begin{definition}\label{pseudo-tor}
    Let $f \colon X \to Y$ be a map of log schemes. We say $f$ is \emph{pseudo-toroidal} if, strict \'etale-locally on $X$ and $Y$, there exists a global chart
    \[
    \begin{tikzcd}
        X \arrow[r] \arrow[d] & Y \arrow[d] \\
        \Aff_Q \arrow[r] & \Aff_P
    \end{tikzcd}
    \]
    such that the induced map $X \to Y \times_{\Aff_P} \Aff_Q $ is smooth.
\end{definition}

\begin{remark}\label{remark:toroidal_implies_pseudotoroidal}
    Toroidal morphisms are pseudo-toroidal. Indeed, given a diagram as in \eqref{toroidal-map-loc}, both $X$ and $Y\times_{\Aff_P} \Aff_Q$ are \'etale over $\Aff_Q \times S$ (the first by definition, the second by base change), so the map $X \to Y\times_{\Aff_P} \Aff_Q$ is \'etale as well by  \cite[\href{https://stacks.math.columbia.edu/tag/02GW}{Tag 02GW}]{stacks-project}. In particular, it is smooth.
\end{remark}

\begin{lemma}\label{lemma:toroidal_embeddings_over_K_are_irred}
    Let $K$ be a field. Any toroidal embedding over $\Spec K$ is a disjoint union of irreducible $K$-toroidal embeddings.
\end{lemma}
\begin{proof}
    Let $X$ be a toroidal embedding over $\Spec K$. For all points $x\in X$, pick a geometric point $\o x \to X$ above $x$ and a connected \'etale neighbourhood $j_x \colon U_x \to X$ of $\o x$ such that $U_x$ is \'etale over an affine toric $K$-variety, hence irreducible. As $x$ varies, the images $j_x(U_x)$ form a Zariski open covering of $X$ by irreducible schemes, so each connected component of $X$ is irreducible.
\end{proof}

\begin{lemma}\label{Dominant=finite kernel}
    Let $f \colon X \to Y$ be a toroidal map between connected $K$-toroidal embeddings and $x \to X$ a strict geometric point. Then, $f$ is dominant if and only if the group homomorphism
    \[
    \gp{\o M}_{Y,f(x)} \to \gp{\o M}_{X,x}
    \]
    is injective.
\end{lemma}
\begin{proof}
    Let $y=f(x)$, $P:=\o M_{Y,f(x)}$, $Q:=\o M_{X,x}$ and $N:=\on{Ker}(\gp P \to \gp Q)$. Since $X$ is toroidal, there exist a connected \'etale neighbourhood $V$ of $f(x)$ in $Y$, a connected \'etale neighbourhood $U$ of $x$ in $X\times_Y V$, and a commutative diagram
    \begin{equation}\label{toric_chart}
        \begin{tikzcd}
            U \arrow[r] \arrow[d] & V \arrow[d] \\
            \Aff_Q \times \Spec K \arrow[r] & \Aff_P \times \Spec K
        \end{tikzcd}
    \end{equation}
    as in \eqref{toroidal-map-loc}. All four schemes in \eqref{toric_chart} are irreducible by Lemma \ref{lemma:toroidal_embeddings_over_K_are_irred}, and the vertical arrows are \'etale by definition, so $U \to V$ is dominant if and only if the toric map
    \[
    \Spec K[Q] = \Aff_Q \times \Spec K \to \Aff_P \times \Spec K = \Spec K[P]
    \]
    is. This is in turn equivalent to the map of dense tori
    \[
    g \colon \Spec K[\gp Q] \to \Spec K[\gp P]
    \]
    being dominant. The cokernel of $g$ is $\Spec K[N]$, so $g$ is dominant if and only if $N=\{0\}$.

    It remains to show that $U \to V$ is dominant if and only if $X \to Y$ is. This reduces to showing that $U \to X$ and $V \to Y$ are dominant, which is true since $X$ and $Y$ are irreducible by Lemma \ref{lemma:toroidal_embeddings_over_K_are_irred}.
 
\end{proof}

\begin{lemma} \label{lemma:pseudotoroidal+dominant+good characteristic => log smooth}
     Let $f \colon X \to Y$ be a log scheme map, $x \to X$ a strict geometric point of characteristic $p$ and $y=f(x)$. Then $f$ is log smooth at $x$ if and only if $f$ is pseudo-toroidal at $x$ and the kernel and torsion part of the cokernel of $\gp{\o M}_{Y,y} \to \gp{\o M}_{X,x}$ are finite groups whose orders are not multiples of $p$.
\end{lemma}

\begin{proof}
    Working locally on $X$ and $Y$ in the strict \'etale topology, we may assume $f$ is globally charted by $\gp{\o M}_{Y,y} \to \gp{\o M}_{X,x}$. Then, this is \cite[3.5]{katooriginal}.
\end{proof}

Conversely, varieties which admit log smooth maps to a field $K$ of characteristic $0$ are toroidal over that field, see \cite[Corollary 4.10]{Ulirsch17} or \cite[Remark 3.2]{ToroidalMorphisms}.

\subsection{Artin fans and cone complexes}\label{ss:AFvsconecomplexes}

For the reader's benefit, we compare Artin fans with the cone complexes used by  Abramovich in \cite{abramovich09ClayNotes}. See \cite{RdC} for a more thorough account. 

Recall the (generalized) cone complex of a toroidal embedding or log scheme $X$. For each geometric point $\bar x \to X$, consider the monoid $\bar M_{X, \bar x}$. A specialization $\bar x_1 \rightsquigarrow \bar x_2$ yields a cospecialization map $\bar M_{X, \bar x_2} \to \bar M_{X, \bar x_1}$, which dually gives
\[
\Hom(\bar M_{X, \bar x_1}, \NN) \longrightarrow \Hom(\bar M_{X, \bar x_2}, \NN).
\]
The colimit of these maps is the cone complex of $X$ \cite[page 71]{KTE}, \cite[Appendix B]{gross2013logarithmic}, \cite[\S 2.4.8]{abramovich09ClayNotes}, \cite[Appendix C]{puncturedloggwabramovichchengrosssiebert}:
\[
\Sigma_X \coloneqq \colim_{\bar x \to X} \Hom(\bar M_{X, \bar x}, \NN).
\]

The cone complex $\Sigma$ is a union of real cones $\sigma \subseteq \RR^{\dim \sigma}$ together with integral structures $\sigma_\NN \subseteq \ZZ^{\dim \sigma}$ such that $\sigma_\NN \otimes \RR = \sigma$ \cite[Definitions II.1.5, II.1.6]{KTE}.

\begin{definition}\label{defn:integralpoints}
    We denote by $\Sigma(\NN)$ the union of the integral structures $\sigma_\NN \subseteq \sigma$ of each cone $\sigma \in \Sigma$, and we call its elements the \emph{integral points} of $\Sigma$.
\end{definition}

If $\Sigma = (\sigma, \sigma_\NN)$ for example, $\Sigma(\NN) = \sigma_\NN$.

\begin{lemma}\label{lem:Artinfansconecomplexes}
    The set of maps $\af{} \to \af{X}$ to the Artin fan of $X$ coincides with the set of integral points of the cone complex $\Sigma_X$ of $X$
    \[
    \Hom(\af{}, \af{X}) = \Sigma_X(\NN).
    \]
\end{lemma}

\begin{proof}
    If $X$ is atomic and $\bar x \to X$ lies in its deepest closed stratum, observe that 
    \[
    \Sigma_X(\NN) = \Hom(\bar M_{X, \bar x}, \NN) = \Hom(\af{}, \af{\bar M_{X, \bar x}}) = \Hom(\af{}, \af{X})
    \]
    and conclude. 
    
    Now let $X$ be arbitrary and let $\bigsqcup U_i \to X$ be a strict-\'etale cover by atomic log schemes, resulting in a strict-\'etale cover $\af{Q_i} \to \af{X}$ of Artin fans, with $Q_i = \Gamma(U_i, \bar M_{U_i})$. Arguing as in Lemma \ref{lem:artinconesfactorthroughartincones}, we recognize the hom-set as the colimit
    \[\Hom(\af{}, \af{X}) = \colim_i \Hom(\af{}, \af{Q_i}).\]
    The same is true for $\Sigma_X(\NN)$. Reduce to the above atomic case $X = U_i$ by checking that such an identification is functorial. 
\end{proof}

\begin{definition}
    Define the (integral points of the) cone complex of an arbitrary log scheme $X$ admitting an Artin fan by 
    \[\Sigma_X(\NN) \coloneqq \Hom(\af{}, \af{X}).\]
\end{definition}

The set $\Sigma_X(\NN)$ does not have a monoid structure, though one can scale its elements to view $\Sigma_X(\NN)$ as an $\NN$-set \cite[\S I.1.2]{ogusloggeom}. Unlike the Artin fan, this cone complex is functorial for all maps of log schemes. 

\begin{lemma}\label{lem:conecomplexfunctorial}
    A morphism $f : X \to Y$ of log schemes yields an associated map
    \[\Sigma_f : \Sigma_X(\NN) \to \Sigma_Y(\NN)\]
    on integral points of cone complexes. 
    The assignment $X \mapsto \Sigma_X(\NN)$ yields a functor 
    \[
    ({\rm f.s.\ log\ schemes}) \longrightarrow (\NN-{\rm sets}).
    \]
\end{lemma}

\begin{proof}
    Choose strict-\'etale covers $\{U_\alpha \to X\}$, $\{V_\beta \to Y\}$ by affine, atomic log schemes $U_\alpha, V_\beta$ that form commutative squares 
    \[
    \begin{tikzcd}
        U_\alpha \ar[r] \ar[d]    &X \ar[d]     \\
        V_\beta  \ar[r]     &Y
    \end{tikzcd}
    \]
    for each $\alpha$ and appropriately chosen $\beta$. There result covers of Artin fans 
    \[\af{U_\alpha} \to \af{X}, \qquad \af{V_\beta} \to \af{Y}.\]
    By Lemma \ref{lem:artinconesfactorthroughartincones}, each map $r : \af{} \to \af{X}$ factors through some $\af{U_\alpha}$. We can therefore assume $X, Y$ are affine and atomic provided we check the assignment $\Sigma_f(r)$ is independent of the choice of lift to $\af{U_\alpha}$, which is left to the reader. 

    If $X, Y$ are atomic, there is a commutative square
    \[
    \begin{tikzcd}
        X \ar[r] \ar[d]       &\af{X} \ar[d]         \\
        Y \ar[r]       &\af{Y}.
    \end{tikzcd}    
    \]
    Define $\Sigma_f(r)$ by composition $\af{} \overset{r}{\to} \af{X} \to \af{Y}$ and conclude. 
\end{proof}

\begin{proposition}\label{prop:functoriallyidentifyingconecomplexes}
    Consider the fully faithful embedding $$({\rm Toroidal\ embeddings}) \subseteq ({\rm f.s.\ log\ schemes})$$ of toroidal embeddings as f.s.\ log schemes. The equality of Lemma \ref{lem:Artinfansconecomplexes} makes the square of functors 
    \[
    \begin{tikzcd}
        ({\rm Toroidal\ embeddings}) \ar[r] \ar[d, "X \mapsto \Sigma_X", swap]        &({\rm f.s.\ log\ schemes}) \ar[d, "X \mapsto \Sigma_X(\NN)"]      \\
        ({\rm Cone\ complexes}) \ar[r, swap, "\Sigma \mapsto \Sigma(\NN)"]     &(\NN-{\rm sets})
    \end{tikzcd}
    \]
    2-commutative. 
\end{proposition}

\begin{proof}
    All that remains to check is that the functoriality defined in Lemma \ref{lem:conecomplexfunctorial} coincides with the usual functoriality of cone complexes for toroidal embeddings on integral points. Let $f : X \to Y$ be a morphism of toroidal embeddings and choose strict \'etale covers $U \to X$, $V \to Y$ by disjoint unions of atomics $U, V$ that fit into a commutative square 
    \[
    \begin{tikzcd}
        U \ar[r] \ar[d]       &V \ar[d]      \\
        X \ar[r]      &Y
    \end{tikzcd}    
    \]
    with resulting square of integral points of cone complexes
    \begin{equation}\label{eqn:functorialityofconecomplexes}
    \begin{tikzcd}
        \Sigma_U(\NN) \ar[r] \ar[d]        &\Sigma_V(\NN)  \ar[d]   \\
        \Sigma_X(\NN) \ar[r]        &\Sigma_Y(\NN). 
    \end{tikzcd}
    \end{equation}
    In fact, we have two such squares with potentially distinct arrows which we need to equate. 
    
    By Lemma \ref{lem:artinconesfactorthroughartincones} with $P = \NN$,  $\Sigma_U(\NN) \to \Sigma_X(\NN)$ is surjective. So it suffices to check the functoriality maps are the same for all the maps in \eqref{eqn:functorialityofconecomplexes} besides the map $\Sigma_X(\NN) \to \Sigma_Y(\NN)$. But then this is by construction. 
\end{proof}

\begin{corollary}\label{cor:logblowupsameintegralconecomplex}
    If $b : \tilde X \to X$ is a log blowup, their cone complexes have the same set of integral points
    \[\Sigma_b : \Sigma_{\tilde X}(\NN) \longsimeq \Sigma_X(\NN).\]
\end{corollary}

\begin{proof}
    The definition of the map $\Sigma_b$ is local in $X$, so we can assume $X$ is atomic and $\tilde X$ is pulled back from a log blowup $\scr B \to \af{X}$ of the Artin fan of $X$. Then maps from $\af{}$ factor uniquely through any log blowup
    \[\Hom(\af{}, \scr B) = \Hom(\af{}, \af{X}),\]
    so we are done. 
\end{proof}

\section{Points under dominant morphisms}
\label{sec:appendix}

The results of this section provide a generalization of \cite[Proposition 4.3]{Campana05surfaces} and \cite[Proposition 2.1.8]{abramovich09ClayNotes} by relaxing all regularity assumptions. As in \cite{Campana05surfaces, abramovich09ClayNotes} we first introduce some invariants that capture properties of the fibers of the morphism, such as the orbifold base, and then we show that images of rational points under the morphism are Campana points with respect to these invariants.

\subsection{The orbifold base invariants}
\label{sec:orbifold base invariants}
In \cite{Campana05surfaces} Campana introduces the concept of orbifold base of a dominant morphism to a smooth curve as a divisor on the curve whose coefficients are determined by the morphism. In \cite[\S 2.2]{abramovich09ClayNotes}, Abramovich extends the concept of orbifold base to higher dimensional base varieties in terms of $b$-divisors. Here, we introduce an invariant $m_{f,\mathscr D}$ for every integral closed subscheme $\mathscr D$ of the base, without passing to birational models.

Let $R$ be a Noetherian ring.
Let $f:\mathscr X\to \mathscr Y$ be a dominant morphism of integral proper $R$-schemes with reduced generic fiber. Let $\mathscr D\subseteq \mathscr Y$ be an integral closed subscheme.
Let $\mathscr U \subseteq \mathscr X$ be an affine open subset $\scr U = \spec A$ such that $\mathscr U \cap f^{-1}(\mathscr D) \neq \emptyset$.

Let $I\subseteq A$ be the ideal that defines $\mathscr D\times_{\mathscr Y} \mathscr U$. 
 Let $\mathfrak p_1,\dots,\mathfrak p_r$ be the isolated primes of $I$ in $A$. Let
$$
m_{f, \mathscr D,\mathscr U}=\min_{1\leq i\leq r}\max\{e:I\subseteq \mathfrak p_i^{e}\}.
$$

Then $m_{f, \mathscr D,\mathscr U}$ is the largest integer such that $I\subseteq \bigcap_{i=1}^r\mathfrak p_i^{m_{f, \mathscr D,\mathscr U}}$. Let $m_{f, \mathscr D}$ be the smallest $m_{f, \mathscr  D,\mathscr U}$, where $\mathscr U$ runs over all affine open subsets of $\mathscr X$ intersecting $f^{-1}(\mathscr D)$. The following example shows that there are in principle no relations between the orbifold base invariants $m_{f,\mathscr D}$ and $m_{f,\mathscr D'}$ for $\mathscr D'\subseteq\mathscr D$ integral closed subschemes of $\mathscr Y$.
\begin{example}\label{eg:inclusions}
    Consider $f:\mathbb{A}^3\to\mathbb{A}^3$ given by $$(x,y,z)\mapsto(x^2,y^2,yz).$$ Denote by $(u,v,w)$ the coordinates on the target affine space. If $\mathscr D=V(u-v)$ and $\mathscr D'=V(u,v)$, then $m_{f,\mathscr D}=1<m_{f,\mathscr D'}=2$. On the other hand, if $\mathscr D=V(v)$ and $\mathscr D'=V(v,w)$, then $m_{f,\mathscr D}=2>m_{f,\mathscr D'}=1$.
\end{example}

\subsection{Campana points}
The results \cite[Proposition 4.3]{Campana05surfaces} and \cite[Proposition 2.1.8]{abramovich09ClayNotes} state that the images of rational points under a dominant morphism to a smooth curve defined over a number field $k$ is contained in the set of $\OO_{k,M}$-Campana points determined by the orbifold base invariants for a suitable finite set of places $M$.
We recall that notions of Campana points have been defined on curves by \cite[\S3.4, \S4.1]{Campana05surfaces} and \cite[\S\S13.5-13.6]{Campana11} as \emph{orbifold rational points}, and on higher dimensional varieties by \cite[\S2.2]{abramovich09ClayNotes} as \emph{soft $S$-integral points} and by \cite{AVA18, PSTVA21, MNS24} as \emph{Campana points}. Here, we use the definition of intersection multiplicity as in \cite{MNS24} to define the local Campana condition with respect to the orbifold base invariants in the setting of Section \ref{sec:orbifold base invariants}.

Assume that $R$ is a discrete valuation ring. For each closed subscheme $\mathscr D\subseteq \mathscr Y$, and each $R$-point $y\in \mathscr Y(R)$ such that $y\notin \mathscr D(R)$, let $\spec R\times_\mathscr Y \mathscr D$ be the fiber product of  $y:\spec R\to \mathscr Y$ with the inclusion $\mathscr D\subseteq \mathscr Y$. Consider $n_{y,\mathscr D} \in \NN$ such that $\spec R\times_\mathscr{Y} \mathscr{D}=\spec R/(\pi)^{n_{y,\mathscr D}}$, where $\pi$ is a uniformizer of $R$. For $m\in\ZZ_{\geq 1}$, let $(\mathscr Y, \mathscr D, m)(R)$ be the set of points $y\in \mathscr Y(R)$ such that either $y\in \mathscr D(R)$ or $n_{y,\mathscr D}=0$ or $n_{y,\mathscr D}\geq m$.
\begin{remark}\label{rmk:inclusions}
    If $\mathscr D'\subseteq\mathscr D$ and $m'\leq m$, there is an inclusion $(\mathscr Y, \mathscr D, m)(R)\subseteq (\mathscr Y, \mathscr D', m')(R)$.
\end{remark}

\subsection{Images of points}

\begin{theorem}\label{thm:DVR}
Let $R$ be a discrete valuation ring and $f:\mathscr X\to \mathscr Y$ be a dominant morphism of integral proper $R$-schemes with reduced generic fiber.  
Let $\mathscr D\subseteq\mathscr Y$ be an integral closed subscheme.
Then $f(\mathscr X(R))\subseteq(\mathscr Y, \mathscr D, m_{f,\mathscr D})(R)$.

\end{theorem}
\begin{proof}
Let $x\in \mathscr X(R)$. Let $f(x)\in\mathscr Y(R)$ the point induced by $x$ by composition with $f$. 
If $f(x)\in \mathscr D(R)$ there is nothing to do. If $f(x)\notin \mathscr D(R)$, consider the commutative diagram of fiber products
$$
\begin{tikzcd}
\spec R\times_{\mathscr Y}\mathscr D \arrow[d]\arrow[r]  & \mathscr X\times_{\mathscr Y}\mathscr D \arrow[d]\arrow[r] & \mathscr D \arrow[hookrightarrow,d] \\
\spec R \arrow[r,"x"] & \mathscr X \arrow[r,"f"] & \mathscr Y.
\end{tikzcd}
$$
Commutativity of the diagram gives $n_{f(x),\mathscr D}=n_{x,\mathscr X\times_{\mathscr Y}\mathscr D}$.
Let $\mathscr U=\spec A$ be an affine open subset of $\mathscr X$ that contains $x(\spec R/(\pi))$.
Let $\psi:A\to R$ the ring homomorphism induced by $x$. 
Let $I\subseteq A$ be the ideal that defines $\mathscr D\times_{\mathscr Y} \mathscr U$. Note that 
$
\spec R\times_{\mathscr Y}\mathscr D \cong  \spec (R/\psi(I)).
$
Thus $\psi(I)=(\pi)^{n_{f(x),\mathscr D}}$. 

If $ f^{-1}(\mathscr D)\cap \mathscr U=\emptyset$, then $n_{f(x),\mathscr D}=0$. Assume that $ f^{-1}(\mathscr D)\cap \mathscr U\neq\emptyset$. 
Let $\mathfrak p_1,\dots,\mathfrak p_r$ be the isolated primes of $I$ in $A$. 
For each $i\in\{1,\dots,r\}$,
there are inclusions
$$\psi(I)\subseteq \psi(\mathfrak p_i^{m_{f, \mathscr D,\mathscr U}})
\subseteq (\pi)^{m_{f, \mathscr D,\mathscr U}}\subseteq (\pi)^{m_{f, \mathscr D}},$$
where  the third inclusion holds by definition of $m_{f, \mathscr D}$.
\end{proof}

Let $k$ be the fraction field of a DVR $R$, and let $m_{f,\mathscr D_k}$ be the multiplicity defined by the base change of $f$ and $\mathscr D$ to $k$. Then $m_{f,\mathscr D}\leq m_{f,\mathscr D_k}$, and hence $(\mathscr Y, \mathscr D, m_{f,\mathscr D_k})(R)\subseteq (\mathscr Y, \mathscr D, m_{f,\mathscr D})(R)$.
 We show that if $R$ is a Dedekind domain, this inclusion is strict at only a finite number of maximal ideals of $R$.

\begin{theorem}\label{thm:Dedekind}
Let $R$ be a Dedekind domain with fraction field $k$.
Let $f:X\to Y$ be a dominant morphism of integral proper varieties over $k$ with reduced generic fiber. 
Let $\mathscr Y$ be a proper $R$-model of $Y$. Let $D\subseteq Y$ be an integral closed subscheme and $\mathscr D\subseteq\mathscr Y$ an integral closed subscheme that is an $R$-model of $D$.
Then $f(X(k))\subseteq(\mathscr Y_{\mathfrak m}, \mathscr D_{\mathfrak m}, m_{f,D})(R_{\mathfrak m})$, for all but finitely many maximal ideals $\mathfrak m$ of $R$. 
\end{theorem}
\begin{proof}
Let $\mathscr X$ be a proper $R$-model of $X$. 
Then $f$ induces a birational map $\mathscr X\dashrightarrow \mathscr Y$. 
By \cite[Theorem 3.2.1]{Poonen17} there is $a\in R$ such that the birational map induced by $f$ on $\mathscr X_{R_a}$ is a morphism $f_{R_a}:\mathscr X_{R_a}\to\mathscr Y_{R_a}$, where $R_a=R[\frac 1a]$. 
Let $\mathcal I$ be the ideal sheaf of $\mathscr D\times_{\mathscr Y}\mathscr X_{R_a}$ as a closed subscheme of $\mathscr X_{R_a}$, and let $\mathcal J_1,\dots,\mathcal J_s$ be the ideal sheaves of the irreducible components of $\mathscr D\times_{\mathscr Y}\mathscr X_{R_a}$ endowed with reduced scheme structure. Let $\mathcal I'$ be the ideal sheaf of $D\times_{Y}X$ as a closed subscheme of $X$, and let $\mathcal J'_1,\dots,\mathcal J'_r$ be the ideal sheaves of the irreducible components of $D\times_{Y}X$ endowed with reduced scheme structure. 
Note that $r\leq s$.
By \cite[Proposition 3.11]{atiyah-macdonald} we can assume without loss of generality that  $\mathcal I'=\mathcal I\otimes_{R_a}k$ and $\mathcal J'_i=\mathcal J_i\otimes_{R_a}k$ for all $i\in\{1,\dots,r\}$. By definition of $m_{f,D}$ we have $\mathcal I'\subseteq{\mathcal J_i'}^{m_{f,D}}$ for all $i\in\{1,\dots,r\}$, where ${\mathcal J_i'}^{m_{f,D}}$ is the sheaf associated to the $m_{f,D}$-th power of $\mathcal J_i'$.
For $i\in\{1,\dots,s\}$, let $\mathscr W_i\subseteq\mathscr X_{R_a}$ be the closed subscheme defined by $\mathcal J_i^{m_{f,D}}$. By \cite[Theorem 3.2.1]{Poonen17} there is $b\in R$ such that the rational map $\mathscr W_{i,R_{ab}}\to\mathscr D\times_{\mathscr Y}\mathscr X_{R_{ab}}$ induced by $\mathcal I'\subseteq{\mathcal J'_i}^{m_{f,D}}$ is a closed immersion for all $i\in\{1,\dots,r\}$ and $\mathscr W_{i,R_{ab}}=\emptyset$ for all $i\in\{r+1,\dots,s\}$.
Since $R$ is a Dedekind domain, the number of maximal ideals of $R$ that contain $ab$ is finite. Let $\mathfrak m$ be a maximal ideal of $R$ that does not contain $ab$, and let $f_{\mathfrak m}:\mathscr X_{\mathfrak m}\to\mathscr Y_{\mathfrak m}$ be the morphism induced by $f$. 
Then $\mathcal I_{\mathfrak m}\subseteq{\mathcal J_{i,\mathfrak m}}^{m_{f,D}}$, and hence, $m_{f_{\mathfrak m},\mathscr D_{\mathfrak m}}\geq m_{f,D}$. Thus $m_{f_{\mathfrak m},\mathscr D_{\mathfrak m}}=m_{f,D}$.
Since $\mathscr X_{\mathfrak m}$ is proper over $R_{\mathfrak m}$, we have $X(k)=\mathscr X_{\mathfrak m}(R_{\mathfrak m})$. Hence, $f(X(k))\subseteq(\mathscr Y_{\mathfrak m}, \mathscr D_{\mathfrak m}, m_{f_{\mathfrak m},\mathscr D_{\mathfrak m}})(R_{\mathfrak m})$ by Theorem \ref{thm:DVR}. 
\end{proof}

Since the set of points of $X$ where the fibers of $f$ are reduced is open \cite[Corollaire 12.1.7]{MR0217086}, there is a proper closed subset $\mathscr W$ of $\mathscr Y$ that contains all integral closed subschemes $\mathscr D$ of $\mathscr Y$ such that $m_{f, \mathscr D}>1$, and among those, there are only finitely many of codimension $1$. These, or the irreducible components of $\mathscr W$, give a first approximation of the set of images of $k$-rational points under $f$ via Theorems \ref{thm:DVR} and \ref{thm:Dedekind}.
By Example \ref{eg:inclusions} it is unclear whether there are finitely many integral closed subschemes $\mathscr D_1,\dots,\mathscr D_n$ of $\mathscr Y$ such that $\bigcap_\mathscr D(\mathscr Y,\mathscr D,m_{f,\mathscr D})(R)=\bigcap_{i=1}^n(\mathscr Y,\mathscr D_i,m_{f,\mathscr D_i})(R)$. By Remark \ref{rmk:inclusions} this question is equivalent to asking whether each $\mathscr D$ contains a proper closed subscheme $\mathscr W_{\mathscr D}\subsetneq\mathscr D$ that contains every integral closed subscheme $\mathscr D'\subseteq\mathscr D$ such that $m_{f,\mathscr D'}>m_{f,\mathscr D}$.

\subsection{Alternative definitions}

In the notation of Section \ref{sec:orbifold base invariants},
let $I=\bigcap_{i=1}^s \mathfrak q_i$ be a minimal primary decomposition of the ideal $I$ in the ring $A$. For $i\in\{1,\dots,s\}$, let $\mathfrak p_i$ be the radical of $\mathfrak q_i$. Without loss of generality we can assume that there is $r\leq s$ such that $\{\mathfrak p_i:1\leq i\leq r\}$ is the set of isolated primes of $I$. 
There are several alternative possibilities to define the orbifold base invariants. For example,
\begin{enumerate}[label=(\alph*),ref=(\alph*)]
    \item $m_a=\min_{\mathscr U}\min_{1\leq i\leq r}\max\{e:\mathfrak q_i\subseteq \mathfrak p_i^{e}\}$,
    \item $m_b=\min_{\mathscr U}\min_{1\leq i\leq r}\max\{e:\mathfrak q_i\subseteq \mathfrak p_i^{(e)}\}$, where $\mathfrak p_i^{(e)}=\mathfrak p_i^eA_{\mathfrak p_i}\cap A$ is the $e$-th symbolic power of $\mathfrak p_i$. 
    
    \item $m_c=\min_{\mathscr U}\max\{e: I\subseteq {\sqrt I}^e\}$,
    \item $m_d=\min_{\mathscr U}\max\{e:{\sqrt I}^e\subseteq I\}$.
\end{enumerate}
We observe that if $\mathscr X$ is regular and $\mathscr D$ has codimension $1$, then $m_a=m_b=m_c=m_d=m_{f,\mathscr D}$.
In general, there are inequalities 
\begin{enumerate}
    \item $m_{f,\mathscr D}\geq m_a$, as $I\subseteq\mathfrak q_i$ for all $i\in\{1,\dots,r\}$;
    \item $m_{f,\mathscr D}\geq m_c$, as $(\bigcap_{i=1}^r\mathfrak p_i)^e\subseteq\bigcap_{i=1}^r\mathfrak p_i^e$ for all $e\geq 0$;
    \item $m_b\geq m_a$, as $\mathfrak p_i^e\subseteq\mathfrak p_i^{(e)}$ for all $i\in\{1,\dots,r\}$ and all $e\geq 0$;
    \item $m_d\geq m_c$, as $\sqrt I^{m_d}\subseteq\sqrt I^{m_c}$.
\end{enumerate}

As a consequence of the first two inequalities, the statement of Theorem \ref{thm:DVR} holds also replacing $m_{f,\mathscr D}$ by $m_a$, or by $m_c$. The same cannot be said about $m_b$ and $m_d$. 
The following example shows that Theorem \ref{thm:Dedekind} doesn't hold replacing $m_{f,\mathscr D}$ with $m_b$. This can be explained by the fact that using symbolic powers
means restricting attention around the generic points of the irreducible components of $f^{-1}\mathscr D$, which can lead to miss conditions around other points of $\mathscr D$. 
\begin{example}
    Let $\mathscr X=V(a(x+y)^2-bz^2)\subseteq \mathbb P^2_{\ZZ}\times\mathbb P^1_{\ZZ}$ with coordinates $((x:y:z),(a:b))$. Let $\mathscr Y=\mathbb P^1_{\ZZ}$, and let $f:\mathscr X\to \mathscr Y$ be the projection onto $\mathbb P^1_{\ZZ}$. 
    This example has a section given by $x=-y, \quad z=0$. Thus $m_{f,\mathscr D}=1$ for all $\mathscr D\in \mathbb P^1(\ZZ)$. Let $\mathscr D=(0:1)\in\mathbb P^1(\ZZ)$. In the affine open given by $b=1,x=1$, the fiber over $(0:1)$ is given by the ideal $(a,z^2)=(a, a(1+y)^2-z^2)$ which is primary with radical $(a,z)=(a,z,a(y+1)^2-z^2)$. Since $y+1\in \mathbb Z[a,y,z]/(a(y+1)^2-z^2)\smallsetminus (a,z)$, then $(a,z^2)\subseteq (a,z)^{(2)}$, but $(a,z^2)\not\subseteq (a,z)^2$. 
   
\end{example}

\bibliographystyle{alpha}
\bibliography{zbib}

\end{document}